\documentclass[reqno]{amsart}

\usepackage{graphicx}
\usepackage{amsthm}
\usepackage{amssymb}
\usepackage{caption}
\usepackage[toc,page]{appendix}
\usepackage{subcaption}
\usepackage{float}
\usepackage{bm}
\usepackage{mathtools}
\usepackage{hyperref}
\usepackage[square,sort&compress,numbers]{natbib}
\usepackage{fancyhdr}
\usepackage[foot]{amsaddr}
\usepackage{cleveref}
\usepackage{afterpage}
\usepackage{enumitem}
\usepackage{romannum}
\usepackage{float}
\usepackage{etoolbox}
\usepackage{bbm}
\usepackage{setspace}
\usepackage{lineno}
\usepackage{glossaries-extra}
\usepackage{nicematrix}
\usepackage{tikz}

\newcommand{\revised}[1]{\textcolor{black}{#1}}
\newcommand{\revisedTC}[1]{\textcolor{black}{#1}}

\hypersetup{
    colorlinks,
    citecolor=black,
    filecolor=black,
    linkcolor=black,
    urlcolor=black
}

\def\equationautorefname#1#2\null{%
  #1(#2\null)%
}

\newcommand{\R}{\mathbb{R}}
\newcommand{\N}{\mathbb{N}}
\newcommand{\cG}{\mathcal{G}}

\newcommand{\cS}{\mathcal{S}}

\newcommand{\cL}{\mathcal{L}}

\newcommand{\cW}{\mathcal{W}}
\newcommand{\cD}{\mathcal{D}}
\newcommand{\lp}{\left(}
\newcommand{\rp }{\right)}

\DeclareMathOperator{\Tr}{tr}
\DeclareMathOperator{\Sp}{Spec}
\newcommand{\pd}{\partial}

\DeclareMathOperator*{\argmin}{arg\,min}
\newcommand{\tpmod}[1]{{\@displayfalse\pmod{#1}}}

\numberwithin{equation}{section}

\usepackage[left=2.5cm,right=2.5cm,top=2.5cm,bottom=2.5cm]{geometry}
\usepackage{array}

\newcolumntype{C}{>{\raggedright\arraybackslash}p{1.5cm}<{}}

\newtheorem{theorem}{Theorem}[section]
\newtheorem{corollary}{Corollary}[section]	
\newtheorem{example}{Example}[section]
\newtheorem{remark}{Remark}[section]	
\newtheorem{lemma}{Lemma}[section]

\newcommand{\blockOne}[1]{
  \underbrace{\begin{matrix} -1 & 1 & \cdots & -1 & 1\end{matrix}}_{#1}
}

\newcommand{\blockTwo}[1]{
  \underbrace{\begin{matrix} 1 & -1 & \cdots & 1 & -1\end{matrix}}_{#1}
}

\crefname{figure}{figure}{figures}

\AtBeginDocument{\pagenumbering{arabic}}

\title[Modelling polarity-driven laminar patterns with mixed signalling mechanisms]{Modelling polarity-driven laminar patterns in bilayer tissues with mixed signalling mechanisms}

\author[Moore]{Joshua W. Moore$^{1}$} 
\email{moorej16@cardiff.ac.uk}
\address{$^{1}$School of Mathematics, Cardiff University, Senghennydd Rd, Cardiff, CF24 4AG, UK}

\author[Dale]{Trevor C.  Dale$^{2}$}
\address{$^{2}$School of Biosciences, Cardiff University, Museum Ave, Cardiff, CF10 3AX, UK}

\author[Woolley]{Thomas E.  Woolley$^{1}$}

 \date{}

\begin{document}

\vspace{0.5cm}
\begin{abstract}
Recent advances in high-resolution experimental methods have highlighted the significance of cell signal pathway crosstalk and localised signalling activity in the development and disease of numerous biological systems.  The investigation of multiple signal pathways often introduces different methods of cell-cell communication,  i.e.  contact-based or diffusive signalling,  which generates both a spatial and temporal dependence on cell behaviours.  Motivated by cellular mechanisms that control cell-fate decisions in developing bilayer tissues,  we use dynamical systems coupled with multilayer graphs to analyse the role of signalling polarity and pathway crosstalk in fine-grain pattern formation of protein activity.  Specifically,  we study how multilayer graph edge structures and weights influence the layer-wise (laminar) patterning of cells in bilayer structures,  which are commonly found in glandular tissues.  We present sufficient conditions for existence,  uniqueness and instability of homogeneous cell states in the large-scale spatially discrete dynamical system.  \revisedTC{Using methods of pattern templating by graph partitioning to generate quotient systems in combination with concepts from monotone dynamical systems,  we exploit the extensive dimensionality reduction to provide existence conditions for the polarity required to induce fine-grain laminar patterns with multiple spatially dependent intracellular components.  We then explore the spectral links between the quotient and large-scale dynamical systems to extend the laminar patterning criteria from existence to convergence for sufficiently large amounts of cellular polarity in the large-scale dynamical system,  independent of spatial dimension and number of cells in the tissue.}
\par
\vspace{0.25cm}
\end{abstract}
\maketitle

\section{Introduction}

Cell-fate determination is the process of stem,  or progenitor, cell commitment to transition to a differentiated state with adapted cellular functions \cite{wagers2002cell}.  This process enables the generation of specialised cell populations during organ development as cells propagate through lineage structures with each cell-fate choice.  Cell-fate decisions are typically regulated by tightly orchestrated intracellular protein cascades often termed genetic regulatory networks (GRNs) that describe complex intracellular protein interactions that depend on both the local cellular environment and intrinsic genetic properties of the cell \cite{saez2022dynamical}.   Subsequently,  the investigation of the autonomous spatial organisation of cell-fate biomarkers (active proteins for cell-fate regulation) in developing tissues has been of significant interest as a strategy to elucidate the intracellular mechanisms that govern such cellular behaviour bifurcations \cite{artavanis1991choosing,Lilja2018,perrimon2012signaling}.  \par

Ductal structures commonly found in glandular tissues possess some of the most simple cell-fate biomarker spatial patterning.  Primarily comprised of just two cell types,  these tissues produce branching morphologies with a consistent bilayer ring of layer-wise contrasting cell types (Figure \ref{fig:Bilayer_intro_fig}A).  To form the bilayer rings,  undifferentiated cells self-organise to autonomously produce distinct laminar patterns of opposing cell-fate biomarkers to promote the substance transportation functions required of the organ \cite{hogan2002molecular}.  Common examples of these laminar pattern features can be found in glandular tissues,  such as: the mammary,  salivary and sweat glands where the emergence of bilayer cell-type expression initiates the formation of ductal features \cite{de2017overview,saga2002structure, sumbal2020primary}.  \par

\begin{figure}[t!]

\centering
\begin{subfigure}[b]{\textwidth}
         \centering
         \includegraphics[width=\textwidth]{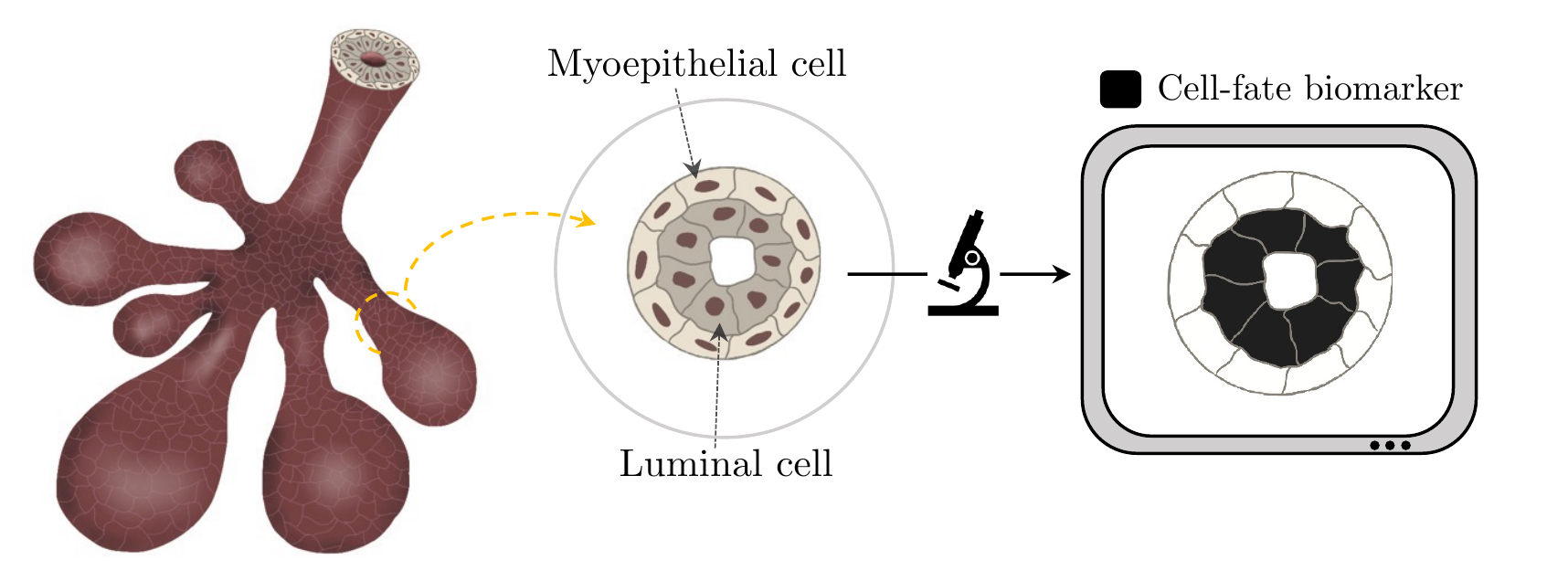}
         \caption{Bilayer structures in glandular tissues.}
\end{subfigure}

     \begin{subfigure}[b]{\textwidth}
     \vspace{1em}
         \centering
         \includegraphics[width=0.9\textwidth]{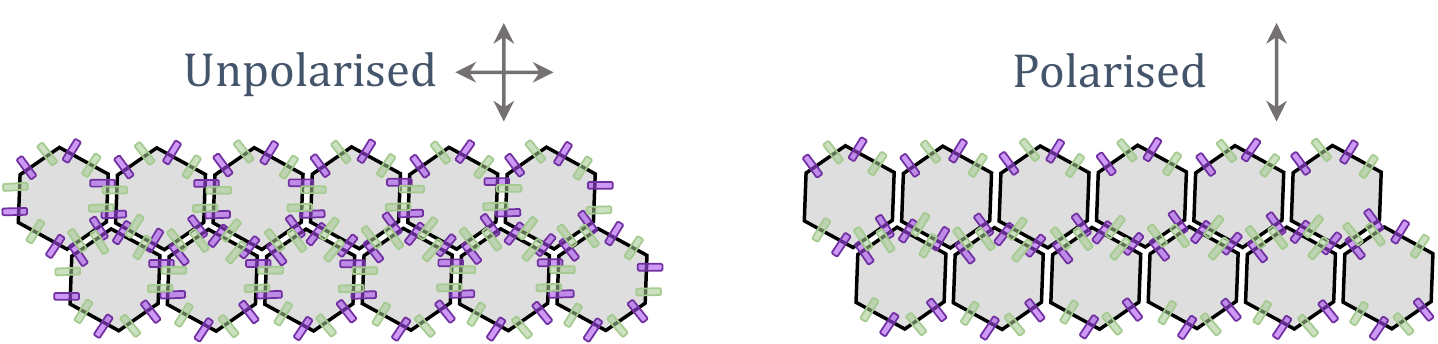}
         \caption{Cell signalling polarity.}
\end{subfigure}

\caption{Schematic representation of the bilayer cellular architecture of ductal tissues and cellular polarity.  (A) Branching tissue structures produce bilayer ducts of layer-wise contrasting epithelial cells with the outer and inner layers of myoepithelial and luminal cells,  respectively,  to facilitate the production and secretion for the transportation of substances such as sweat,  saliva and breast milk \cite{de2017overview,saga2002structure,sumbal2020primary}.  Imaging the ducts for key cell-fate biomarkers highlight the presence of laminar patterns intracellular protein expression.  Some well-established cell-fate biomarkers are p63 and Notch for myoepithelial and luminal cells in mammary and sweat glands \cite{Lloyd-Lewis2019, kurata2014isolation}.  (B) An example contact-dependent signalling mechanism with activators and receptors are represented using purple and green rectangles depicting the effect of layer-wise polarity in a bilayer of cells.    } \label{fig:Bilayer_intro_fig}

\end{figure}

Though the spatial patterning of the cell-fate biomarkers can be simple in bilayer tissues,  the cellular regulatory mechanisms that generate these cell states are complex,  often involving multiple inter-linked GRNs,  which is referred to as pathway crosstalk \cite{jimenez2012stem}.  Consequently,  pathway crosstalk may induce multiple modes of independent cell-cell communication channels,  such as contact-based (juxtacrine) or short-range diffusive (paracrine) protein interactions between local cells \cite{Lloyd-Lewis2019,  jarde2016wnt}.  Thus,  cells receive local tissue information from a range of sources to consistently select the appropriate cell type for functional ductal formation.  \par

\revisedTC{In addition to the rich interactions of intracellular GRNs,  a common control mechanism present in developing tissues is cellular polarity.  In a broad sense,  cell polarity is the asymmetry of the shape or molecular distribution of the cell \cite{campanale2017development}.  In this study we focus biochemical asymmetries that are associated with localisation of signalling activators and receptors within,  or on,  the surface of the cell and refer to this as polarity herein \cite{buckley2022apical}.  In general, polarity in cell signalling proteins causes anisotropic communication flow between cells and therefore specifies a niche of adjacent interaction cells,  as shown in Figure \ref{fig:Bilayer_intro_fig}B as an example of apical-basal polarity in an epithelial bilayer.  The role of polarity is well-established in mechanical processes of mammalian development such as division and adhesion,  to ensure consistent morphological features of the tissue during growth \cite{vorhagen2014mammalian}. 
 }\par
 
\revisedTC{
There is growing evidence in a variety of biological systems where polarity also acts as a spatial coordinator of cell-fate specification and which promotes mechanical feedback loops to preserve local cell types for healthy tissue function \cite{motegi2020novel}.  However,  the precise role polarity has in influencing intracellular kinetics that govern cell-fate choices is widely unknown,  since cell-cell interactions compound with complex pathway crosstalk \cite{vorhagen2014mammalian}.  Therefore,  in this study we explore the interplay of polarity and multiple cell-cell signalling mechanisms associated with pathway crosstalk in generating laminar patterns of biomarkers,  conforming with the process of cell-fate determination of developing bilayer tissues. 
}\par

Following Turing's seminal paper in 1952 \cite{turing1952chemical},  the majority of theoretical results of pattern formation in developmental biology focus on diffusion-driven instabilities of reaction-diffusion (RD) systems \cite{krause2021introduction,  woolley2017turing}.  RD systems rely on the assumption that cells communicate using short-range and/or long-range paracrine signalling mechanisms,  namely the local diffusion of proteins coupled with intracellular kinetics.   However,  there exists many pattern forming biological systems that rely on non-diffusive,  juxtacrine communication,  such as lateral-inhibition mechanisms,  where adjacent cells inhibit each other from converging to the same state,  facilitating fine-grain pattern formation \cite{GilbertScottF.2016Db,sprinzak2010cis}.  \par

The fundamental differences in paracrine and juxtacrine signalling motivate contrasting modelling approaches.  The diffusion process in paracrine signalling extends to a spatial continuum limit,  generating small systems of partial differential equations (PDEs) allowing protein patterns to form over multiple cell lengths to represent phenomena such as morphogen gradients over the tissue \cite{woolley2017turing}.  Whereas the discrete nature of juxtacrine signalling induced by membrane contact necessitates the use of spatially discrete systems of ordinary differential equations (ODEs) \cite{Wearing2000}.  Subsequently,  these contrasting modelling paradigms restrict the specific continuum and discrete approaches to pattern analysis in systems where both diffusive and non-diffusive mechanisms are present.  \par

Graphs representing spatially discrete analogues of diffusive mechanisms have previously been employed to homogenise the analytical approaches to pattern formation and,  further,  investigate cell structure on pattern emergence \cite{horsthemke2004network}.  That is,  graph vertices depict cells and edges are drawn between cells if they are communicating via diffusive proteins.  Critically,  this approach preserves the concept of cell identity within diffusive models and transforms the systems of PDEs into much larger systems of ODEs,  consistent with the juxtacrine model formulation.  However,  the central theme of pattern analysis is understanding the conditions that yield the degradation of stable homogeneity of the system and is typically conducted via linear stability analysis with coupled spatio-temporal components \cite{Wearing2001,  Collier1996,  murray2003ii}.  Consequently,  the high-dimension of these ODE descriptions and required nonlinear kinetics of multicellular domains limit analytical approaches intractable which lead to many studies focusing on spatially reduced systems accompanied by numerical simulations for the larger cellular domains \cite{Collier1996,Wearing2000,Wearing2001}.  Critically,  the analysis conducted on such spatially reduced models have been shown to be insufficient for predicting the types of patterning observed numerically \cite{Wearing2001},  with similar results for cell-resolution discretised diffusive systems \cite{moore2005localized}.  \par

\revised{Adopting concepts from systems engineering,  the application of interconnected dynamical systems theory was employed in \cite{Arcak2013} to derive analytic pattern formation conditions for juxtacrine models,  independent of the number of cells and therefore the size of the ODE system.  Namely,  cells were treated as input-output (IO) components components within a circuit,  i.e.  cells receive signals and then produce response signal to other connected cells. This approach recasts the large-scale ODE system in a macroscopic perspective to analyse the behaviour of only the directly spatially-dependent intracellular proteins using signal transfer functions.  Furthermore,  edge symmetries of the cell-cell connectivity graphs were exploited in \cite{RufinoFerreira2013} to develop methods of graph partitioning to form quotient graphs that represent a pattern template.  Embedding the intracellular ODE systems defined by the GRNs within these quotient graphs produce a quotient interconnected dynamical system which are comparatively significantly smaller in dimension.  These quotient systems were then used to provide pattern existence conditions for prescribed cellular patterns for interconnected juxtacrine models.  These methods of pattern predictions were later extended in \cite{gyorgy2017pattern} and \cite{gyorgy2017patternMulti} to simultaneously couple diffusive and non-diffusive signalling mechanisms within the interconnected dynamical systems framework using directed multilayer graphs,  namely graphs with unidirectional edges connected to cells with multiple-input and output signals. } \revisedTC{ However,  the influence of edge weights with on pattern existence and convergence in these multi-channel interconnected systems is yet to be investigated.  } \par   

The spatial scalability of the interconnected methods of pattern analysis follows from the theory of monotone dynamical systems \cite{Angeli2003}.  Provided the intracellular proteins regulated by the prescribed GRN react monotonically to intercellular stimuli,  then global dynamics become predictable in a closed-loop of cells and facilitates the introduction of control theoretic tools for pattern stability \cite{RufinoFerreira2013}.  Although,  the restriction to bipartite connectivity graphs was imposed in \cite{Arcak2013} and \cite{gyorgy2017patternMulti} as a sufficient measure to preserve the monotonic behaviour of lateral-inhibition models in the large-scale forms,  these restrictions limit the biological applications.  Such conditions can be relaxed when seeking pattern existence in quotient systems as demonstrated in \cite{RufinoFerreira2013} but how such behaviour translates to the large-scale counterpart is not fully understood.

We have previously analysed the role of polarity in laminar pattern formation using interconnected methods for a single juxtracrine signalling mechanism \cite{moore2021algebraic}.  In this study,  we generalise and extend these results to include multiple signalling mechanisms of any type using a multilayer graph approach as defined in \cite{gyorgy2017patternMulti}.  Namely,  we explore the interplay of multilayer network topology and edge weights in laminar pattern formation in bilayer tissues using dynamical systems of generic competitive kinetics.  \par

Initially,  we present conditions for the existence,  uniqueness and instability of a homogeneous steady state for large-scale multi-input-multi-output (MIMO) dynamical system which extends the conditions of \cite{gyorgy2017patternMulti} to yield analytically applicable statements for low-spatial order intracellular GRNs.  Thereafter we use methods of multilayer graph partitioning to derive polarity conditions for the existence of laminar patterning in large-scale systems.  Critically,  we demonstrate the graph commutativity requirements imposed in \cite{gyorgy2017patternMulti} for simultaneous diagonalisation can be relaxed when seeking patterns of only two states,  allowing a broader range of quotient connectivities to be explored.    \par

Next,  we investigate the spectral links between quotient and large-scale dynamical systems.  We demonstrate positional changes of the eigenvalues associated with laminar patterns in the multigraphs are dependent on the amount of polarity for non-bipartite graphs.  We then discuss the implications of spectral rearrangements with respect to bipartite graphs and laminar patterning.  Finally,  combining our insights from the spectral rearrangements and quotient system analysis we explore the convergence of laminar patterns in the associated large-scale dynamics systems.\par  

The structure of the study is as follows.  In Section \ref{sec:model_definition} we define the large-scale interconnected dynamical system analysed in this study.  In Section \ref{sec:large_Scale_instability} we present conditions for the existence,  uniqueness and instability of a homogeneous steady state for large-scale MIMO dynamical system.  Our main results are presented in Section \ref{sec:laminar_converge} where we introduce the necessary results from monotone dynamical systems in Section \ref{sec:monotone_intro} and then present conditions for the existence of laminar patterning in Section \ref{sec:quotient_pattern}.  In Section \ref{sec:spec_links},  we demonstrate positional changes of the eigenvalues associated with laminar patterns in the multigraphs are dependent on the amount of polarity for non-bipartite graphs.  Finally, in Section \ref{sec:global_convergence},  we present sufficient polarity dependent conditions for the convergence of laminar patterns in the large-scale systems.

\section{Existence of cellular heterogeneity}
In this section,  we are interested in deriving conditions for the existence and instability of a homogeneous steady state (HSS) of a large-scale dynamical system that describes intracellular kinetics within a tissue of cells.  First, we define the types of interconnected dynamical systems considered in this study,  namely,  coupling the multiple input and output signal dynamics of individual cells using weighted connectivity graphs associated with each respective signalling mechanism.  Thereafter,  we exploit the repetitive structure of large-scale interconnected dynamical systems to provide analytically tractable conditions for the existence,  uniqueness and stability of the HSS that is necessary for the investigation of spatially-driven cellular heterogeneity.

\subsection{The signal polarity interconnected system for bilayer geometries with multiple signal mechanisms}\label{sec:model_definition}$\,$ \\
\revised{
Consider a large-scale interconnected dynamical systems representing $N$ spatially discrete cells,  each containing $n$ intracellular proteins.  Namely,  for each cell $i \in \{ 1,...,N \}$,   let $\bm{x}_{i}  = [x_{i,1},...,x_{i,n}]^{T} \in X \subset \R_{\geq 0}^{n}$ be the concentration of the intracellular proteins.  The cellular signal inputs and outputs are defined by $\bm{u}_{i} = [u_{i,1},...,u_{i,r}]^{T},  \, \bm{y}_{i} = [y_{i,1},...,y_{i,r}]^{T}  \in U,Y \subset \R_{\geq 0}^{r}$,  respectively,  for $1 \leq r \leq n$.  The interconnected ODE system has the form
\begin{align} \label{eqn:IO_system}
\dot{\bm{x}}_{i} &= \bm{f} \lp \bm{x}_{i}, \bm{u}_{i} \rp, \\
 \bm{y}_{i} & = \bm{h} \lp \bm{x}_{i} \rp, 
\end{align}
where $\dot{\bm{x}}_{i}$ represents the derivative with respect to time.  The function $\bm{f}: X \times U \rightarrow X$ defines the intracellular protein dynamics which are dependent on external stimuli, $\bm{u}_{i}$,  produced by connected cells.  We define cellular connectivity in terms of multiple signalling mechanisms later in this section.  Furthermore,  $\bm{h}: X \rightarrow Y$ describes the translation of intracellular dynamics to signal outputs of the cell.  We assume that both functions $\bm{f}\lp \cdot \rp$ and $\bm{h} \lp \cdot \rp$ are both $\mathcal{C}^{2}$ over their respective domains to ensure the continuity of the corresponding linearised system that is required for the interconnected pattern analysis in Section \ref{sec:laminar_converge}.  The structure of the IO system (\ref{eqn:IO_system}) in context of the tissue is shown in Figure \ref{fig:Mimo_overview}.  For convenience when discussing tissue behaviour,  we define the large-scale vectorised counterparts of the intracellular state variables,  signal inputs and outputs by $\bm{x} = [\bm{x}_{1}, ...,\bm{x}_{n}]^{T}$,  $\bm{u} = [\bm{u}_{1}, ...,\bm{u}_{r}]^{T}$ and $\bm{y} = [\bm{y}_{1}, ...,\bm{y}_{r}]^{T}$. } \par

\revised{
For the transition of signal outputs to inputs,  we assume that each output signal is independent and defines a linear relationship between output and input signals.  Let $V \coloneqq \{ v_{1},...,v_{N} \}$,  be vertices representing the cells in the tissue,  then for each output signal $y_{i,j}$ there is an associated connectivity graph $\cG_{j} = \cG_{j}\lp V, E_{j}  \rp $,  where $E_{j}$ is the set of edges for each output signal mechanism $1 \leq j \leq r$.  Note that the vertex set $V$ is identical for each connectivity graph whereas edge structure may differ between the respective graphs to allow for different signalling mechanisms within the IO system (\ref{eqn:IO_system}).  For example,  the cellular connectivity graphs of contact-dependent and long-diffusion mechanisms have potentially different edge structures as it is expected that the average degree of the contact-based graph is less than that of a diffusive mechanism due to the physical constraints of cellular junctions (Figure \ref{fig:Mimo_overview}). } \par

\begin{figure}[h!]
\includegraphics[width=\textwidth]{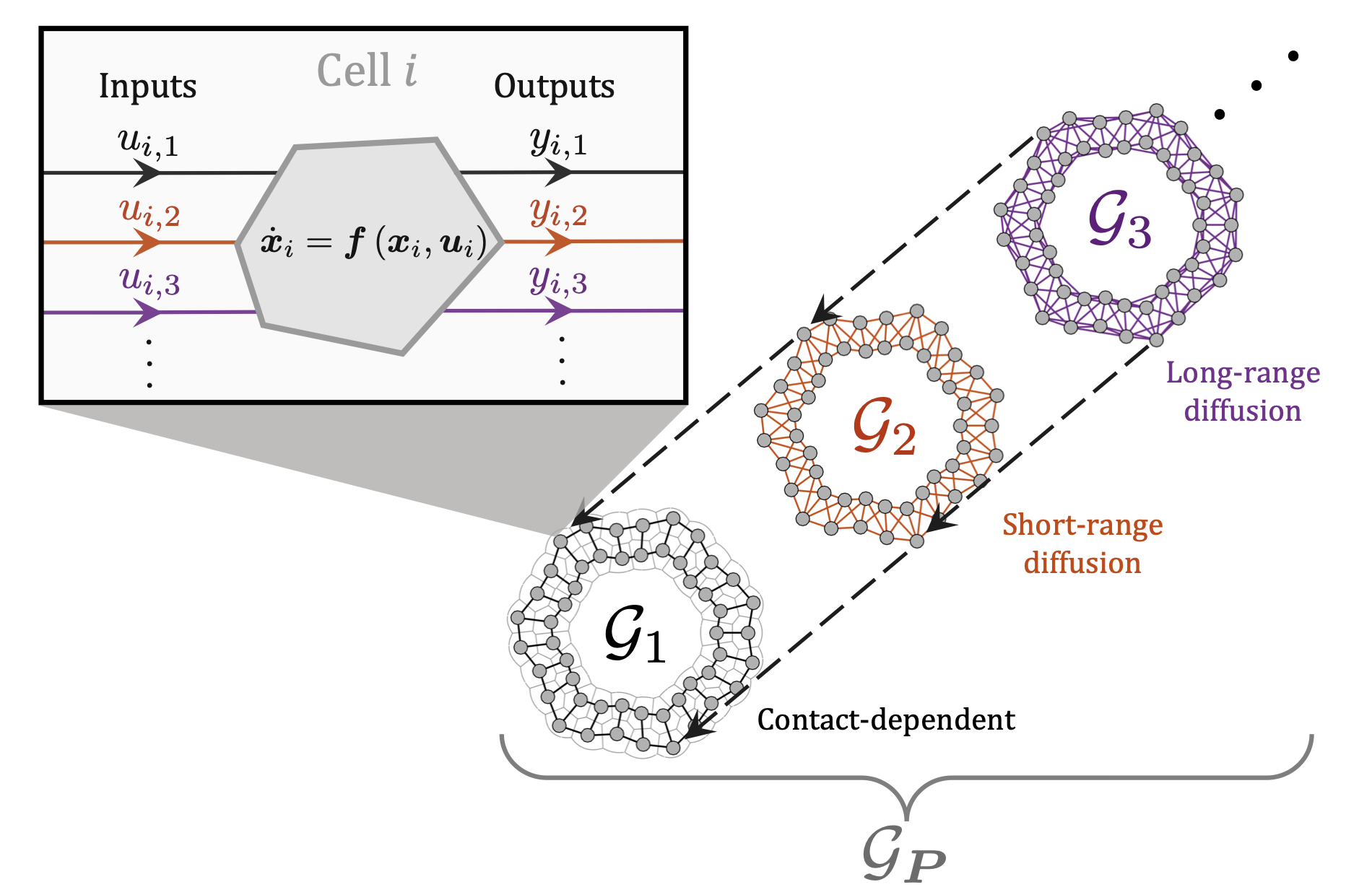}
\caption{A graphical representation of the IO system (\ref{eqn:IO_system}) for bilayer geometries with multiple signalling mechanisms defined by the global adjacency matrix $\bm{P}$ (\ref{eqn:interwoven_adj_mat}).  Example 2D bilayer graphs are shown for contact-dependent short and long-range diffusion over the same vertex set representing the cells in the tissue,  which are explicitly shown with membranes in $\cG_{1}$ to highlight the bilayer cellular structure.   Each of the connectivity graphs is then embedded within the same vertex set as indicated by the dashed arrows.  Therefore,  each vertex contains the intracellular kinetics defined by the IO system (\ref{eqn:IO_system}) which responds to the signal outputs of adjacent cells for each signalling mechanism which are transformed to signal inputs by $\bm{P}$ as defined in equation (\ref{eqn:input_to_output}).  } \label{fig:Mimo_overview}
\end{figure}

\revised{
Algebraically,  the cell-cell interaction graphs are represented using the weighted adjacency matrix,  $ \bm{W}_{j} \in \R^{N \times N}_{\geq 0 }$.  Let $\bm{\cW} \coloneqq \left\lbrace \bm{W}_{j}  \right\rbrace$ be the set of weighted and row-stochastic adjacency matrices,  namely,  for any  $j \in \{1,...,r\}$ and any row $i \in \{ 1,...,n\}$ then row-sum $\sum_{k} \lp \bm{W}_{j} \rp _{ik}  = 1$ which represents the weighted average of signal transfer between connected cells.  In addition,  we assume that the connectivity graph $\cG_{j}$ associated with $\bm{W}_{j}$ is undirected and connected,  and thus $\bm{W}_{j}$ is symmetric and irreducible,  i.e.,  there exists no permutation matrix that transform $\bm{W}_{j}$ to upper triangular form \cite{godsil2001algebraic}.  
}

\revised{
To preserve the order of signal outputs, $y_{i,j}$,  and therefore the cellular structure within the IO system (\ref{eqn:IO_system}),  we define the global interconnection matrix, $\bm{P}$, that is constructed by interweaving the each $\bm{W}_{j} \in \bm{\cW}$ in order of output signal defined by $y_{i,j}$,  namely,
\begin{equation}\label{eqn:interwoven_adj_mat}
\bm{P} = \sum_{j=1}^{r} \bm{W}_{j} \otimes \bm{D}_{j},
\end{equation}
where $\otimes$ is the Kronecker product and $\bm{D}_{j} = \text{diag}\lp \delta_{j,1}, ..., \delta_{j,r} \rp$ for $\delta_{i,j}$ the Kronecker delta function 
\begin{equation}\label{eqn:Kron_delta}
\delta_{i,j} =  \begin{cases} 
      1 & i = j, \\
      0 & i \neq j.
   \end{cases}
\end{equation}}
\revised{
The global interconnection matrix $\bm{P} \in \R^{rN \times rN}$ therefore produces a multilayer graph $\cG_{\bm{P}}$ that is layer-wise independently as shown in Figure \ref{fig:Mimo_overview}.  Critically,  the construction of $\bm{P}$ defines the linear relationship between global signal outputs and inputs 
 \begin{equation}\label{eqn:input_to_output}
 \bm{u} = \bm{P}\bm{y}
 \end{equation}
where cell-wise input-output structure is preserved.} The fundamental cellular identity preserving structure of $\bm{P}$ is demonstrated in the following example. 
\begin{example}\label{example:interweave_mats}
Consider the two general matrices 
\begin{equation}
\bm{W}_{1} = \begin{bmatrix}
a_{11} & a_{12} \\
a_{21} & a_{22}
\end{bmatrix} \quad \text{and} \quad \bm{W}_{2} = \begin{bmatrix}
b_{11} & b_{12} \\
b_{21} & b_{22}
\end{bmatrix} 
\end{equation}
for the IO system (\ref{eqn:IO_system}) with only two cells each with two signal inputs and outputs i.e.,  $r = 2$.  Then the global interconnection matrix,  $\bm{P}$,  has the form
\begin{equation}
\bm{P} = \begin{bmatrix}
a_{11} & a_{12} \\
a_{21} & a_{22}
\end{bmatrix} \otimes  \begin{bmatrix}
1& 0 \\
0 & 0
\end{bmatrix}  + \begin{bmatrix}
b_{11} & b_{12} \\
b_{21} & b_{22}
\end{bmatrix} \otimes  \begin{bmatrix}
0& 0 \\
0 & 1
\end{bmatrix}  =  \begin{bmatrix} 
a_{11} & 0 & a_{12}  & 0 \\
0 & b_{11} & 0 & b_{12}\\
a_{21} & 0 & a_{22}  & 0 \\
0 & b_{21} & 0 & b_{22}\\
\end{bmatrix}.
\end{equation}

\end{example}

In order to study the role of layer-dependent signalling polarity for the generation of laminar pattern in bilayer geometries,  we consider each graph to have two layers $\cL_{1} \coloneqq \{v_{1},...,v_{|\cL_{1}|} \} $ and $\cL_{2} \coloneqq \{v_{|\cL_{1}| + 1},...,v_{N} \} $ where $|\cL_{1}| = 1,..., N-1$,  as shown in figures \ref{fig:Mimo_overview} and \ref{fig:Contact_ex}. 

\revised{This layer-wise grouping of the vertices also provides consistent structure to the weighted adjacency matrices $\bm{W}_{k} \in \bm{\cW}$.  As a first-approach to the layer-dependent signal polarity,  we consider only two values of edge weights for connected cells in the same and different layers as highlighted in Figure \ref{fig:Contact_ex}} .  \revised{Namely,  consider the graph $\cG_{k}$ associated with $\bm{W}_{k}$,  then if $ v_{i},v_{j} \in \cL_{1} \, ( \text{or }\cL_{2})$ such that $v_{i} $ and $v_{j}$ are connected by and edge in $\cG_{k}$ and are in the same layer,  then $(\bm{W}_{k})_{ij} = \hat{w}_{1}^{[k]}$,  for $\hat{w}_{1}^{[k]}$ the row-normalised intralayer edge weight.  Similarly,  if $v_{i}$ and $v_{j}$ are in different layers,  $v_{i} \in \cL_{1}$ and $v_{j} \in \cL_{2}$, and are connected in $\cG_{k}$ then,  $(\bm{W}_{k})_{ij} = \hat{w}_{2}^{[k]}$,  for $\hat{w}_{2}^{[k]}$ the row-normalised interlayer edge weight.}  Consequently,  when vertices are indexed consecutively from $\cL_{1}$ then $\cL_{2}$,  each $\bm{W}_{k}$ has block form
\begin{equation}\label{eqn:w_block_form}
\bm{W}_{k} = \begin{bmatrix}
\widehat{\bm{W}}_{1,\cL_{1}}^{[k]} & \widehat{\bm{W}}_{2,\cL_{1}}^{[k]}  \\
\lp \widehat{ \bm{W}}_{2,\cL_{1}}^{[k]}\rp^{T} & \widehat{\bm{W}}_{1,\cL_{2}}^{[k]} 
\end{bmatrix} ,
\end{equation}
where $\widehat{\bm{W}}_{1,\cL_{1}} \in \R_{\geq 0}^{|\cL_{1}| \times |\cL_{1}| }$ contains all intralayer connections scaled by $\hat{w}_{1}^{[k]}$ for all the vertices in $\cL_{1}$,  $\widehat{\bm{W}}_{2,\cL_{1}} \in \R_{\geq 0}^{|\cL_{1}| \times |\cL_{2}| }$ contains all interlayer connections scaled by $\hat{w}_{2}^{[k]}$ for all the vertices in $\cL_{1}$.  Similarly  $\widehat{\bm{W}}_{1,\cL_{2}} \in \R_{\geq 0}^{|\cL_{2}| \times |\cL_{2}| }$ accounts for the intralayer connections within $\cL_{2}$.  As each $\cG_{k}$ is undirected,  the interlayer connections for all vertices in $\cL_{2}$ are represented by $\widehat{\bm{W}}_{2,\cL_{1}}^{T}$,  that is,  $\bm{W}_{k}$ is symmetric.\par
\begin{example}
The weighted adjacency matrix $\bm{W}_{1}$ associated with $\cG_{1}$ in Figure \ref{fig:Contact_ex} has the block matrices
\begin{equation}
\widehat{\bm{W}}_{1,\cL_{1}} = \begin{bmatrix}
0 & \hat{w}_{1}^{[1]} & 0 & \cdots & 0 & \hat{w}_{1}^{[1]} \\
\hat{w}_{1}^{[1]} & 0 & \hat{w}_{1}^{[1]} & & &\\
& & \ddots & & & \\
& & & \ddots & & \\
& &  &\hat{w}_{1}^{[1]} &0 &\hat{w}_{1}^{[1]} \\
\hat{w}_{1}^{[1]}&0 & \cdots  &0 &\hat{w}_{1}^{[1]} &0 
\end{bmatrix} \quad \text{and} \quad 
\widehat{\bm{W}}_{2,\cL_{1}} = \begin{bmatrix}
\hat{w}_{2}^{[1]}& 0 & 0 & \cdots & 0 & 0 \\
0& \hat{w}_{2}^{[1]} & 0& & &\\
& & \ddots & & & \\
& & & \ddots & & \\
& &  &0 &\hat{w}_{2}^{[1]}&0 \\
0&0 & \cdots  &0 &0 &\hat{w}_{2}^{[1]}
\end{bmatrix}, 
\end{equation}
for $\hat{w}_{1}^{[1]} = w_{1}^{[1]} / |w^{[1]}|$  and $\hat{w}_{2}^{[1]} = w_{2}^{[1]} / |w^{[1]}|$ where  $|w^{[1]}| = 2w_{1}^{[1]} +  w_{2}^{[1]}$  is the normalising factor for all rows ensuring the row-stochastic property of $\bm{W}_{1}$.  From the regularity of $\cG_{1}$ in Figure \ref{fig:Contact_ex},  we have that $\widehat{\bm{W}}_{1,\cL_{1}} = \widehat{\bm{W}}_{1,\cL_{2}}$ as the connections within layers are identical for $\cL_{1}$ and $\cL_{2}$.
\end{example}

\begin{figure}[h!]
\includegraphics[width=0.8\textwidth]{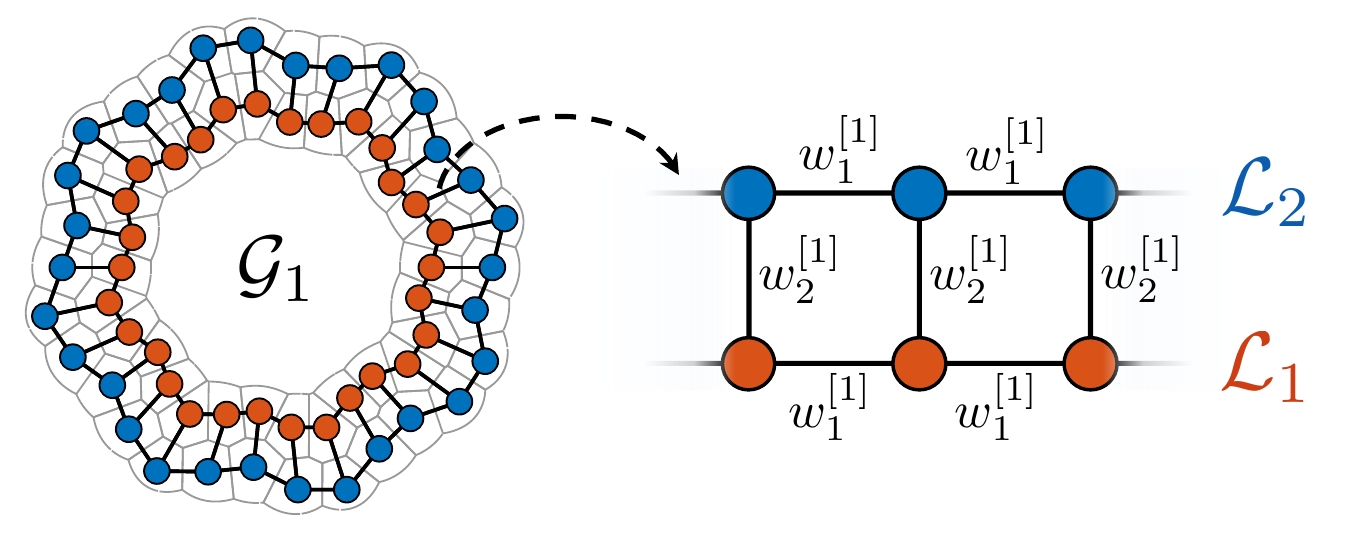}
\caption{Layer-dependent edge weight structure in bilayer graphs.  A 2D bilayer graph,  $\cG_{1}$, with contact-dependent edge connections highlight the layer vertex partitions with vertices in $\cL_{1}$ and  $\cL_{2}$ coloured orange and blue,  respectively.  The edge weight structure within and between the layers is shown with edges between vertices in the same layer weighted by $w_{1}^{[1]}$ and connected vertices in different layers weighted by $w_{2}^{[1]}$.  } \label{fig:Contact_ex}
\end{figure}

To summarise the internal cellular dynamics in terms of signal inputs and outputs,  as proposed in \cite{Arcak2013},  we introduce the transfer function $\bm{T}: U \rightarrow Y$ that describes cellular signal output response with respect to changes to input signals determined by connected cells.  It is assumed that $\bm{T}\lp \cdot \rp$ is bounded and $\mathcal{C}^{2}$ which conforms with the biological context of the IO system (\ref{eqn:IO_system}),  namely,  intracellular expression must remain finite with continuous dependence on the cellular microenvironment.  The introduction of $\bm{T}\lp \cdot \rp$ allows of the analysis of the IO system (\ref{eqn:IO_system}) from an alternative macroscopic perspective,  such that $\bm{T}\lp \cdot \rp$ retains the underlying features of the intracellular kinetics defined by $\bm{f}\lp \cdot \rp$ and $\bm{h}\lp \cdot \rp$ while not explicitly defining the intracellular interactions.  For instance,  intercellular communication of lateral-inhibition and lateral-induction pathways have a decreasing and increasing transfer function $\bm{T}\lp \cdot \rp$,  respectively \cite{wolpert2015principles, Arcak2013}.  Explicitly,  the transfer function allows for the definition of the auxiliary input to output transition relation
\begin{equation}\label{eqn:aux_input_output}
\bm{u} = \bm{P} [ \bm{T} \lp   \bm{u}_{1} \rp , ...,  \bm{T} \lp \bm{u}_{N} \rp ]^{T},
\end{equation}
which reduces the analytic complexity of the macroscopic analysis of spatially driven pattern formation in large-scale systems as the dependence of cellular coupling is more accessible in this form \cite{Arcak2013,RufinoFerreira2013,gyorgy2017patternMulti}.  However,  these methods require the characteristic behaviour of the GRNs and therefore $\bm{T} \lp \cdot \rp$ is known with respect to spatial stimuli.  \par

In the following section,  we show that the zeros of the auxiliary input to output transition equation (\ref{eqn:aux_input_output}) are the steady states of the IO system (\ref{eqn:IO_system}),  thus enabling stability analysis of the homogeneous steady states macroscopically.  Subsequently,  we derive conditions for the existence and uniqueness of the HSS in the large-scale system.  Thus, to induce polarity driven pattern formation within the IO system (\ref{eqn:IO_system}),  we seek sufficient conditions for the instability of HSS dependent on the bilayer connectivity graphs $\cG_{k}$,  and in particular,  the polarity weights,  $w_{1}^{[k]}$ and $w_{2}^{[k]}$.

\subsection{Existence,  uniqueness and stability of the homogeneous steady state in the large-scale IO systems}\label{sec:large_Scale_instability}$\,$\\
The majority of statements presented in this section were initially stated in \cite{gyorgy2017patternMulti} for MIMO IO systems.  Here,  we have independently proven them and partially extended them to comment on the uniqueness of the HSS.  We include all results for completeness with a focus on the application to mixed signal mechanisms in bilayer geometries.
Consider the $\mathcal{C}^{2}$ function $\bm{S}: U \rightarrow X$ that describes the changes to the intracellular kinetics $\bm{x}_{i}$ by the input signals $\bm{u}_{i}$ emanating from connected cells.  Therefore,  the following statement demonstrates that the zeros of the the auxiliary transfer relation (\ref{eqn:aux_input_output}) are the steady states of the IO system (\ref{eqn:IO_system}).
\begin{lemma}[\cite{gyorgy2017pattern}]\label{lemma:aux_ss}
Assume that for any $\bm{u}_{0} \in \R^{r}$ the function $\bm{f} \lp \bm{x} , \bm{u}_{0}  \rp = 0$ has a solution denoted by $\bm{x}_{0}  = \bm{S} \lp \bm{u}_{0}  \rp$ and therefore $\bm{T} \lp \bm{u}_{0} \rp = \bm{h} \lp \bm{S} \lp \bm{u}_{0} \rp \rp$.  Then if $\bm{u}_{0} $ satisfies
\begin{equation}\label{eqn:lemma_aux_eqn}
\begin{bmatrix}
\bm{u}_{0}\\
\vdots \\
\bm{u}_{0}
\end{bmatrix} =
\bm{P} \begin{bmatrix}
\bm{T} \lp \bm{u}_{0} \rp \\
\vdots\\
\bm{T} \lp \bm{u}_{0} \rp
\end{bmatrix}
\end{equation}
then $\bm{x}_{0} = \bm{S} \lp \bm{u}_{0} \rp$ is a steady state of the IO system (\ref{eqn:IO_system}).  Conversely,  if $\bm{S} \lp \cdot \rp$ is injective and $\bm{x}_{0}$ is a fixed point of the IO system (\ref{eqn:IO_system}),  then the corresponding $\bm{u}_{0}$ satisfies the auxiliary system (\ref{eqn:lemma_aux_eqn}).
\end{lemma}

Following from Lemma \ref{lemma:aux_ss},  we now study the transfer dynamics defined by $\bm{T}\lp \cdot \rp$ for the existence of the steady states of the IO system (\ref{eqn:IO_system}).  Critically,  as $\bm{T}\lp \cdot \rp$ represents changes in intercellular signalling,  $\bm{T}\lp \cdot \rp$ is bounded.  Therefore,  the following statement ensures the existence and uniqueness of a homogeneous steady states of the  IO system (\ref{eqn:IO_system}) using the boundness of transfer dynamics. 

\begin{lemma}\label{lemma:exist_and_uni_HSS}
Provided that the conditions of Lemma \ref{lemma:aux_ss} hold,  then for any $\bm{u}_{0} \in \R^{r}$ where $\bm{u}^{*} = \bm{1}_{N} \otimes \bm{u}_{0}$ there exists $\bm{x}_{0} \in \R^{n}$ such that $\bm{x}^{*} = \bm{1}_{N} \otimes \bm{x}_{0}$ is a steady state if the IO system (\ref{eqn:IO_system}).  Moreover,  if $ \partial \bm{f}  \lp \bm{x}, \bm{u}_{0} \rp  / \partial \bm{x} $ is invertible for all $\bm{x} \in X$ then $\bm{x}^{*}$ is unique.
\end{lemma}
\begin{proof}
It is sufficient to show that $\exists \bm{u}_{0} \in \R^{r}$ such that $\bm{u}^{*} = \bm{1}_{N} \otimes \bm{u}_{0}$ satisfies the auxiliary system (\ref{eqn:lemma_aux_eqn}) as $\bm{x}_{0} = \bm{S} \lp  \bm{u}_{0} \rp$.  As each $\bm{W}_{j} \in \bm{\cW}$ is row-stochastic,  then the global interconnection matrix $\bm{P}$ is also row-stochastic by construction.  Consequently,  there exists an eigenvalue $\lambda$ of $\bm{P}$ such that $\bm{P}\bm{1}_{rN} = \lambda \bm{1}_{rN}$ \cite{johnson1981row},  and therefore the proof follows from verifying the existence of $\bm{u}_{0}$ that satisfies $\bm{u}_{0} = \lambda \bm{T} \lp \bm{u}_{0} \rp$. \par

By the bounded property of $\bm{T}:U \rightarrow Y$,  there exists some constant $m>0$ where $|| \lambda \bm{T} \lp \cdot \rp ||_{2} \leq m$.  Consider the function $\bm{F}: B_{m} \rightarrow B_{m}$ where $\bm{F} \lp \cdot \rp = \lambda \bm{T} \lp \cdot \rp$ and $B_{m} = \{ \bm{v} \in \R^{r} : || v ||_{2}  \leq m  \}$,  noting that $B_{m}$ is a convex set and the continuity of $\bm{F}\lp \cdot \rp$ is induced by the continuity of $\bm{T} \lp \cdot \rp$. Therefore,  by the Brouwer Fixed-Point Theorem \cite{leoni2017first},  there exists some $\bm{u}_{0} \in B_{m}$ such that $\bm{u}_{0} = \bm{F}\lp \bm{u}_{0} \rp = \lambda \bm{T} \lp \bm{u}_{0} \rp$.  \par

The uniqueness of the HSS is guaranteed by the following.  Assume that for any $\bm{u}_{0} \in \R^{r}$ there exists $\overline{\bm{x}}_{1} ,\overline{\bm{x}}_{2} \in \R^{n}$ where both are solutions to $\bm{f}\lp \bm{x},  \bm{u}_{0} \rp = \bm{0}$.  Specifically,  $f_{j} \lp \overline{\bm{x}}_{1},  \bm{u}_{0} \rp = f_{j} \lp \overline{\bm{x}}_{2},  \bm{u}_{0} \rp  $ for all $j \in [1,n]$.  Therefore by the Mean Value Theorem \cite{kumar2014advanced},  we construct the linear system
\begin{equation}
\begin{bmatrix}
0\\
\vdots \\
0
\end{bmatrix} 
= 
\begin{bmatrix}
\frac{\partial f_{1}} { \partial x_{1}}(\bm{x},\bm{u}_{0}) & \cdots & \frac{\partial f_{1}} { \partial x_{n}} (\bm{x},\bm{u}_{0})  \\
& \vdots & \\
\frac{\partial f_{n}} { \partial x_{1}} (\bm{x},\bm{u}_{0})  & \cdots & \frac{\partial f_{n}} { \partial x_{n}} (\bm{x},\bm{u}_{0}) 
\end{bmatrix}
\begin{bmatrix}
\overline{\bm{x}}_{11} - \overline{\bm{x}}_{21}  \\
\vdots\\
\overline{\bm{x}}_{1n} - \overline{\bm{x}}_{2n}
\end{bmatrix}
\end{equation} 
and from the Invertible Matrix Theorem the kernel of $\partial \bm{f} / \partial \bm{x}$ contains only the null vector \cite{liesen2011lineare},  \textit{i.e.} $\overline{\bm{x}}_{1} = \overline{\bm{x}}_{2}$.
\end{proof}

\begin{remark}
If the transfer function $\bm{T}: U \rightarrow Y$ is Lipschitz continuous with Lipschitz constant $k\in \left( 0 , 1 \right]$,  namely,
\begin{equation}
|| \bm{T} \lp \bm{u}_{i} \rp - \bm{T} \lp \bm{u}_{j} \rp ||_{2} \leq k || \bm{u}_{i} - \bm{u}_{j} ||_{2} 
\end{equation}
for all $\bm{u}_{i},\bm{u}_{j} \in U$.  Then the HSS defined in Lemma \ref{lemma:exist_and_uni_HSS} is unique by the Banach Fixed-Point Theorem \cite{agarwal2018banach},  independent of the invertibility of $\bm{f} \lp \bm{x}, \bm{u} \rp$.
\end{remark}

As we seek spatially driven instabilities of the HSS,  we assume the asymptotic stability of $\bm{x}^{*}$ in the absence of cellular connections.  We say a fixed-point of a system is stable if the associated Jacobian has all eigenvalues with negative real-part.  Therefore,  we are assuming that $\bm{A} \coloneqq \partial \bm{f} / \partial \bm{x}_{i}$ evaluated at $\bm{x}_{0}$ is stable i.e. ,  the intracellular kinetics are not self-exciting in the absence of interconnections.\par

A necessary feature for polarity-driven pattern formation in spatially discrete interconnected systems is the connectivity-induced instability of the HSS,  $\bm{x}^{*}$,  associated with the IO system (\ref{eqn:IO_system}),  which can be approached by linearisation.  The following results provided a convenient method of analysing the linear stability of homogeneous large-scale IO systems by assuming each cellular connectivity graph $\cG_{j}$ commutes,  thus enabling the parallel computation of eigenvalues for each adjacency matrix $\bm{W}_{j} \in \bm{\cW}$,  reducing the dimensionality of the linearisation.  

\begin{lemma}\label{lemma:large_linear}
Let $\bm{A} \coloneqq \partial \bm{f} / \partial \bm{x}_{i}$,  $\bm{B} \coloneqq \partial \bm{f} / \partial \bm{u}_{i}$ and $\bm{C} \coloneqq \partial \bm{h} / \partial \bm{x}_{i}$,  be each evaluated at the steady state $\bm{x}_{0}$ for fixed $\bm{u}_{0}$.  Let the steady state of the global IO system be $\bm{x}^{*} = \bm{1}_{N} \otimes \bm{x}_{0}$.  Assume that all $\bm{W}_{j} \in \bm{\cW}$ commute and denote $\bm{\Lambda}_{j} = \text{diag} \lp \lambda_{1,j},...,\lambda_{r,j}  \rp$ where $\lambda_{i,j}$ is the $j$th eigenvalue of $\bm{W}_{i}$ w.r.t.  the common eigenbasis of all matrices in $\bm{\cW}$.  Then $\bm{x}^{*}$ is asymptotically stable if $\bm{A}  + \bm{B} \bm{\Lambda}_{j} \bm{C}$ is stable for all $j$ and unstable otherwise.
\end{lemma}
\begin{proof}
Linearisation of the global IO system (\ref{eqn:IO_system}) about the fixed point $\bm{x}^{*} =  \bm{I}_{N} \otimes \bm{x}_{0} $ yields the Jacobian 
\begin{align}
\bm{J} &=  \bm{I}_{N} \otimes \bm{A} + \lp \bm{I}_{N} \otimes \bm{B} \rp \bm{P} \lp \bm{I}_{N} \otimes \bm{C} \rp, \nonumber \\
 &= \bm{I}_{N} \otimes \bm{A} + \lp \bm{I}_{N} \otimes \bm{B} \rp \lp \sum_{i = 1}^{r} \bm{W}_{i} \otimes \bm{D}_{i}    \rp \lp \bm{I}_{N} \otimes \bm{C} \rp ,\nonumber \\
 &= \bm{I}_{N} \otimes \bm{A} + \sum_{i = 1}^{r} \bm{W}_{i} \otimes \bm{B} \bm{D}_{i} \bm{C}
\end{align}
by direct substitution of the definition of $\bm{P}$ in terms of the independent signalling mechanisms and the mixed products property of Kronecker products \cite{hardy2019matrix}.  As $\bm{W}_{i}\bm{W}_{j} = \bm{W}_{j} \bm{W}_{i}$ for all $\bm{W}_{i},\, \bm{W}_{j} \in \bm{\cW}$ and all matrices $\bm{W}_{i}$ are real and symmetric,  then there exists a matrix $\bm{R}$ that simultaneously diagonalises all adjacency matrices $\bm{W}_{i}\in \bm{\cW}$ \cite{liesen2011lineare}.  Moreover,  the eigenbasis defined by $\bm{R}$ fixes the order of the diagonal entries in each $ \bm{Z}_{i} = \bm{R}^{-1} \bm{W}_{i} \bm{R} = \text{diag} \lp \lambda_{i,1},...,\lambda_{i,N} \rp$ such that the sum of the diagonalised matrices $\bm{Z}_{i}$ are unique.  Specifically,  reordering the eigenvectors that form the eigenbasis $\bm{R}$ would only permute the sum of the diagonal values of $\bm{Z}_{i}$. \par

Consider the transformed Jacobian $\bm{H} = \lp \bm{R}^{-1} \otimes \bm{I}_{n}   \rp \bm{J}  \lp \bm{R} \otimes \bm{I}_{n}   \rp$ then by the mixed products property of Kronecker products
\begin{align}
\bm{H} &=  \lp \bm{R}^{-1} \otimes \bm{I}_{n}   \rp \lp \bm{I}_{N} \otimes \bm{A} \rp \lp \bm{R} \otimes \bm{I}_{n}   \rp  +   \lp \bm{R}^{-1} \otimes \bm{I}_{n}   \rp   \lp\sum_{i = 1}^{r} \bm{W}_{i} \otimes \bm{B} \bm{D}_{i} \bm{C}   \rp   \lp \bm{R} \otimes \bm{I}_{n}   \rp  ,\nonumber \\
&= \bm{R}^{-1} \bm{I}_{N} \bm{R} \otimes \bm{I}_{n} \bm{A} \bm{I}_{n} + \sum_{i = 1}^{r}  \bm{R}^{-1} \bm{W}_{i} \bm{R} \otimes \bm{I}_{n} \bm{B} \bm{D}_{i} \bm{C} \bm{I}_{n}, \nonumber   \\
&= \bm{I}_{N} \otimes \bm{A}  + \sum_{i = 1}^{r}  \bm{Z}_{i} \otimes  \bm{B} \bm{D}_{i} \bm{C}.
\end{align}
By the diagonal structure of $\bm{Z}_{i}$ the matrix $\bm{H}$ has the block diagonal form 
 \begin{equation}
 \bm{H}= \begin{bmatrix}
           \bm{A} +\sum_{i=1}^{r} \lambda_{i,1} \bm{B} \bm{D}_{i} \bm{C} &  &   \\
           & \ddots &  \\
           & &  \bm{A} + \sum_{i=1}^{r} \lambda_{i,n}  \bm{B} \bm{D}_{i} \bm{C}
         \end{bmatrix},
\end{equation}
and therefore,  as $\sum_{i=1}^{r} \lambda_{i,j} \bm{B} \bm{D}_{i} \bm{C}=  \bm{B} \bm{\Lambda}_{j} \bm{C}$ then the eigenvalues of $\bm{H}$ are those of $\bm{A} + \bm{B} \bm{\Lambda}_{j} \bm{C}$ for all $1 \leq j \leq N$.   Consequently,  if $\bm{A} + \bm{B} \bm{\Lambda}_{j} \bm{C}$ has eigenvalues with all negative real-part,  for all $1 \leq j \leq N$,  then $\bm{H}$ is stable and therefore the stability of $\bm{J}$ follows by the bijection between the linearised systems $\bm{H}$ and $\bm{J}$.
\end{proof}

Before discussing the behaviour of flows of the IO system (\ref{eqn:IO_system}) near the HSS,  we first introduce a convenient condition for the instability of a matrix.
\begin{lemma}[\cite{RufinoFerreira2013}]\label{lemma:suff_stable}
If $\bm{M} \in \R^{n \times n}$ is stable then $\lp -1 \rp^{n} \det \lp \bm{M}  \rp >0  $.  Conversely,  if $\lp -1 \rp^{n} \det \lp \bm{M}  \rp < 0$ then $\bm{M}$ has an eigenvalue with positive real-part.
\end{lemma}

Invoking lemmas \ref{lemma:large_linear} and \ref{lemma:suff_stable} leads to the following sufficient condition for the instability of the HSS associated with an IO system (\ref{eqn:IO_system}) with commuting connectivity graphs $\cG_{j}$.

\begin{theorem}\label{thm: HSS instability det}
Consider the large-scale IO system (\ref{eqn:IO_system}) that is spatially coupled via the global interconnection matrix $\bm{P}$ (\ref{eqn:interwoven_adj_mat}) such that each $\bm{W}_{j} \in \bm{\cW}$ commute.  Denote $\cD \bm{T} \coloneqq \pd \bm{T} / \pd \bm{u}_{i}$ and let $\bm{\Lambda}_{j} = \text{diag} \lp \lambda_{1,j},...,\lambda_{r,j}  \rp$ where $\lambda_{i,j}$ is the $j$th eigenvalue of $\bm{W}_{i}$ w.r.t.  the common eigenbasis of all matrices in $\bm{\cW}$. Then the HSS $\bm{x}^{*} = \bm{1}_{N} \otimes \bm{x}_{0}$ is unstable if there exists a $\bm{\Lambda}_{j}$ such that 
\begin{equation} \label{eqn:Hss_instab_cond}
\prod_{i = 1}^{r} \lp 1  - \mu_{i,j}  \rp < 0,
\end{equation}
where $\mu_{i,j}$ are the eigenvalues of $\bm{\Lambda}_{j} \cD \bm{T} \lp \bm{u}_{0} \rp$ and $\bm{u}_{0}$ is the steady state input vector associated with $\bm{x}_{0}$.
\end{theorem}

\begin{proof}
By Lemma \ref{lemma:large_linear} we only need to show that there exists a positive eigenvalue of $\bm{A} +  \bm{B} \bm{\Lambda}_{j} \bm{C}$ for $\bm{\Lambda}_{j}$ some diagonal matrix of eigenvalues of all matrices $\bm{W}_{i} \in \bm{\cW}$ to demonstrate the instability of the HSS.  Consider $ \lp -1 \rp ^{n} \det \lp \bm{A} +  \bm{B} \bm{\Lambda}_{j} \bm{C} \rp $,  \revised{then by Sylvester's Determinant Theorem \cite{zhang2011matrix} we have that,
\begin{align}
 \lp -1 \rp ^{n} \det \lp \bm{A} +  \bm{B} \bm{\Lambda}_{j} \bm{C} \rp  &=  \lp -1 \rp ^{n}  \det\lp \bm{A} \rp \det \lp  \bm{I}_{r} +  \bm{\Lambda}_{j} \bm{C} \bm{A}^{-1} \bm{B}   \rp ,  \nonumber \\
 &= \lp -1 \rp ^{n}  \det\lp \bm{A} \rp  \det \lp  \bm{I}_{r} - \bm{\Lambda}_{j} \cD \bm{T} \lp \bm{u}_{0} \rp   \rp ,
\end{align}
where the final equality holds from $ \cD \bm{T} \lp \bm{u}_{0} \rp = - \bm{C} \bm{A}^{-1} \bm{B} $ as derived in \cite{Arcak2013}.  As $\bm{A}$ is stable by assumption we have that $\bm{A}^{-1}$ exists and $\lp -1 \rp ^{n} \det \lp  \bm{A} \rp >0$ by Lemma \ref{lemma:suff_stable}.  Therefore if $\det \lp \bm{1}_{r} - \bm{\Lambda}_{j} \cD \bm{T} \lp \bm{u}_{0} \rp \rp < 0$ then $\bm{x}^{*} =  \bm{1}_{N} \otimes \bm{x}_{0} $ is unstable,  by the converse statement of Lemma \ref{lemma:suff_stable}. }  \revised{Hence as the determinant of a matrix is the product of the eigenvalues \cite{liesen2011lineare},  we have that}
\begin{equation}\label{eqn: char_poly_set}
\det \lp \bm{I}_{r} - \bm{\Lambda}_{j} \cD \bm{T} \lp \bm{u}_{0} \rp \rp = \prod_{i = 1}^{r} \lp 1 - \mu_{i,j} \rp
\end{equation}
for all matrices $\bm{\Lambda}_{j}$ $\lp 1 \leq j \leq N \rp$.
\end{proof}
Applying the HSS instability condition derived in Theorem \ref{thm: HSS instability det} IO systems with one,  or two,  spatially dependent components,  known as single-input-single-output (SISO) and double-input-double-output (DIDO) interconnected systems produces simple forms of the instability condition (\ref{eqn:Hss_instab_cond}).  Explicitly,  the IO system (\ref{eqn:IO_system}) is SISO when $r = 1$ and DIDO when $r = 2$.  Let $\Sp \lp \bm{M} \rp$ denote the set of eigenvalues of $\bm{M}$ then,  critically,  we recover the SISO instability condition initially derived in \cite{Arcak2013} where we allow for generic intracellular kinetics here.  

\begin{corollary}\label{cor:instab_cond}
Consider the large-scale IO system (\ref{eqn:IO_system}) and denote $\cD \bm{T} \coloneqq \pd \bm{T} / \pd \bm{u}_{i}$.  Then if:
\begin{enumerate}[label=(\roman*)]
\item The IO system (\ref{eqn:IO_system}) is SISO with connectivity matrix $\bm{W}_{1}$ then the HSS $x^{*} = \bm{1}_{N} \otimes \bm{x}$ is unstable if 
\begin{equation}\label{SISO bif}
1< \lambda_{1,j}T^{\prime} \lp u_{0} \rp
\end{equation}
for some $\lambda_{1,j} \in \Sp \lp \bm{W}_{1}\rp $.
\item The IO system (\ref{eqn:IO_system}) is DIDO with global interconnection matrix $\bm{P}$ constructed by the commutative adjacency matrices $\bm{W}_{1}$ and $\bm{W}_{2}$ then the HSS $\bm{x}^{*} = \bm{1}_{N} \otimes \bm{x}$ is unstable if 
\begin{equation}\label{DIDO bif}
1 <  \Tr \lp \bm{\Lambda}_{j} \cD \bm{T} \lp \bm{u}_{0} \rp \rp   -  \det \lp \bm{\Lambda}_{j} \cD \bm{T} \lp \bm{u}_{0} \rp \rp
\end{equation}
for some $\bm{\Lambda}_{j} = \text{diag} \lp \lambda_{1,j},  \lambda_{2,j}  \rp$,  where $\lambda_{1,j} \in \Sp \lp \bm{W}_{1} \rp $ and $\lambda_{2,j} \in \Sp \lp \bm{W}_{2} \rp $ both associated with the same eigenvector.
\end{enumerate}
\end{corollary}

\begin{proof}
In the case of a SISO system when $r = 1$,  the $T: U \rightarrow V$ is a scalar function and we have that inequality (\ref{eqn:Hss_instab_cond}) simply becomes $1 -  \lambda_{1,j} T^{\prime} \lp u_{0} \rp <0 $ yielding the SISO condition (\ref{SISO bif}).  For a DIDO system where $r = 2$,  there are two potentially different adjacency matrices $\bm{W}_{1}$ and $\bm{W}_{2}$ that form $\bm{P}$.  Therefore from inequality (\ref{eqn:Hss_instab_cond}) we have that 
\begin{equation}\label{eqn: dido_ineq}
0 > \lp 1 - \mu_{1} \rp \lp 1 - \mu_{2} \rp = 1 + \Tr \lp -\bm{\Lambda}_{j} \cD \bm{T} \lp \bm{u}_{0} \rp \rp + \det \lp   -\bm{\Lambda}_{j} \cD \bm{T} \lp \bm{u}_{0} \rp \rp= 1 - \Tr \lp \bm{\Lambda}_{j} \cD \bm{T} \lp \bm{u}_{0} \rp \rp + \det \lp   \bm{\Lambda}_{j} \cD \bm{T} \lp \bm{u}_{0} \rp \rp 
\end{equation}
using the relations between determinant, trace and the eigenvalues of a matrix \cite{liesen2011lineare}.  Rearrangement of inequality (\ref{eqn: dido_ineq}) yields the DIDO HSS instability condition (\ref{DIDO bif}).
\end{proof}

The HSS instability conditions outlined in Theorem \ref{thm: HSS instability det} allow the study polarity regimes via graph edge weights to induce heterogeneity of cellular states within the bilayer tissues using analytic methods.  Critically,  the sufficient patterning conditions of Theorem \ref{thm: HSS instability det}  are independent of the precise intracellular kinetics are we do not impose any specific feature on the transfer function,  $\bm{T} \lp  \cdot \rp$,  other than the mild requirement of boundedness that follows immediately when modelling protein dynamics.
 
In the following section,  we introduce methods of graph partitioning for templating laminar patterns in bilayer geometries that produce analytic conditions for the existence of the laminar patterns with multiple signalling mechanisms.  In particular, we show that the commutative properties of the adjacency matrices $\bm{W}_{j} \in \bm{\cW}$ required for the HSS instability condition in Theorem \ref{thm: HSS instability det} can be relaxed when seeking dichotomous cell states in bilayer structures with same layer connectivity symmetries,  namely semi-regular bilayer graphs.  In addition,  by restricting the characteristic behaviour of intracellular kinetics to competitive interactions,  we ensure that the HSS instability converges to laminar patterns by applying results from monotone dynamical system theory.

\section{Laminar pattern convergence with monotone kinetics in semi-regular bilayer graphs}\label{sec:laminar_converge}

The instability of the HSS of the IO system (\ref{eqn:IO_system}) does not imply the existence of stable heterogeneous cell states,  even in systems with a unique HSS and bounded dynamics as there may exist oscillatory or chaotic solution trajectories.  We leverage results from monotone dynamics systems and techniques of graph symmetry reduction to ensure the convergence to dichotomous cell states at the instance of HSS instability in the bilayer geometries.   These methods of discrete pattern analysis were first introduced for SISO systems in \cite{RufinoFerreira2013} and later briefly extended to MIMO systems in \cite{gyorgy2017patternMulti}.  Here we demonstrate the applicability of these methods to two-state pattern formation with pathway crosstalk kinetics in bilayer geometries.  In addition,  we emphasise the link to the corresponding large-scale IO system (\ref{eqn:IO_system}),  namely,  when are the predicted patterns in the symmetry reduced system preserved in the large-scale system.  

\subsection{Monotone kinetics for pattern convergence} \label{sec:monotone_intro}$\,$\\
Let $\bm{\phi}_{t} \lp \bm{x}_{1} \rp $ and $\bm{\phi}_{t} \lp \bm{x}_{2} \rp $ be two solutions to the IO system (\ref{eqn:IO_system}) where $\bm{x}_{1} \leq \bm{x}_{2}$ are initial conditions.  It is said that the dynamical system (\ref{eqn:IO_system}) is monotone if $\bm{\phi}_{t} \lp \bm{x}_{1} \rp \leq \bm{\phi}_{t} \lp \bm{x}_{2} \rp$ for all $t \in [0,  \infty )$ \cite{smith2008monotone}.  Furthermore,  the IO system (\ref{eqn:IO_system}) is said to be strongly monotone if $\bm{\phi}_{t} \lp \bm{x}_{1} \rp < \bm{\phi}_{t} \lp \bm{x}_{2} \rp$ for all $t \in [0,  \infty )$ \cite{smith2008monotone}.  Critically,  the property of strong monotonicity is crucial for the asymptotic convergence of solutions $\bm{\phi}_{t} \lp \bm{x} \rp $ on bounded domains $X \subset \R_{\geq 0}^{n}$,  analogous to the Monotone Convergence Theorem for bounded sequences \cite{howie2006real}.  \par

\revised{A dynamical system can be shown to be monotone by studying the sign structure of the associated Jacobian matrix on convex domains.  The trajectory domain $X$ is convex if for any $\bm{a},\bm{b} \in X$ then $t\bm{a}+ (1-t)\bm{b} \in X$ for all $t \in [0,1]$,  i.e.,  there exists a line segment between any two points in the domain that lies in the interior of $X$.  Note that is does not restrict the solutions of the IO system (\ref{eqn:IO_system}) as $\R_{\geq 0}^{n}$ is a convex set and $X \subset \R_{\geq 0}^{n}$.  The monotone identification via the Jacobian matrix relies on the inter-component monotonicity of vector-valued functions,  initially studied by Kamke \cite{kamke1932theorie},  leading to the classification of type K functions.  Namely,  a function $\bm{g} \lp \cdot \rp$ is said to be type K if for each $i$,  $\bm{g}_{i}\lp \bm{a} \rp \leq \bm{g}_{i} \lp  \bm{b}\rp$ for any two points $\bm{a},\bm{b} \in X$ satisfying $\bm{a} \leq \bm{b}$ and $\bm{a}_{i} = \bm{b}_{i}$ where $X$ is a convex domain \cite{smith2008monotone}.  The identification of type K functions in dynamical systems leads to the sufficient condition for monotone trajectories.  }

\begin{lemma}[Type K monotone systems \cite{smith2008monotone}]\label{lemma:type_k}
Consider the general autonomous dynamical system
 \begin{equation}\label{eqn:general_ode}
\dot{\bm{z}} = \bm{g} \lp \bm{z} \rp,
\end{equation} where $\bm{z} \in Z$ and $Z \subset \R^{n}$ is convex.  Then the dynamical system (\ref{eqn:general_ode}) is monotone if it is type K.  Furthermore,  by the Fundamental Theorem of Calculus,  the general autonomous dynamical system (\ref{eqn:general_ode}) is guaranteed to be type K when the row-sums of the associated Jacobian satisfy
\begin{equation}
 \sum_{j \neq i} \frac{ \pd \bm{g}_{i}}{ \pd \bm{z}_{j}}  \geq 0 
\end{equation}
for all $1 \leq i \leq n$.
\end{lemma}

A direct consequence of Lemma \ref{lemma:type_k} is that the IO system (\ref{eqn:IO_system}) is monotone provided that all off-diagonal components of the associated Jacobian are non-negative for all $\bm{x} \in X$ as previously applied in large-scale IO pattern formation studies \cite{gyorgy2017pattern,Arcak2013}.  In addition,  Hirsch \cite{hirsch1983differential} provided a sufficient condition for strong monotonicity that is dependent on the irreducibility of the Jacobian of the dynamical system.  Specifically,  a matrix $\bm{M}$ is said to irreducible if there exists no permutation matrix $\bm{U}$ such that $\bm{U}^{T} \bm{M} \bm{U}$ is in upper block triangular form \cite{zhang2011matrix}.  

\begin{lemma}[\cite{smith2008monotone}]\label{lemma:strongly_monotone}
Consider the dynamical system (\ref{eqn:general_ode}) as in Lemma \ref{lemma:type_k}.  If the Jacobian,  $\frac{\pd \bm{g}}{ \pd \bm{z}}$,  is irreducible and type K for all $z \in Z$ then system (\ref{eqn:general_ode}) is strongly monotone.
\end{lemma}

The combination of lemmas \ref{lemma:type_k} and \ref{lemma:strongly_monotone} yield sufficient conditions for the identification of strongly monotone dynamical systems using standard linearisation methods,  which are particularly applicable to interconnected dynamical systems.  Namely,  connected graphs have irreducible adjacency matrices \cite{godsil2001algebraicBook}.\par

Time-dependent monotone systems are often be characterised into two distinctive classes: cooperative dynamics where all solutions are monotone in forward-time $(t \rightarrow \infty)$,  and competitive dynamics where all solutions are monotone in backward-time $(t \rightarrow - \infty)$ \cite{smith2008monotone}.  It has previously been demonstrated that competitive dynamics lead to pattern generation in large-scale IO systems,  specifically,  when studying processes of mutual cellular inhibition which are a common feature of cell-fate dynamics in developing tissues \cite{Arcak2013}. \revisedTC{ For example,  the lateral-inhibition interactions of Notch1 and Delta1 are often found in tissues with a dichotomy of spatially organised cell-types and conform to the monotone competitive description \cite{Lloyd-Lewis2019,  Collier1996}. } Subsequently,  we focus our attention on competitive intracellular kinetics which leads to the following assumption on the behaviour of the transfer function $\bm{T} \lp \cdot \rp $ to ensure the asymptotic convergence of solutions with tissue heterogeneity.
\vspace{1em}
\begin{enumerate}[label=(A\arabic*)]
\item  The derivative of the transfer function $\cD\bm{T} \lp \bm{u} \rp$ of the IO system (\ref{eqn:IO_system}) has one of the following sign structures
\begin{equation*}
 \cS_{1} =  \begin{bmatrix}
           - & + & \cdots & - & +  \\
           + & - & \cdots & + & -  \\
             &  & \vdots &  &   \\
             - & + & \cdots & - & +  \\
           + & - & \cdots & + & -  \\
           \end{bmatrix} 
            \qquad 
           \cS_{2} = \begin{bmatrix}
           - & - & \cdots & - & -  \\
           - & - & \cdots & - & -  \\
             &  & \vdots &  &   \\
             - & - & \cdots & - & -  \\
           - & - & \cdots & - & -  \\
           \end{bmatrix}
\end{equation*}
for all $\bm{u} \in U$ where any sign can be replaced by zero provided $\cD\bm{T} \lp \bm{u} \rp$ is irreducible.
\end{enumerate}
\vspace{1em}
Critically,  the conditions imposed on the intracellular kinetics by (A1) are not restrictive in the context of cellular pattern formation as activation and repression of intracellular signals are typically modelled using monotonic functions,  such as Hill or logistic functions that relate to Michaelis-Menten kinetics for enzyme-catalyst reactions \cite{klipp2006mathematical}.  Furthermore,  the irreducibility of $\cD\bm{T} \lp \bm{u} \rp$ follows immediately if there exist no zero entries,  that is,  each spatially dependent component is continuously dependent on all other spatially dependent components.  \par

In the following section,  we will use the competitive properties of the transfer function to predict the existence of laminar pattern formation in bilayer geometries graph partitioning.  In particular,  we focus on the analysis of the transfer function,  as this considers only the spatially dependent components of the IO system (\ref{eqn:IO_system}),  which potentially reduces the dimensionality of the analysis while preserving the underlying behaviour of the system.



\subsection{Dimension reduction by graph partition for polarity laminar pattern existence}\label{sec:quotient_pattern}$\,$ \\
Methods of graph partitioning have previously been employed in large-scale IO systems to predict the existence of patterns with a predefined pattern structure \cite{RufinoFerreira2013, gyorgy2017pattern}.  These predefined pattern structures allow for the construction of bespoke systems by exploiting the symmetries of the cellular connectivity graphs, $\cG_{k}$,  thereby analysing only representative vertices from each pattern partition of the large-scale graphs,  vastly reducing the dimensionality associated IO systems.  Under the assumption of monotone transfer kinetics (A1),  we provide sufficient conditions for the existence of polarity-driven laminar patterns in bilayer geometries with multiple spatially dependent components using graph partitioning.  Critically,  we demonstrate the prior requirement of commutative connectivity graphs $\cG_{j}$ can be relaxed when seeking patterns with only two contrasting states.\par

The method of pattern templating via graph partitions seeks to group cells that are assumed to have the same steady-state solutions and therefore impose that cells within the same group behave identically.  This assumption allows for the study of two representative cells from each layer in the bilayer large-scale graphs, $\cG_{k}$,  to predict the existence of laminar patterns as shown in Figure \ref{fig:quotient_example}.  Formally,  we are assuming the existence of an equitable partition, $\pi_{2}$,  of the vertices $v_{i} \in V$ into the pattern groups $\cL_{1}$ and $\cL_{2}$ of each layer for all connectivity graphs $\cG_{k}$.  This means that $v \in \cL_{i}$ has the same number of adjacent vertices in both $\cL_{1}$ and $\cL_{2}$,  independent of the vertex,  $v$ \cite{godsil2001algebraic}.  We are imposing that cells within the same layer have the same edge connectivity structure,  and therefore the connectivity graphs $\cG_{k}$ must be layer-wise regular as highlighted in Figure \ref{fig:quotient_example}.  Algebraically,  the partition $\pi_{2}$ is equitable if there exists some $\overline{w}_{ij}^{[k]} > 0$ such that

\begin{equation}\label{eqn:quotient_constants}
\sum_{v \in \cL_{j}} w_{uv}^{[k]} = \overline{w}_{ij}^{[k]} \quad \forall u \in \cL_{i},
\end{equation}
where $\hat{w}_{ij}^{[k]}$ are the $ij$-th elements of the row-stochastic adjacency matrix $\bm{W}_{k} \in \bm{\cW}$ \cite{godsil2001algebraic}.  In addition,  we say that the laminar pattern partition, $\pi_{2}$,  is simultaneously equitable if $\pi_{2}$ is equitable for all graphs $\cG_{k}$.


\begin{figure}[h!]
\includegraphics[width=0.6\textwidth]{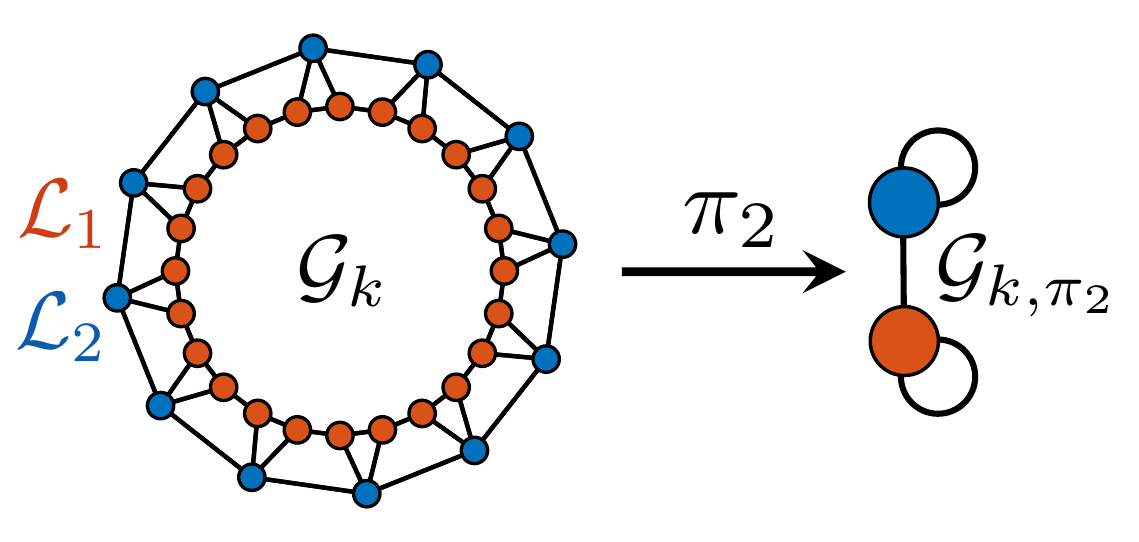}
\caption{Templating for laminar patterns in bilayer geometries using the equitable partition $\pi_{2}$.  The cellular connectivity graph, $\cG_{k}$,  is semi-regular such vertices within the same layer,  $\cL_{1}$ or $\cL_{2}$,  have the same number of adjacent vertices in each of the layers which induce an edge symmetry with respect to vertices in the same layer.  The equitable partition,  $\pi_{2}$,  leverages the edge symmetries of $\cG_{k}$ to generate a quotient graph $\cG_{k,\pi_{2}}$ consisting of two representative cells,  one from each layer $\cL_{1}$ and $\cL_{2}$. } \label{fig:quotient_example}
\end{figure}

Let $\overline{\bm{W}}_{k} \in \R_{\geq 0}^{2 \times 2} $ be the reduced adjacency matrix for the quotient graph $\cG_{k,\pi_{2}} = \cG_{k} / \pi_{2}$ as depicted in Figure \ref{fig:quotient_example},  that are element-wise composed with the constants defined by equation (\ref{eqn:quotient_constants}).  Applying the IO preserving interconnection matrix definition (\ref{eqn:interwoven_adj_mat}) to the set of reduced adjacency matrices,  we have the reduced interconnection matrix of the form,
\begin{equation}\label{eqn:reduced_interweave}
\overline{\bm{P}} = \sum_{i = 1}^{r} \overline{\bm{W}}_{i} \otimes \bm{D}_{i},
\end{equation}
noting that the row-stochastic property of each $\bm{W}_{i} \in \bm{\cW}$ is preserved in the quotient mapping such that each $\overline{\bm{W}}_{i}$ is row-stochastic.  In particular,  as the partition $\pi_{2}$ allocates the vertices $v\in V$ into either of the sets,  $\cL_{1}$ or $\cL_{2}$,  each reduced adjacency matrix is of the form,
\begin{equation}\label{eqn:reduce_adj_mat}
\overline{\bm{W}}_{i} = \begin{bmatrix}
a_{i} & 1 - a_{i} \\
1 - b_{i} & b_{i}
\end{bmatrix}
\end{equation}
for all $1 \leq i \leq r$,  where $a_{i},b_{i} \in (0,1)$ are composed of the polarity weights $w_{1}^{[i]}$ and $w_{2}^{[i]}$.  Explicitly,  $a_{i}$ and $b_{i}$ have the layer-dependent form
\begin{equation} \label{eqn:a_b_def}
a_{i} = \frac{n_{1,\cL_{1}}^{[i]}w_{1}^{[i]}  }{ n_{1,\cL_{1}}^{[i]}w_{1}^{[i]}  + n_{2,\cL_{1}}^{[i]}w_{2}^{[i]}   } \quad \text{and} \quad b_{i} = \frac{n_{1,\cL_{2}}^{[i]}w_{1}^{[i]}  }{ n_{1,\cL_{2}}^{[i]}w_{1}^{[i]}  + n_{2,\cL_{2}}^{[i]}w_{2}^{[i]}   }, 
\end{equation} 
where the superscripts correspond to the spatial connectivity mechanism, $i$,  and $ n_{1,\cL_{j}}^{[i]} \geq 1$ and $ n_{2,\cL_{j}}^{[i]} \geq 1$ are the number of connected vertices in the same and opposing layer,  respectively,  from the perspective of each layer, $j = 1,2$.  For example,  $ n_{1,\cL_{1}}^{[k]} = n_{1,\cL_{2}}^{[k]}  = 2$,  $n_{2,\cL_{1}}^{[k]}  = 1$ and $n_{2,\cL_{2}}^{[k]}  = 2$ for $\cG_{k}$ in Figure \ref{fig:quotient_example}.
\par

A key property of the equitable partition, $\pi_{2}$,  is the preservation of eigenvalues when mapping between the large-scale and quotients graphs,  that is,  $\Sp \lp \overline{\bm{W}}_{i} \rp \subset \Sp \lp \bm{W}_{i} \rp $ \cite{godsil2001algebraic}.  Using this property,  any spatially driven instability of the HSS observed in the quotient system also exists in the associated large-scale system.  \revised{However,  to apply the HSS instability conditions derived in Theorem \ref{thm: HSS instability det} to large-scale connectivity graphs,  we require that all $\overline{\bm{W}}_{i}$ must commute to generate a common eigenbasis for simultaneous diagonalisation.  Commutativity is not preserved in the quotient transformation in general due to the reduced form of equation (\ref{eqn:reduce_adj_mat}).  Although,  the following statement enables the use of the HSS instability conditions independent of the commutative properties of $\overline{\bm{W}}_{i}$ by demonstrating the existence of a common eigenbasis for all reduced adjacency matrices partitioned by $\pi_{2}$,  independent of commutativity.}

\begin{lemma}\label{lemma: sim_upper_triangular}
Let $\overline{\bm{P}} \in \R_{\geq 0 }^{2r \times 2r}$ be the  reduced mixed interconnection matrix\autoref{eqn:reduced_interweave} associated with the equitable partition $\pi_{2}$.  Given any matrix $\widetilde{\bm{M}} \in \R^{r \times r}$ where $\bm{M} = \bm{I}_{2} \otimes \widetilde{\bm{M}}$ the eigenvalues of $\overline{\bm{P}}\bm{M}$ are those of $\widetilde{\bm{M}}$ and $\overline{\bm{\Lambda}}_{2}\widetilde{\bm{M}}$ where $\overline{\bm{\Lambda}}_{2} = \text{diag} \lp a_{1} + b_{1} - 1, ..., a_{r} + b_{r} - 1 \rp$.
\end{lemma}
\begin{proof}
By definition of the family of the reduced adjacency matrices\autoref{eqn:reduce_adj_mat},  Spec$\lp \overline{\bm{W}}_{i} \rp = \{ 1,  a_{i} + b_{i} - 1 \}$,  where all reduced adjacency matrices share the common eigenvector $\bm{v}_{1} = \bm{1}_{2}$,  associated with the common eigenvalue $\overline{\lambda}_{i,1} = 1$.  Without loss of generality,  let $\widetilde{\bm{R}}$ be the transformation matrix for $\overline{\bm{W}}_{1}$ such that $\widetilde{\bm{R}}^{-1}  \overline{\bm{W}}_{1} \widetilde{\bm{R}} $ is in Jordan normal form \cite{jorg2015linear}.  Specifically,  as $\bm{v}_{1}$ must represent a column of $\widetilde{\bm{R}}$ as it is an eigenvector for all $\overline{\bm{W}}_{i} $,  then let $\bm{v}_{1}$ form the first column of $\widetilde{\bm{R}}$ such that $\widetilde{\bm{R}}^{-1}  \overline{\bm{W}}_{1} \widetilde{\bm{R}} $ has diagonal entries $1$ and $a_{1}+b_{1} - 1$,  respectively.  Moreover, as each $\overline{\bm{W}}_{i} \in \R_{\geq0}^{2 \times 2}$,  then  $\widetilde{\bm{R}}^{-1}  \overline{\bm{W}}_{i} \widetilde{\bm{R}} $ must be upper triangular form as 1 is a common eigenvalue for all $1 \leq i \leq r$,  that is, $\widetilde{\bm{R}}$ simultaneously upper triangularises the family of reduced adjacency matrices such that each $\widetilde{\bm{R}}^{-1}  \overline{\bm{W}}_{i} \widetilde{\bm{R}} $ has diagonal entries 1 and $ a_{i} + b_{i} - 1 $.\par

Consider the invertible transformation $\bm{R} = \widetilde{\bm{R}} \otimes \bm{I}_{r}$.  Denote the adjacency triangulation transformation of $\overline{\bm{P}}\bm{M}$ by $\bm{H} = \bm{R}^{-1} \overline{\bm{P}}\bm{M} \bm{R} $. Therefore,  we have that 
\begin{align}
\bm{H} &= \lp \widetilde{\bm{R}}^{-1} \otimes \bm{I}_{r}\rp \overline{\bm{P}}\bm{M} \lp \widetilde{\bm{R}} \otimes \bm{I}_{r}\rp,  \nonumber \\
&=  \lp \widetilde{\bm{R}}^{-1} \otimes \bm{I}_{r}\rp         \lp \sum_{i=1}^{r} \overline {\bm{W}}_{i} \otimes \bm{D}_{i} \rp  \lp  \bm{I}_{2} \otimes \widetilde{\bm{M}}  \rp         \lp \widetilde{\bm{R}} \otimes \bm{I}_{r}\rp,  \nonumber \\
&=  \sum_{i=1}^{r}   \widetilde{\bm{R}}^{-1}  \overline {\bm{W}}_{i}   \widetilde{\bm{R}} \otimes \bm{D}_{i} \widetilde{\bm{M}}.
\end{align}
Specifically,  $\bm{H}$ is of block upper triangle form such that
\begin{equation}
\bm{H} = \begin{bmatrix}
\bm{I}_{r} \widetilde{\bm{M}} & \bm{Z} \widetilde{\bm{M}} \\
\bm{0} & \overline{\bm{\Lambda}}_{2}  \widetilde{\bm{M}},
\end{bmatrix},
\end{equation}
where $\bm{Z}$ is some real $r \times r$ matrix constructed by interweaving the upper right entries of the transformed reduced adjacency matrices.  Thus the eigenvalues of $\bm{H}$ are those of $\widetilde{\bm{M}}$ and $\overline{\bm{\Lambda}}_{2}  \widetilde{\bm{M}}$,  and therefore are the eigenvalues of $\overline{\bm{P}}\bm{M}$ via bijective transformation defined by $\bm{R}$.
\end{proof}

Subsequently,  by seeking the existence of laminar patterns using the partition $\pi_{2}$,  Lemma \ref{lemma: sim_upper_triangular} enables an analytic approach to determine the spatially driven instability of the HSS with any combination of layer-wise semi-regular bilayer graphs.  Specifically,  we need only determine the eigenvalues of $\cD \bm{T} \lp \bm{u}^{*} \rp$ to ensure the HSS instability condition (\ref{eqn:Hss_instab_cond}) is satisfied.  \par

By applying the strongly monotone properties of the transfer kinetics outlined in Section \ref{sec:monotone_intro},  we seek to ensure the asymptotic convergence of heterogeneous solutions in the instance of HSS instability.  However,  it can be shown (see Lemma \ref{lemma:permutation} in Appendix \ref{Appendix:Interweave}) that the interconnection matrix, ${\bm{P}}$,  and consequently the reduced interconnection matrix $\overline{\bm{P}}$ is reducible,  and therefore unable to conform to the strongly monotone criteria in Lemma \ref{lemma:strongly_monotone}.  However,  we recover the irreducibility of $\bm{P}$ and $\overline{\bm{P}}$ by multiplication with a suitable class of matrices.

\begin{lemma}\label{lemma:irreducible}
Let $\bm{P}$ be the mixed interconnection matrix (\ref{eqn:interwoven_adj_mat}) and $\bm{Q} = \text{diag} \lp \bm{Q}_{1}, ..., \bm{Q}_{N} \rp$ such that $\bm{Q}_{k} \in \R^{r \times r}$ is irreducible for each $k \in \{  1,  ..., N\}$.  Then $\bm{PQ}$ is irreducible. 
\end{lemma}
\begin{proof}
A graph is said to be strongly connected if there exists a path between any two vertices.  We aim to show that the graph defined by the weighted adjacency matrix $\bm{PQ}$ is strongly connected and therefore use the property that a graph is strongly connected if and only if the associated adjacency matrix is irreducible \cite{brualdi1991combinatorial}.\par

For an unweighted,  nonnegative adjacency matrix $\bm{M}$,  it can be shown that the $(i,j)$th element of $\bm{M}^{k}$ represents the number of ways to travel from vertex $v_{i}$ to vertex $v_{j}$ along exactly $k$ edges.  Therefore if $\bm{M}$ defines a connected graph of $l$ vertices,  then $\bm{M}^{l}$ contains no zero entries for all $(i,j)$,  that is,  there exists a path between any two vertices in less than,  or equal,  to $l$ steps \cite{brualdi1991combinatorial}.  The converse statement is also true.  In the case of weighted,  nonnegative adjacency matrices,  the elements $(i,j)$ of $\bm{M}^{k}$ no longer represent the number of ways to get from vertex $i$ to vertex $j$ along exactly $k$ edges, but nevertheless are non-zero if there exists a path between $v_{i}$ to vertex $v_{j}$ along $k$,  or less,  edges.\par

The set of vertices has cardinality $| V_{\bm{PQ}}| = rN$ owing to the total number of interconnections within the large-scale IO system (\ref{eqn:IO_system}).  Hence consider the adjacency matrix $\lp \bm{PQ} \rp^{rN}$.  From \autoref{appendB:expo_lemma} it can be shown that
\begin{equation}
\bm{P}^{rN} = \lp \sum_{i = 1}^{r} \bm{W}_{i} \otimes \bm{D}_{i}  \rp^{rN} = \sum_{i = 1}^{r} \bm{W}_{i}^{rN} \otimes \bm{D}_{i}  , 
\end{equation}
and by the above argument $\bm{W}_{i}^{rN}$ has no zero elements as each $\bm{W}_{i}$ represents a connected graph of $N$ vertices. Therefore,  $\bm{P}^{rN}$ is the interweave of $r$ completely non-zero matrices and thus w.l.o.g.  for any non-zero elements $p_{i,j}$ of $\bm{P}^{rN}$ then $p_{i\pm r,j}$ and $p_{i,j \pm r}$ are also non-zero.  Specifically there exist no two non-zero elements in $\bm{P}^{rN}$ that are more than $r$ elements apart in each row and column,  as in Example \ref{example:interweave_mats} where $r=2, \,  N = 2$.  In addition,  define $\bm{Q}_{k}^{rN} \coloneqq \tilde{\bm{Q}}_{k} $.  By assumption,  $\tilde{\bm{Q}}_{k} $ has no zero entries for all $i,j\in \{ 1,...,r \}$ by irreducibility and so $\bm{Q}^{rN} = \text{diag} \lp\tilde{\bm{Q}}_{1},...,\tilde{\bm{Q}}_{N}\rp$.  Applying the definition of the matrix product,  the elements of $\lp \bm{PQ} \rp^{rN}$ are given by
\begin{equation}
\lp \bm{P}^{rN} \bm{Q}^{rN} \rp_{ij} = \sum_{k=1}^{rN} p_{i,k} \tilde{q}_{k,j} \neq 0
\end{equation}
for all $i,j \in \{1,...,rN \}$,  as every column of $\bm{Q}^{rN}$ contains $r$ consecutive non-zero elements.  Therefore $\lp \bm{PQ} \rp^{rN}$ is a non-zero matrix which implies that the graph of $\bm{PQ}$ is strongly connected,  thus $\bm{PQ}$ is irreducible.
\end{proof}

The statement of Lemma \ref{lemma:irreducible} applies also to the reduced interconnection matrix $\overline{\bm{P}}$ as it has identical structure to the corresponding large-scale interconnection matrix $\bm{P}$ and therefore the irreducibility of the product is preserved under the quotient mapping by $\pi_{2}$.  Hence by ensuring the irreducibility of the Jacobian of the reduced IO system (\ref{eqn:IO_system}) spatially coupled by $\overline{\bm{P}}$,  then by Lemma \ref{lemma:irreducible} and (A1),  the following statement provides polarity-dependent conditions that guarantee the existence of laminar patterns in semi-regular bilayer graphs by using the strongly monotone dynamics of solution trajectories.

\begin{theorem}[Existence of laminar patterns with semi-regular graphs] \label{thm: existence_of_laminar}
Consider the IO system (\ref{eqn:IO_system}) with interconnection matrix $\bm{P}$ (\ref{eqn:interwoven_adj_mat}).  Let $\pi_{2}$ be the layer-wise simultaneously equitable partition for all bilayer connectivity graphs, $\cG_{k}$,  defined by $\bm{P}$ such that the associated reduced interconnection matrix $\overline{\bm{P}}$ (\ref{eqn:reduced_interweave}) defines the reduced IO system of representative cells from each layer.  Assuming that (A1) is satisfied and there exists $\overline{\bm{\Lambda}}_{2}$ such that the HSS instability condition (\ref{eqn:Hss_instab_cond}) holds for all $n_{1,\cL_{i}}^{[k]}w_{1}^{[k]} \leq  n_{2,\cL_{i}}^{[k]}  w_{2}^{[k]}$  ($i \in \{1,2 \} , k \in \{1,...,r \}$),  then any solutions in the neighbourhood of the HSS,  $\bm{x}^{*}$,  converge to laminar patterns in the reduced system.
\end{theorem}

\begin{proof}
Following from Lemma \ref{eqn:lemma_aux_eqn} we consider the auxiliary dynamic system defined by the transfer kinetics for the reduced IO system

\begin{equation}\label{eqn:aux_reduced}
\begin{bmatrix}
\dot{\bm{z}}_{1} \\
\dot{\bm{z}}_{2} \\
\end{bmatrix} = 
-\begin{bmatrix}
\bm{z}_{1} \\
\bm{z}_{2} \\
\end{bmatrix} + \overline{\bm{P}} \begin{bmatrix}
\bm{T} \lp \bm{z}_{1} \rp \\
\bm{T} \lp \bm{z}_{2} \rp \\
\end{bmatrix}   \coloneqq \bm{F} \lp \bm{z} \rp,
\end{equation}
as this represents the behaviour of reduced IO system using only the spatially dependent components.  Note that the fixed points of the auxiliary system (\ref{eqn:aux_reduced}) are those of the IO system (\ref{eqn:IO_system}).  Namely,  the auxiliary system (\ref{eqn:aux_reduced}) has HSS $\bm{z}^{*} = \bm{1}_{2} \otimes \bm{u}_{0}$ for the cell-wise input steady state $\bm{u}_{0}$ associated with $\bm{x}^{*}$.  Linearising the auxiliary system about the HSS yields the following Jacobian 
\begin{equation}\label{eqn:jaco_reduced}
\frac{ \pd \bm{F}} { \pd \bm{z}} \lp \bm{z}^{*} \rp = -\bm{I}_{2r} + \overline{\bm{P}} \lp \bm{I}_{2} \otimes \cD \bm{T} \lp \bm{z^{*}} \rp \rp.
\end{equation}
First,  we show that sign structures,  $\cS_{1}$ and $\cS_{2}$ of (A1),  are equivalent up to linear transformation on the Jacobian (\ref{eqn:jaco_reduced}),  thereby ensuring the competitive solution dynamics of the auxiliary system (\ref{eqn:aux_reduced}).  Following that,  we then use a competitive to cooperative bijective transformation to show that the auxiliary system is strongly monotone.  Critically,  the boundedness in combination with strongly monotone kinetics of the transfer function ensures the convergence of heterogeneous solutions in the auxiliary system (\ref{eqn:aux_reduced}) and thus the reduced IO system by Lemma \ref{eqn:lemma_aux_eqn}.  A sketch of the following proof is given in Figure \ref{fig:laminar_proof}. \par

Denote the reflection transformation $\bm{M} = \bm{I}_{2} \otimes \widetilde{\bm{M}}$ where $\widetilde{\bm{M}} = \text{diag} \lp -1,  1  \rp$.  Note that $\widetilde{\bm{M}}^{-1} = \widetilde{\bm{M}} $ and therefore $\bm{M}^{-1} = \bm{M}$.  Introducing the coordinate transformation $\bm{w} = \bm{M} \bm{z}$ which converts between Jacobians with sign structures $\cS_{1}$ and $\cS_{2}$.  Explicitly,  consider the auxiliary system (\ref{eqn:aux_reduced}) with $ \text{Sgn} \lp   \cD \bm{T} \lp \cdot \rp \rp = \cS_{2}$,  then the Jacobian (\ref{eqn:jaco_reduced}) with respect to $\bm{w}$ yields 
\begin{align}\label{eqn:reduced_sign_swap_jac}
\bm{M} \lp   \frac{ \pd \bm{F}} { \pd \bm{z}} \lp \bm{M}\bm{w} \rp    \rp \bm{M} &= \lp \bm{I}_{2} \otimes \widetilde{\bm{M}} \rp \lp  -\bm{I}_{2r} + \overline{\bm{P}} \lp \bm{I}_{2} \otimes \cD \bm{T} \lp \bm{M} \bm{w} \rp \rp  \rp \lp \bm{I}_{2} \otimes \widetilde{\bm{M}} \rp, \nonumber \\
 &= -\bm{I}_{2r} + \lp \bm{I}_{2} \otimes \widetilde{\bm{M}} \rp \lp \sum_{i = 1}^{r} \overline{\bm{W}}_{i} \otimes \bm{D}_{i}  \rp  \lp \bm{I}_{2} \otimes \cD \bm{T} \lp \bm{Mw} \rp \rp  \lp  \bm{I}_{2} \otimes \widetilde{\bm{M}}    \rp ,\nonumber \\
  &= -\bm{I}_{2r} + \lp \sum_{i = 1}^{r} \overline{\bm{W}}_{i} \otimes \bm{D}_{i}  \rp \lp \bm{I}_{2} \otimes \widetilde{\bm{M}} \rp   \lp \bm{I}_{2} \otimes \cD \bm{T} \lp \bm{Mw} \rp \rp  \lp  \bm{I}_{2} \otimes \widetilde{\bm{M}}    \rp ,\nonumber \\
  &=  -\bm{I}_{2r} + \lp \sum_{i = 1}^{r} \overline{\bm{W}}_{i} \otimes \bm{D}_{i}  \rp \lp \bm{I}_{2} \otimes \widetilde{\bm{M}} \cD  \bm{T}\lp \bm{Mw} \rp   \widetilde{\bm{M}}  \rp,
\end{align}
where the third and fourth inequality follows from the commutativity of diagonal matrices and the mixed multiplication property of the Kronecker product.  The transformed Jacobian (\ref{eqn:reduced_sign_swap_jac}) is a non-positive matrix as 
\begin{equation}
 \widetilde{\bm{M}} \cD  \bm{T}\lp \bm{Mw} \rp   \widetilde{\bm{M}} = \begin{bmatrix}
 -1 & 0 \\ 
 0 & 1 
 \end{bmatrix}
 \begin{bmatrix}
 -t_{11} & t_{12} \\
 t_{21} & -t_{22}
 \end{bmatrix}
 \begin{bmatrix}
 -1 & 0 \\ 
 0 & 1 
 \end{bmatrix} =  
-  \begin{bmatrix}
 t_{11} & t_{12} \\
 t_{21} & t_{22}
 \end{bmatrix},
\end{equation}
for $t_{ij} >0$ are the signless elements of $ \cD \bm{T}\lp \bm{Mw} \rp $,  i.e.  $ \widetilde{\bm{M}} \cD  \bm{T}\lp \bm{Mw} \rp   \widetilde{\bm{M}}$ has sign structure $\cS_{2}$.  Therefore,  we continue by considering the transfer function with $ \text{Sgn} \lp   \cD \bm{T} \lp \cdot \rp \rp = \cS_{2}$.  \par

The Jacobian (\ref{eqn:jaco_reduced}) with  $ \text{Sgn} \lp   \cD \bm{T} \lp \cdot \rp \rp = \cS_{2}$ is a non-positive matrix as all element of $\overline{\bm{P}}$ are non-negative.  From Lemma \ref{lemma: sim_upper_triangular} the polarity dependent eigenvalues $\lambda_{i,2}$ of $\overline{\bm{W}}_{i}$ have eigenvectors, $\bm{v}_{i,2}$,  with sign structure $\text{Sgn} \lp \bm{v}_{i,2}  \rp = [-,+]^{T}$.  Therefore,  motivated by polarity-driven patterning and the requirement of the positivity of the dominant instability mode for monotone kinetics \cite{smith2008monotone},   we construct a transformation,  $\bm{R}$,  to ensure that any polarity driven instability satisfies the monotonicity criteria,  that is,  monotone with respect to alternating domains.  Then consider the transformation $\bm{R} =  \widetilde{\bm{R}} \otimes \bm{I}_{r}$ where $\widetilde{\bm{R}} = \text{diag} \lp - 1, 1 \rp$.  Noting again that $\bm{R}^{-1} = \bm{R}$ as  $ \widetilde{\bm{R}}^{-1} =  \widetilde{\bm{R}}$.  By similar calculations as above,  it can be shown that by the coordinate transformation $\bm{w} = \bm{Rz}$ the Jacobian (\ref{eqn:jaco_reduced}) has the form
\begin{equation}\label{eqn:trans_aux_jac_pattern}
\bm{R} \lp  \frac{\pd \bm{F}}{ \pd \bm{z} }  \lp \bm{Rw} \rp \rp \bm{R} = -\bm{I}_{2r} + \sum_{i = 1}^{r}   \widetilde{\bm{R}} \overline{\bm{W}}_{i} \widetilde{\bm{R}} \otimes \bm{D}_{i} \cD \bm{T} \lp \bm{Rw} \rp,
\end{equation}
where the quotient adjacency matrix is transformed to the following form
\begin{equation}
 \widetilde{\bm{R}} \overline{\bm{W}}_{i} \widetilde{\bm{R}}  = \begin{bmatrix}
 a_{i}  & -\lp 1- a_{i} \rp \\ 
 - \lp 1 - b_{i} \rp & b_{i}
 \end{bmatrix}.
\end{equation}
Therefore,  let $\tau \lp i \rp = \lp i-1 \rp \bmod r + 1$ then the row-sum of the transformed auxiliary Jacobian (\ref{eqn:trans_aux_jac_pattern}) can be expressed as
\begin{equation}
\sum_{j \neq i}  \lp \bm{R} \lp  \frac{\pd \bm{F}}{ \pd \bm{z} }  \lp \bm{Rw}  \rp \bm{R} \rp \rp_{ij} = 
\begin{dcases} 
\lp 2a_{\tau \lp i\rp} - 1 \rp \sum_{ j = 1,i \neq j}^{r} \lp \bm{D}_{\tau \lp i\rp} \cD \bm{T} \lp \bm{Rw} \rp\rp  _{ij} &  \text{ }  1 \leq i  \leq r,\\
\lp 2b_{\tau \lp i\rp} - 1 \rp \sum_{ j = 1,i \neq j}^{r} \lp \bm{D}_{\tau \lp i\rp} \cD \bm{T} \lp \bm{Rw} \rp\rp  _{ij} &  \text{ } r+1 \leq  i \leq 2r.
\end{dcases}
\end{equation}
Hence by the assumption $n_{1,\cL_{i}}^{[k]}w_{1}^{[k]} \leq  n_{2,\cL_{i}}^{[k]}  w_{2}^{[k]}$  ($i \in \{1,2 \} , k \in \{1,...,r \}$),  we have that $2a_{k} - 1 \leq 0$ and $2b_{k} - 1 \leq 0$ by direct substitution into equation (\ref{eqn:a_b_def}).  Critically,  as $\cD \bm{T} \lp \bm{Rw} \rp$ is a negative matrix,  we have that 
\begin{equation}
\sum_{j \neq i}  \lp \bm{R} \lp  \frac{\pd \bm{F}}{ \pd \bm{z} }  \lp \bm{Rw}  \rp \bm{R} \rp \rp_{ij} \geq 0 
\end{equation}
for all $i \in \{1,...,2r \}$,  thus satisfying the type K condition in Lemma \ref{lemma:type_k}.  Furthermore,  by Lemma \ref{lemma:irreducible},  the transformed auxiliary Jacobian (\ref{eqn:trans_aux_jac_pattern}) is irreducible and therefore the auxiliary dynamical system (\ref{eqn:aux_reduced}) is strongly monotone (cooperative) with respect to the laminar pattern transformation $\bm{R}$.    \par

The cooperative auxiliary dynamical system (\ref{eqn:aux_reduced}) is monotone with respect to the standard domain $\R_{\geq 0}^{2r}$ and has a positive eigenvector $\bm{v} >0$ associated with the polarity driven instability $\overline{\bm{\Lambda}}_{2}$ of the transformed HSS $\bm{R}\bm{z}^{*}$ by the Perron-Frobenius Theorem \cite{Chang2008}.  Consequently,  for small $\epsilon $,  any solution starting at $\bm{R}\bm{z} = \bm{R} \bm{z}^{*} + \epsilon\bm{v}$ must have positive derivative and increase in the transformed trajectory domain $\R_{\geq 0}^{2r}$ \cite{smith2008monotone}.  Critically,  if the solutions of the cooperative auxiliary dynamical system (\ref{eqn:aux_reduced}) are bounded,  then the strongly monotone property ensures the convergence to another steady state,  $\bm{R} \bm{z}^{**} \neq \bm{R} \bm{z}^{*}$.

The transfer function $\bm{T} \lp \cdot  \rp$  is bounded and so $\exists b >0$ such that $|| \overline{\bm{P}} [\bm{T} \lp \bm{z}_{1} \rp ,  \bm{T} \lp \bm{z}_{2} \rp ]^{T}||_{2} < b $ for all $\bm{z}_{i}$.  Thus,  as the cooperative auxiliary dynamical system (\ref{eqn:aux_reduced}) is monotone with respect to $\R_{\geq 0}^{2r}$,  we have that the sets centred about the HSS $\mathcal{V}_{\pm} = \bm{R} \bm{z}^{*} \pm \lp \R_{\geq 0}^{2r} \cap [0,b]^{2r} \rp$ are forward invariant,  i.e.   $\bm{\phi}_{t} \lp \bm{R}\bm{z} \rp \in \mathcal{V}_{\pm}$ for all $t\in [0, \infty)$.  Therefore all solutions are bounded within a compact domain and thus converge to $\bm{R} \bm{z}^{**} \neq \bm{R} \bm{z}^{*}$ by the Cooperative Irreducible Convergence Theorem (Theorem 4.3.3 in \cite{smith2008monotone}).  Subsequently,  the corresponding non-transformed system (\ref{eqn:aux_reduced}) must have each vertices with solutions in $\mathcal{V}_{+} $ and $\mathcal{V}_{-}$,  respectively,  ensuring contrasting cell-wise solutions.  Finally,  as any steady state solution to the auxiliary dynamical system (\ref{eqn:aux_reduced}) is a steady state of the associated reduced IO system (\ref{eqn:IO_system}),  by Lemma  \ref{eqn:lemma_aux_eqn} the reduced IO system (\ref{eqn:IO_system}) converges to laminar patterns.
\end{proof}

\begin{figure}[h!]
\includegraphics[width=0.8\textwidth]{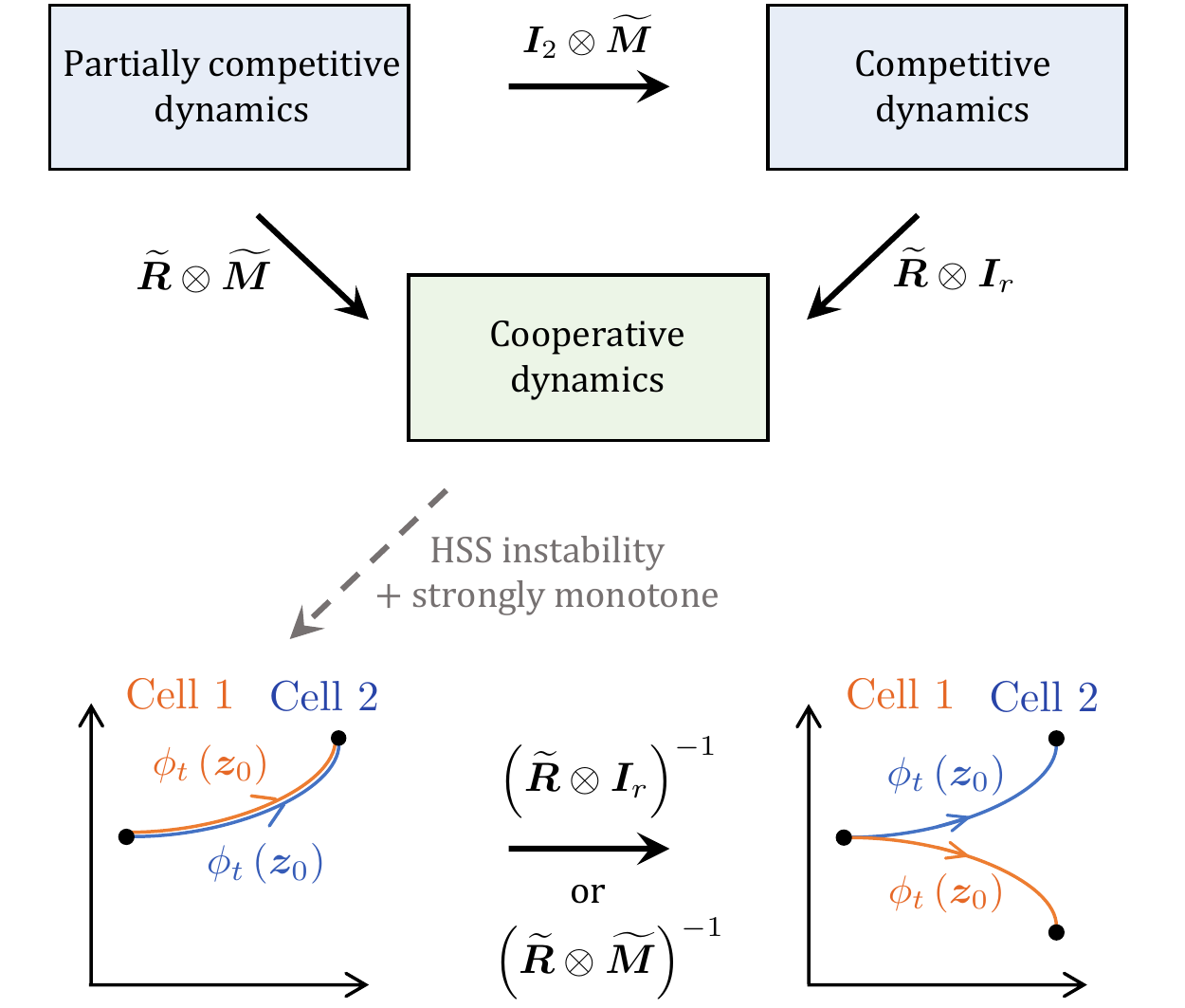}
\caption{A sketch of the proof of Theorem \ref{thm: existence_of_laminar} for system transformations where $\bm{I}_{r}$ is the identity matrix,  $\widetilde{\bm{M}}$ is the competitive sign structure transformation,  $\widetilde{\bm{R}}$ is the competitive to cooperative transformation and  $\bm{\phi}_{t} \lp \bm{z}_{0} \rp$ is a solution trajectory with initial condition $\bm{z}_{0}$. } \label{fig:laminar_proof}
\end{figure}

From Theorem \ref{thm: existence_of_laminar} we can conclude that the existence of a polarity-driven instability of the HSS implies the existence of heterogeneous steady states within the quotient system. This follows as solution trajectories diverge when transforming between competitive to cooperative systems as highlighted in Figure \ref{fig:laminar_proof}.  Moreover,  as the competitive dynamics of the reduced IO system (\ref{eqn:IO_system}) are isomorphic to cooperative dynamics,  all periodic solutions are unstable \cite{smith1986periodic},  implying the convergence to contrasting cell states.  The following example demonstrates how Theorem \ref{thm: existence_of_laminar} can be applied to prove the existence of laminar patterns in large-scale IO systems.

\begin{example}\label{example:laminar_existence}
Consider the DIDO system with two spatially-dependent components describing lateral-inhibition with a diffusive crosstalk as represented in Figure \ref{fig:lamina_ex_1},

\begin{align} \label{example_io_system_1}
\dot{x}_{i,1} &= g_{1} \lp u_{i,1} \rp  \cdot g_{2} \lp x_{i,2} \rp \cdot f_{1} \lp u_{i,2}  \rp - x_{i,1}, \\
\dot{x}_{i,2} &=   f_{2} \lp x_{i,1}  \rp - x_{i,2}, \\
\dot{x}_{i,3} &=   g_{3} \lp x_{i,1}  \rp -x_{i,3},\\ 
y_{i,1} &= x_{i,2},\\
y_{i,2} &= x_{i,3},\label{example_io_system_2}
\end{align}
for each cell $1 \leq i \leq 60$. The functions $f_{j} $ and $g_{j}$,  $j = 1,2,3$,  are positive,  bounded,  and increasing and decreasing functions,  respectively,  of the form,
\begin{equation}
f_{j} \lp x \rp = \frac{x^{k_{j}} }{ \alpha_{j} +x^{k_{j}} } \quad \text{ and } \quad  g_{j} \lp x \rp = \frac{1 }{ 1+  \beta_{j} x^{h_{j}}}
\end{equation}
where $\alpha_{j},\beta_{j},  k_{j},  h_{j} >0$.  Let $u_{i,1}$ and $u_{i,2}$ be defined be short-range diffusion and contact-based bilayer connectivity graphs $\cG_{1}$ and $\cG_{2}$,  respectively as in Figure \ref{fig:lamina_ex_1}.  Explicitly,  we have that outputs are converted to inputs via the global interconnection matrix such that $\bm{u} = \lp \bm{W}_{1}\otimes \bm{D}_{1} + \bm{W}_{2}\otimes \bm{D}_{2} \rp \bm{y}$ for $\bm{W}_{1}, \bm{W}_{2} \in \bm{\cW}$.  Here,  we focus on the associated reduced IO system (\ref{example_io_system_1}-\ref{example_io_system_2}) which is defined by the simultaneously equitable partition $\pi_{2}$.  Namely,  in the reduced IO system,  outputs are converted to inputs by $\bm{u} = \lp \overline{\bm{W}}_{1}\otimes \bm{D}_{1} +\overline{ \bm{W}}_{2}\otimes \bm{D}_{2} \rp \bm{y}$ where 
\begin{equation}
\overline{\bm{W}}_{1} = \begin{bmatrix}
\frac{2w_{1}^{[1]}}{ 2w_{1}^{[1]} + 4w_{2}^{[1]}  }  & \frac{4w_{2}^{[1]}}{ 2w_{1}^{[1]} + 4w_{2}^{[1]}  }  \\
 \frac{4w_{2}^{[1]}}{ 2w_{1}^{[1]} + 4w_{2}^{[1]}  } &\frac{2w_{1}^{[1]}}{ 2w_{1}^{[1]} + 4w_{2}^{[1]}  } 
\end{bmatrix} \quad \text{ and } \quad \overline{\bm{W}}_{2} = \begin{bmatrix}
\frac{2w_{1}^{[2]}}{ 2w_{1}^{[2]} + 2w_{2}^{[2]}  }  & \frac{2w_{2}^{[2]}}{ 2w_{1}^{[1]} +2 w_{2}^{[2]}  }  \\
 \frac{2w_{2}^{[2]}}{ 2w_{1}^{[2]} + 2w_{2}^{[2]}  } &\frac{2w_{1}^{[2]}}{ 2w_{1}^{[2]} + 2w_{2}^{[2]}  } 
\end{bmatrix}
\end{equation}
such that $n_{1,\cL_{1}}^{[1]}  = n_{1,\cL_{2}}^{[1]}  = n_{1,\cL_{1}}^{[2]}  = n_{1,\cL_{2}}^{[2]}  = 2$,  $n_{2,\cL_{1}}^{[1]} = n_{2,\cL_{2}}^{[1]} = 4$
and $n_{2,\cL_{1}}^{[2]} = n_{2,\cL_{2}}^{[2]} = 2$.  We seek to show the existence of polarity driven laminar patterns using the quotient graphs and so we first require the HSS of the IO system (\ref{example_io_system_1}-\ref{example_io_system_2}),  then we derive the derivative of the transfer function $\cD \bm{T} \lp \bm{u}_{i} \rp $,  highlighting that (A1) is satisfied.  Applying Theorem \ref{thm: existence_of_laminar},  we generate polarity regimes for the existence of patterning. \par

\begin{figure}[h!]
\includegraphics[width=0.9\textwidth]{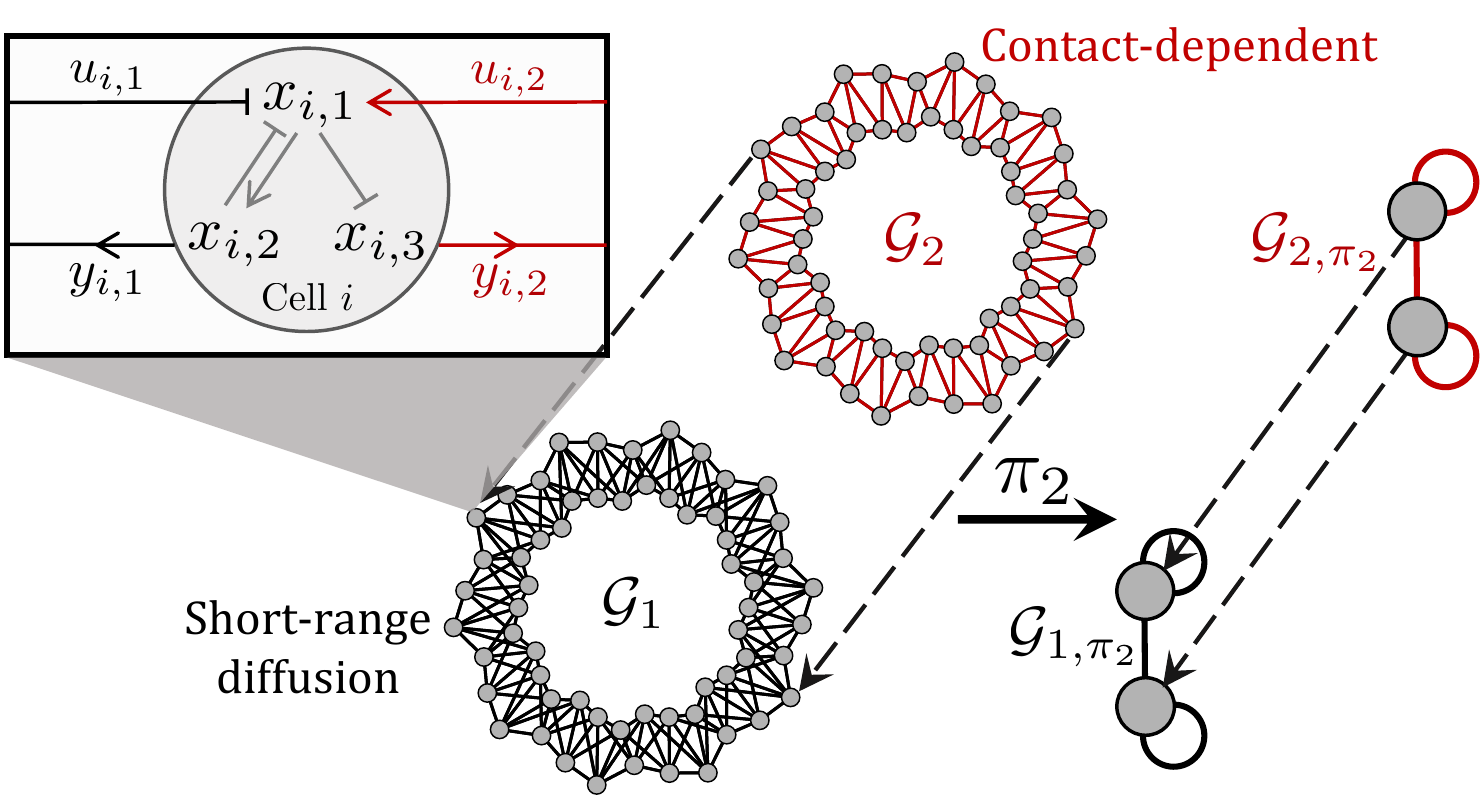}
\caption{A schematic of the IO system considered in Example \ref{example:laminar_existence}.  } \label{fig:lamina_ex_1}
\end{figure}

The HSS of the IO system (\ref{example_io_system_1} - \ref{example_io_system_2}) can be determined by solving
\begin{equation}\label{eqn:example_HSS}
g_{1} \lp f_{2} \lp x_{1}^{*} \rp  \rp \cdot g_{2} \lp f_{2} \lp  x_{1}^{*} \rp \rp \cdot f_{1} \lp g_{3} \lp     x_{1}^{*} \rp \rp - x_{1}^{*} = 0
\end{equation}
for $x_{1}^{*}$ by setting $u_{i,1} = x_{i,2}$ and $u_{i,2} = x_{i,3}$,  conforming to homogeneous input and outputs of the tissue.  Furthermore, the HSS defined by solving equation (\ref{eqn:example_HSS}) is always stable in the absence of interconnections.  This can be shown by considering the linearisation of the intracellular kinetics
\begin{equation}
\bm{A} \coloneqq \frac{  \pd \bm{f} }{ \pd \bm{x}_{i} } = \begin{bmatrix}
- 1  & f_{1} g_{1} g_{2} ^{\prime} & 0 \\
f_{2}^{\prime} & -1 & 0 \\
g_{3}^{\prime} & 0 & -1
\end{bmatrix}.
\end{equation}
As $\det \lp  \bm{A} \rp =f_{1}g_{1} f_{2}^{\prime} g_{2}^{\prime} - 1 < 0$ always holds by the monotonicity of the functions $f_{j}$ and $g_{j}$,  then the HSS defined by solving equation (\ref{eqn:example_HSS}) is unique by Lemma \ref{lemma:exist_and_uni_HSS}.  In addition,  $\bm{A}$ has eigenvalues
\begin{equation}
 \mu_{1} = -1,  \quad  \mu_{2} = -1 + \sqrt{f_{1} g_{1} f_{2}^{\prime} g_{2}^{\prime}}  \quad \text{ and } \quad \mu_{3} = -1 - \sqrt{f_{1} g_{1} f_{2}^{\prime} g_{2}^{\prime}}
\end{equation}
 and so as $ \mu_{1} , \Re \lp \mu_{2} \rp, \Re \lp \mu_{3} \rp<0 $ we have that $\bm{A}$ is stable.  Thus any instability of the HSS will be induced by the interconnection of cells in the tissue. \par

The derivative of the transfer function can be determined by linearisation of the IO kinetics (\ref{example_io_system_1} - \ref{example_io_system_2}) as demonstrated in \cite{Arcak2013} such that $\cD \bm{T}\lp  \bm{u}_{i} \rp = -\bm{C} \bm{A}^{-1} \bm{B}$ where $\bm{B}$ and $\bm{C}$ are the linearised inputs and outputs respectively as in Lemma \ref{lemma:large_linear}.  For the IO system (\ref{example_io_system_1} - \ref{example_io_system_2}),  the derivative of the transfer function has the form

\begin{align}
\cD \bm{T}  \lp  \bm{u}_{i} \rp  &= - \det \lp  \bm{A} \rp^{-1} \begin{bmatrix}
0 & 1 & 0 \\
0 & 0 & 1
\end{bmatrix}
\begin{bmatrix}
1 &  f_{1} g_{1} g_{2}^{\prime} & 0 \\
f_{2}^{\prime} & 1 & 0 \\
g_{3}^{\prime} & f_{1} g_{1} g_{2}^{\prime} g_{3}^{\prime} & 1 - f_{1}g_{1} g_{2}^{\prime} 
\end{bmatrix}
\begin{bmatrix}
f_{1}g_{2}g_{1}^{\prime} & g_{1}g_{2} f_{1}^{\prime} \\
0 & 0\\
0 & 0 \\
\end{bmatrix} \nonumber \\
&=   - \det \lp  \bm{A} \rp^{-1}  \begin{bmatrix}
f_{1} g_{2} f_{2}^{\prime} g_{1}^{\prime}  & g_{1} g_{2} f_{1}^{\prime}  f_{2} ^{\prime}  \\
f_{1} g_{2} g_{1}^{\prime} g_{3} ^{\prime}  &  g_{1} g_{2} f_{1} ^{\prime} g_{3} ^{\prime}
\end{bmatrix},
\end{align}
where each of the functions $f_{j}$ and $g_{j}$ are evaluated using the corresponding arguments for the given input state $\bm{u}_{i}$.  The multiplication of bounded functions are bounded \cite{liesen2011lineare} and subsequently $\cD \bm{T} \lp \bm{u}_{i} \rp$ is element-wise bounded as $f_{j},g_{j},f_{j}^{\prime}$ and $g_{j}^{\prime}$ are bounded.  In addition,  from the monotonicity of $f_{j}$ and $g_{j}$ we have that 
\begin{equation}
\text{Sgn} \lp \cD \bm{T} \lp \bm{u}_{i} \rp \rp = \begin{bmatrix}
- & + \\
+ & -
\end{bmatrix} = \cS_{1}
\end{equation}
and so the IO system (\ref{example_io_system_1} - \ref{example_io_system_2}) satisfies (A1). Therefore by Theorem \ref{thm: existence_of_laminar} we have that the IO system (\ref{example_io_system_1} - \ref{example_io_system_2}) spatially coupled using the quotient graphs $\cG_{1,\pi_{2}}$ and $\cG_{2,\pi_{2}}$,  the instability of the HSS in addition to the monotone polarity conditions $w_{1}^{[1]} \leq 2w_{2}^{[1]} $ and  $w_{1}^{[2]} \leq w_{2}^{[2]} $,  produce contrasting cell-wise states. \par

By Corollary \ref{cor:instab_cond} we apply the DIDO instability inequality (\ref{DIDO bif}) to the IO system (\ref{example_io_system_1} - \ref{example_io_system_2}).  As $\det \lp \cD \bm{T} \lp  \bm{u}_{i} \rp  \rp = 0$,  the DIDO instability inequality (\ref{DIDO bif}) reduces to $1 < \Tr \lp \overline{\bm{\Lambda}}_{2} \cD \bm{T} \lp  \bm{u}_{i} \rp \rp$,  namely the HSS is unstable only if
\begin{equation}\label{eqn:hss_instab_example}
1 < - \det \lp \bm{A} \rp^{-1} \lp  \lp \frac{w_{1}^{[1]}  - 2w_{2}^{[1]}}{  w_{1}^{[1]}  + 2w_{2}^{[1]} }  \rp f_{1} g_{2} g_{1}^{\prime} f_{2}^{\prime}  +   \lp \frac{w_{1}^{[2]}  - w_{2}^{[2]}}{  w_{1}^{[2]}  + w_{2}^{[2]} }  \rp   g_{1} g_{2} f_{1}^{\prime} g_{3}^{\prime} \rp
\end{equation}
for the reduced IO system (\ref{example_io_system_1} - \ref{example_io_system_2}).  The monotone polarity conditions $w_{1}^{[1]} \leq 2w_{2}^{[1]} $ and  $w_{1}^{[2]} \leq w_{2}^{[2]} $ of Theorem \ref{thm: HSS instability det} confirm that each of the reduced connectivity matrices must have negative eigenvalues to produce the instability of the HSS as $f_{1} g_{2} g_{1}^{\prime} f_{2}^{\prime} < 0$ and $ g_{1} g_{2} f_{1}^{\prime} g_{3}^{\prime} < 0$ by the monotone properties of the functions $f_{j}$ and $g_{j}$.  Critically,  the HSS instability inequality (\ref{eqn:hss_instab_example}) highlights that as the layer-wise activator/receptor polarity increases,  i.e.  $w_{1}^{[i]} \ll w_{2}^{[i]}$,  the potential to induce laminar patterns also increases in the quotient system.  Then by the spectral retention property of the equitable partition $\pi_{2}$,   we have that laminar patterns must exist in the pattern space of the associated large-scale system.\par

To illustrate the application Theorem \ref{thm: existence_of_laminar} to the IO system (\ref{example_io_system_1} - \ref{example_io_system_2}) numerical verification of the polarity parameter regime for laminar pattern existence determined by inequality (\ref{eqn:hss_instab_example}) is given in Figure \ref{fig:lamina_ex_2}. 

\begin{figure}[h!]
\includegraphics[width=\textwidth]{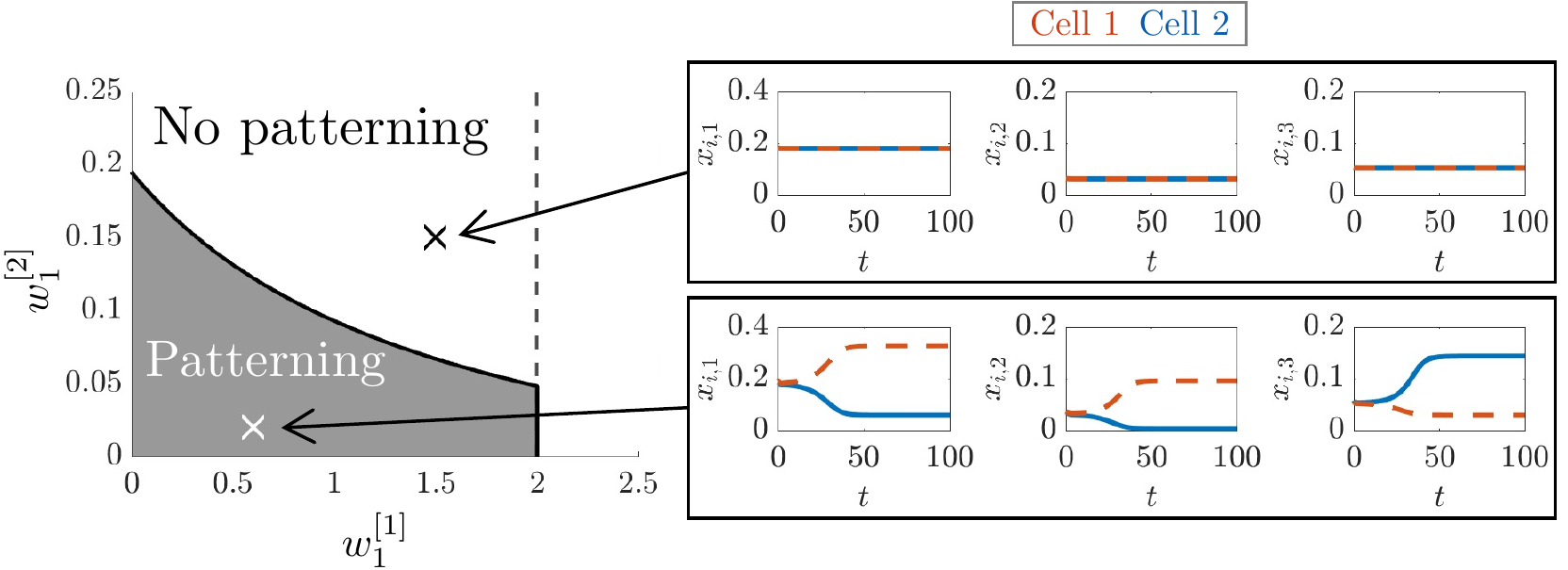}
\caption{Polarity parameter regimes for the existence of laminar patterns in the IO system (\ref{example_io_system_1}-\ref{example_io_system_2}).  For fixed $w_{2}^{[1]} = w_{2}^{[2]} = 1$,   inequality (\ref{eqn:hss_instab_example}) in addition to the monotone polarity conditions $w_{1}^{[1]} \leq w_{2}^{[1]} $ and  $w_{1}^{[2]} \leq 2w_{2}^{[2]} $ to define a regions in $\lp w_{1}^{[1]}, w_{1}^{[2]} \rp$-space for the existence of laminar patterns.  The dashed line in the $\lp w_{1}^{[1]}, w_{1}^{[2]} \rp$-space corresponds to the monotone condition $w_{1}^{[1]} \leq 2w_{2}^{[1]} $.  Example simulations are given for polarity parameter values inside the pattern region,  $\lp 0.6 , 0.02 \rp$,  and outside the pattern region $\lp 1.5 , 0.15 \rp$.  Initial conditions were given as small random perturbations about the HSS,  $\bm{x}^{*} = [0.18, 0.03,  0.05 ]^{T}$.    IO system (\ref{example_io_system_1}-\ref{example_io_system_2}) parameter values and  details on simulations are given in Appendix \ref{sec:ODE_methods}.  } \label{fig:lamina_ex_2}
\end{figure}

\end{example}

As demonstrated in Example \ref{example:laminar_existence},  the method of pattern templating for contrasting solutions between cells in opposing layers can be used to show the existence of layer-wise differing steady states via polarity-driven instabilities.   However,  the associated large-scale systems may have many locally stable steady states that produce the pattern space of the IO system which could have been lost during the dimension reducing transformation by the partition, $\pi_{2}$ \cite{RufinoFerreira2013}.  Therefore,  in the following section, we investigate the spectral properties of the bilayer connectivity graphs to ensure that the laminar patterns produced by Theorem \ref{thm: existence_of_laminar} are indeed globally dominant.




\subsection{Spectral links between quotient and large-scale bilayer connectivity graphs} \label{sec:spec_links} $\,$ \\
For linearised dynamical systems near steady state,  the local solution trajectories are a linear combination of the associated eigenvectors scaled by the corresponding exponent of the eigenvalues \cite{murray2002mathematicalbiology}.  Thus, in the instance of steady-state instability, all trajectories close to the steady-state will locally tend in the direction of the eigenvector associated with the largest real-part eigenvalue.  Critically,  to ensure the monotone convergence of laminar patterns in the reduced IO systems in Theorem \ref{thm: existence_of_laminar},  we transformed the polarity-dependent eigenvector to be directed in the positive orthant,  conforming to the behaviour of cooperative dynamics.  Thus,  motivated by this positive direction transformation,  we seek to understand when the eigenvalue associated with laminar pattern formation dominates the large-scale spectra to ensure perturbed trajectories from the HSS to be preferably pointed in the direction to achieve layer-wise contrasting states in the large-scale IO systems.   \par


Previous studies on pattern formation using IO systems have imposed the sufficient condition that the large-scale and quotient multilayer connectivity graphs $\cG_{k}$ are bipartite,  as this generates monotone dynamics with respect to the bipartition vector \cite{Arcak2013,gyorgy2017patternMulti,gyorgy2017pattern,RufinoFerreira2013}.  Namely,  a graph $\cG_{k}$ is said to be bipartite if the vertices $v \in V$ can be partitioned into two independent sets $V_{1}$ and $V_{2}$ such that no two vertices in the same set are adjacent \cite{godsil2001algebraicBook}.  Example bipartite bilayer graphs are given in Figure \ref{fig:bipartite_example}A.  However,  it can be demonstrated that for bipartite bilayer graphs,  the polarity-dependent eigenvalue,  $\overline{\lambda}_{k,2}$,  associated with laminar pattern formation cannot be dominant.

\begin{lemma}\label{lemma:bipartite_neq_min}
Let $\cG_{k}$ be a bipartite bilayer graph with weighted adjacency matrix $\bm{W}_{k} \in \bm{\cW}$.  Then for any $w_{1}^{[k]},w_{2}^{[k]}>0$ the polarity-dependent eigenvalue $\overline{\lambda}_{k,2}$ associated with the reduced adjacency matrix $\overline{\bm{W}}_{k}$ satisfies
\begin{equation}
\overline{\lambda}_{k,2} \neq \min \lp \Sp \lp \bm{W}_{k} \rp\rp.
\end{equation}
\end{lemma}

\begin{proof}
Consider $\lambda_{k,j} \in \Sp  \lp \bm{W}_{k} \rp $,  then by the spectral symmetry of bipartite graphs about the origin we have that $-\lambda_{k,j} \in \Sp  \lp \bm{W}_{k} \rp $ \cite{godsil2001algebraicBook}.  As $\bm{W}_{k} \in \bm{\cW}$ then $\lambda_{k,1} =  \max \lp \Sp  \lp \bm{W}_{k} \rp \rp  = 1$ by the connected and row-stochastic properties of $\bm{W}_{k}$ \cite{johnson1981row}.  Consequently,  $-\lambda_{k,1} =  \min \lp \Sp \lp \bm{W}_{k} \rp \rp  = -1$.  However,  the minimal eigenvalue of the reduced adjacency matrix $\overline{\bm{W}}_{k}$ defined by the laminar pattern partition, $\pi_{2}$,  must be of the form $\overline{\lambda}_{k,2}  = a_{k} + b_{k} - 1$ for $a_{k},b_{k} \in \lp 0,1 \rp$ by Lemma \ref{lemma: sim_upper_triangular}.  Critically,  this implies that $\overline{\lambda}_{k,2} \in \lp -1,1 \rp$ and therefore  $\overline{\lambda}_{k,2}  \neq \min \lp \Sp \lp \bm{W}_{k} \rp\rp$ for any layer-wise polarity values $w_{1}^{[k]},w_{2}^{[k]}>0$.
\end{proof}

\begin{figure}[h]
\centering
\begin{subfigure}[b]{\textwidth}
         \centering
         \includegraphics[width=0.8\textwidth]{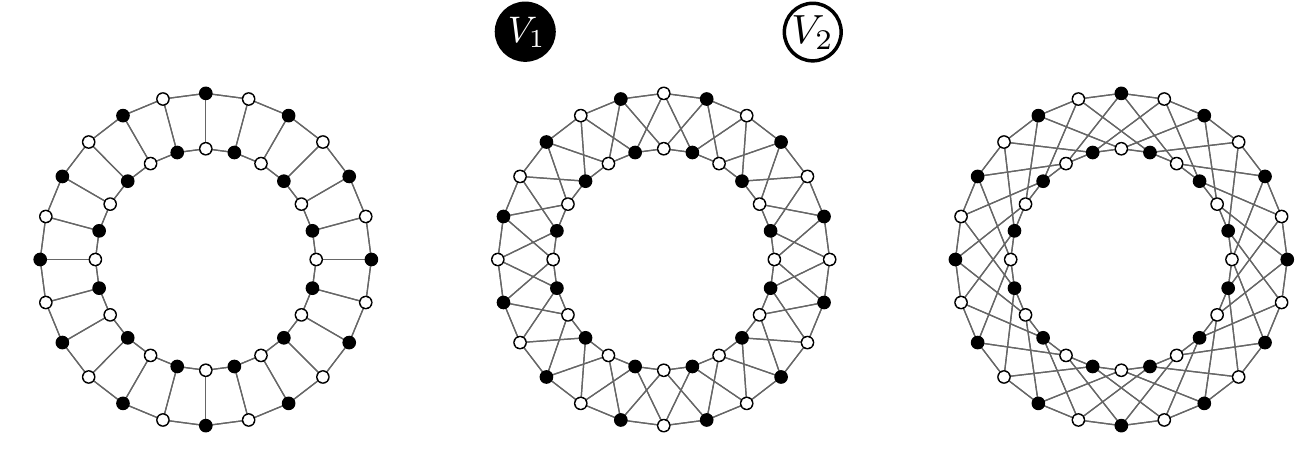}
         \caption{Bipartite bilayer graphs.}
         \label{fig:Bipart_A}
\end{subfigure}

     \begin{subfigure}[b]{\textwidth}
     \vspace{1em}
         \centering
         \includegraphics[width=0.95\textwidth]{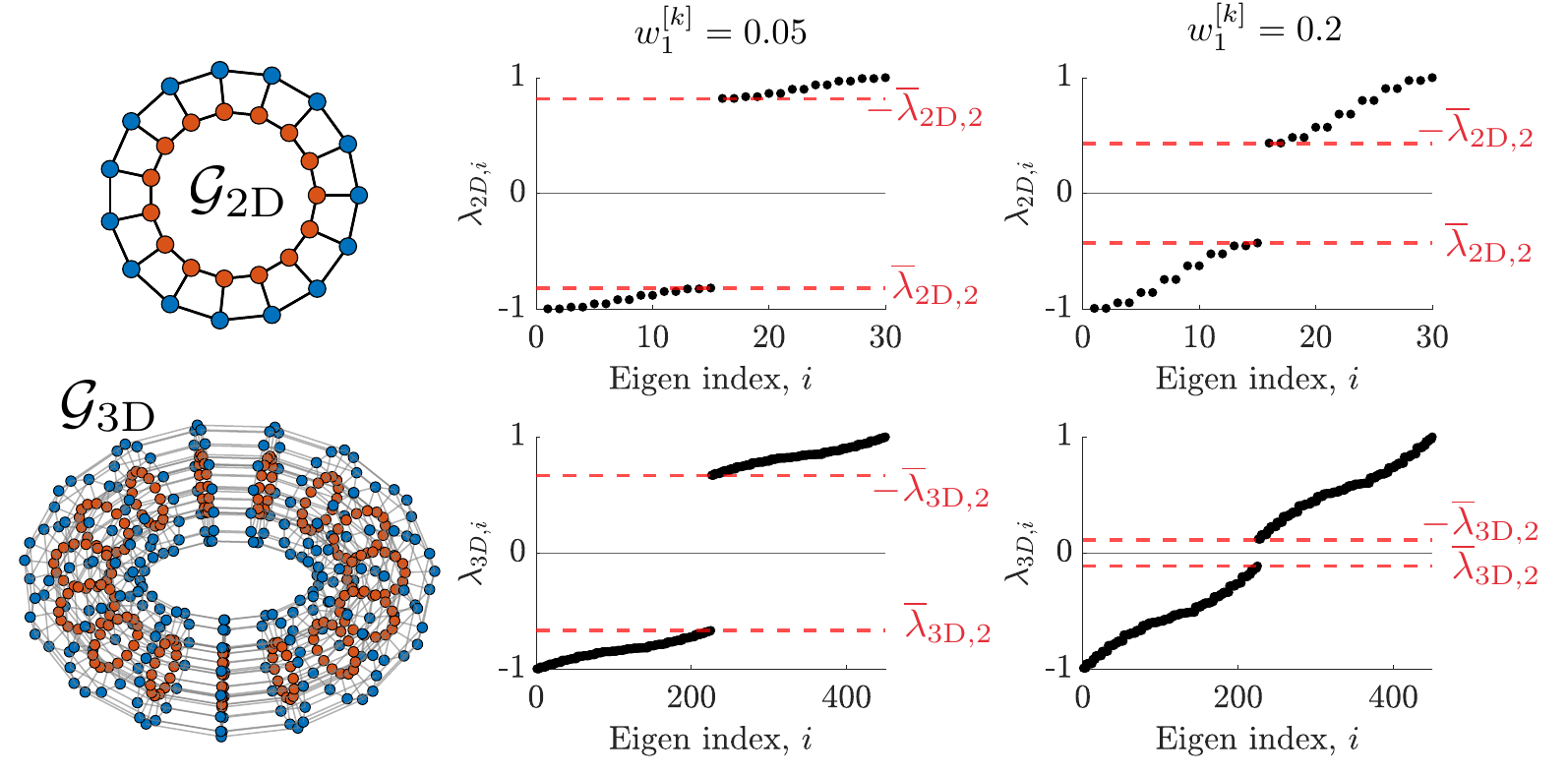}
         \caption{Example spectra of bipartite bilayers.}
         \label{fig:Bipart_B}
\end{subfigure}

\caption{Structure and spectra of bipartite bilayer connectivity graphs.  (A) Example regular bipartite graphs where vertices are coloured with respect to the bipartition sets $V_{1}$ and $V_{2}$ in black and white,  respectively.  (B) Spectra of two bipartite graphs $\cG_{\text{2D}}$ and $\cG_{\text{3D}}$ is shown for $w_{1}^{[k]} = 0.05$ and $w_{1}^{[k]} = 0.2$ for fixed $w_{2}^{[k]} = 1$ ($k \in \{ \text{2D}, \text{3D}  \}$) where the eigen index refers to the position of the eigenvalue when listed in ascending order.  The dashed red lines correspond to the polarity-dependent eigenvalue $\overline{\lambda}_{k,2}$ highlighting its position with respect to the ascending spectrum of the associated large-scale graph.  The vertices of the graphs are coloured layer-wise to emphasise their bilayer structure. } \label{fig:bipartite_example}
\end{figure}

A direct consequence of Lemma \ref{lemma:bipartite_neq_min} is that if the large-scale IO system (\ref{eqn:IO_system}) is spatially coupled by a bipartite bilayer graph $\cG_{k}$ then any trajectory initiated from a small perturbation of an unstable HSS will not be dominantly travelling in the direction of the eigenvector associated with laminar patterning.  Critically,  there will always exist a greater instability mode of the IO system (\ref{eqn:IO_system}).  Figure \ref{fig:bipartite_example}B demonstrates the consequences of Lemma \ref{lemma:bipartite_neq_min},  and for the given bipartite graphs,  the laminar patterning polarity-dependent eigenvalue $\overline{\lambda}_{k,2}$ defines a spectral gap about the origin which is proved in Appendix \ref{Appendix:spec_gap}.\par 

Following Lemma \ref{lemma:bipartite_neq_min},  we focus our attention on the spectral investigation of non-bipartite semi-regular bilayer graphs.  As we are interested in the polarity-driven pattern events using a pre-defined pattern template,  $\pi_{2}$,  we seek layer-wise polarity conditions in which $\overline{\lambda}_{k,2}$ becomes minimal.  Subsequently,  we considered a variety of non-bipartite graphs each with different edge connectivity structure and varied the same-layer weighting parameter $w_{1}^{[k]}$ for fixed $w_{2}^{[k]} =1$,  measuring the position of $\overline{\lambda}_{k,2}$ in terms of the ascending spectrum of the associated large-scale graph.  A summary of the non-bipartite connectivity structures that were considered are in given Table \ref{tab:non_bi_summary} in Appendix \ref{sec:ODE_methods}. \par

For each of the non-bipartite bilayer graphs that were considered,  we observed that decreasing same-layer weighting parameter,  $w_{1}^{[k]}$,  shifted the eigenvalue $\overline{\lambda}_{k,2}$ associated with laminar pattern formation towards the minimum of the spectrum (Figure \ref{fig:nonbipartite_example}).  Furthermore,  we demonstrate that $\overline{\lambda}_{k,2} = \min \lp \Sp \lp \bm{W}_{k} \rp\rp$ for values of $w_{1}^{[k]} < w_{2}^{[k]}$,  noting that this was achieved for higher values of $w_{1}^{[k]}$ in the graphs with more cross-layer connections than same-layer connections,  $n_{1, \cL_{i}} < n_{2, \cL_{i}} $.  Critically,  Figure \ref{fig:nonbipartite_example} highlights that there exists large-scale non-bipartite connectivity graphs that have the capacity to be fully characterised by the extrema of the spectrum of the laminar quotient graph by control of the amount of polarity in the system.  That is,  with high layer-wise polarity,  $w_{1}^{[k]} \ll w_{2}^{[k]}$,  we have $\min \lp \Sp \lp \overline{\bm{W}}_{k} \rp\rp=\min \lp \Sp \lp \bm{W}_{k} \rp\rp$ and $\max \lp \Sp \lp \overline{\bm{W}}_{k} \rp\rp=\max \lp \Sp \lp \bm{W}_{k} \rp\rp$.  \par

By Theorem \ref{thm: existence_of_laminar} we demonstrated that the existence of laminar patterns with competitive kinetics is dependent on the existence of connectivity polarity within the quotient connectivity graphs to induce both HSS instability and monotonicity of solutions.  Therefore,  in the following section, we explore whether solution behaviours observed in the reduced systems are preserved in the associated large-scale systems when the quotient graphs preserve the extrema of the spectra of the large-scale graphs.  Namely,  we show that the analysis conducted on the reduced IO systems yields global pattern convergence in high polarity regimes.

\begin{figure}[h!]
\includegraphics[width=1\textwidth]{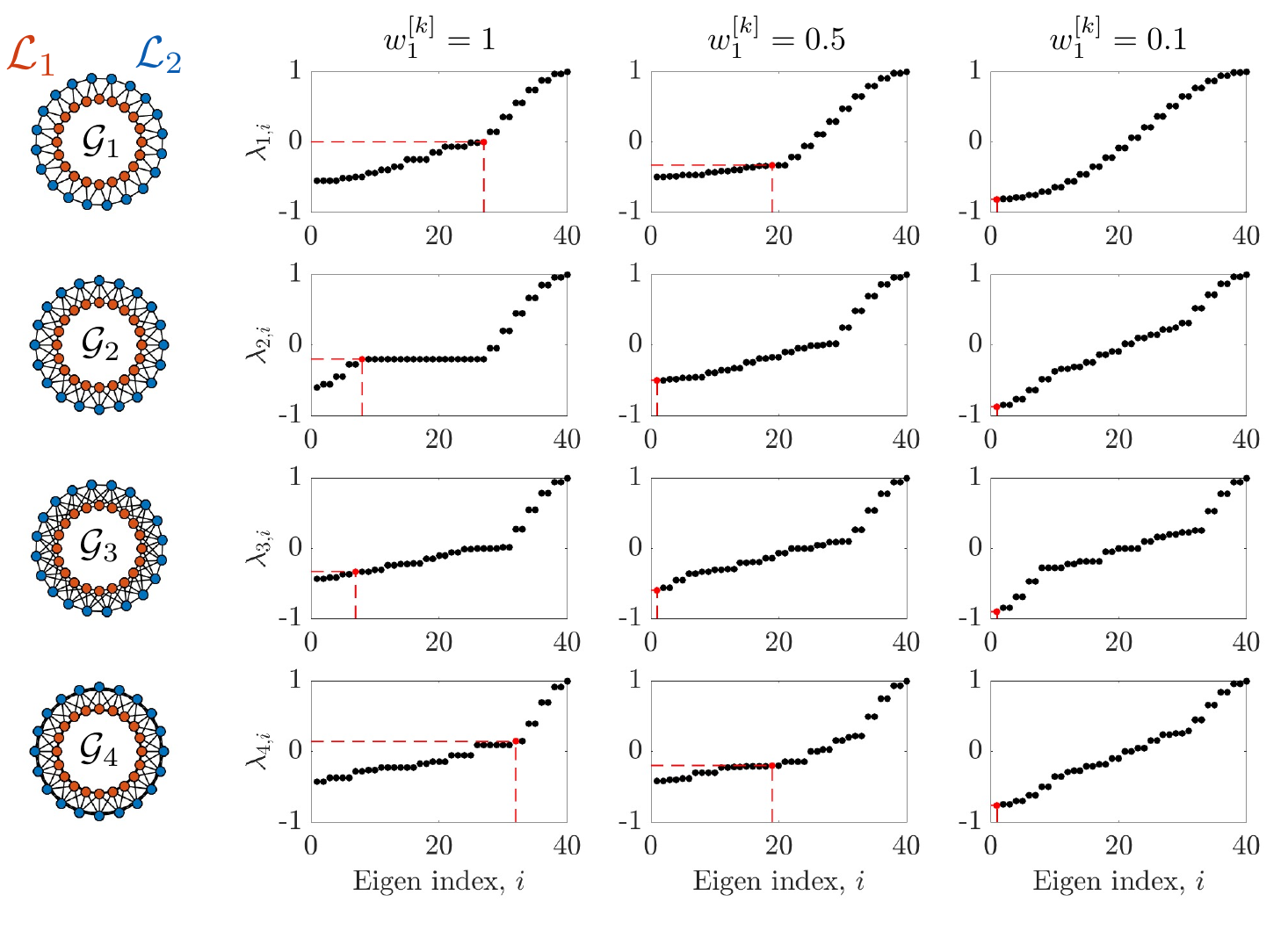}
\caption{Eigenvalues associated with laminar pattern trajectories tend to the minimum of the large-scale spectra as polarity increases in non-bipartite bilayer graphs.  The non-bipartite graphs $\cG_{k}$ ($k = 1,2,3,4$) are shown on the left with vertices coloured layer-wise.  The spectrum of each graph is then shown in ascending order for same-layer polarity values $w_{1}^{[k]} = 1$,   $w_{1}^{[k]} = 0.5$ and $w_{1}^{[k]} = 0.1$,  where the eigen index corresponds to the increasing ordering of the eigenvalues.  In each of the plots,   $\overline{\lambda}_{k,2}$ is highlighted in red with dashed lines included to emphasise the value position in the ascending ordering of eigenvalues.  Details on the connectivity structures each of the graphs is provided in Appendix \ref{sec:ODE_methods}.} \label{fig:nonbipartite_example}
\end{figure}

\subsection{Polarity induced laminar pattern formation derived by quotient systems for large-scale bilayer geometries}\label{sec:global_convergence} $\, $ \\
In this section, we investigate the conditions in which the patterns predicted using the dimension reduction technique of quotient templating are the globally dominant patterns produced in the large-scale IO systems.  We have demonstrated in Section \ref{sec:spec_links} that the spectra of non-bipartite semi-regular bilayer connectivity graphs have the capacity to be bounded by the extrema of the spectra of the associated quotient graphs defined by $\pi_{2}$.  This implies that the polarity-driven HSS instability imposed by the pattern existence condition of Theorem \ref{thm: existence_of_laminar} in the quotient systems must also exist in the large-scale systems and can become dominant in high-polarity regimes.  Therefore,  we now focus our attention on whether the large-scale IO system is monotone with respect to the eigenvector locally directing solutions to laminar patterns,  thus preserving trajectory direction. 

\begin{lemma}\label{lemma:large_scale_mono}
Consider the large-scale IO system (\ref{eqn:IO_system}) spatially coupled by the global adjacency matrix $\bm{P}$ (\ref{eqn:interwoven_adj_mat}) where $\bm{W}_{k} \in \bm{\cW}$ for $k \in \{1,...,r \}$.   Assuming that (A1) is satisfied and the laminar pattern partition,  $\pi_{2}$,  is simultaneously equitable,  then if all connectivity graphs $\cG_{k}$ are highly polarised,  $w_{1}^{[k]} \ll w_{2}^{[k]}$,  the large-scale IO system (\ref{eqn:IO_system}) generates monotone solutions in the direction of laminar patterns.
\end{lemma}
\begin{proof}
Similar to Theorem \ref{thm: existence_of_laminar} we consider the large-scale auxiliary system 
\begin{equation}\label{eqn:large_aux}
\begin{bmatrix}
\dot{\bm{z}}_{1} \\
\vdots \\
\dot{\bm{z}}_{N} 
\end{bmatrix}
 = 
- \begin{bmatrix}
\bm{z}_{1} \\
\vdots \\
\bm{z}_{N} 
\end{bmatrix}
 + \bm{P} 
 \begin{bmatrix}
\bm{T} \lp \bm{z}_{1} \rp \\
\vdots \\
\bm{T} \lp \bm{z}_{N} \rp
 \end{bmatrix} \coloneqq \bm{F} \lp \bm{z} \rp,
\end{equation}
which has the identical behaviour to the large-scale IO system (\ref{eqn:IO_system}) by Lemma \ref{lemma:aux_ss} yet the auxiliary system (\ref{eqn:large_aux}) only explicitly considers the spatially dependent components of the model.  First,  we will construct the sign structure of the eigenvector associated with laminar patterns in the large-scale graphs.  Then,  by transforming the auxiliary system (\ref{eqn:large_aux}) to ensure the positivity of the laminar pattern eigenvector,  we demonstrate that the large-scale IO system (\ref{eqn:IO_system}) has the capacity to become type K for high polarity bilayers.  \par

Linearising the auxiliary system (\ref{eqn:large_aux}) about a generic point $\bm{z} \in \R_{\geq}^{rN}$ yields
\begin{equation} \label{eqn:lin_large_aux}
\frac{ \pd \bm{F}}{ \pd \bm{z} } = -\bm{I}_{rN} + \bm{P}  \begin{bmatrix}
 \cD \bm{T} \lp \bm{z}_{1} \rp & & \\
 & \ddots  & \\
 & &  \cD \bm{T} \lp \bm{z}_{N} \rp
\end{bmatrix}
\end{equation}
where $\bm{T}\lp \cdot \rp$ satisfies (A1) and thus $\text{sgn} \lp \cD \bm{T} \rp = \cS_{1}$ or $\text{sgn} \lp \cD \bm{T} \rp = \cS_{2}$.  As in the proof of Theorem (\ref{thm: existence_of_laminar}),  the transformation $\bm{M} = \bm{I}_{N} \otimes \widetilde{\bm{M}}$ where $\widetilde{\bm{M}} = \text{diag} \lp -1 ,1\rp$ demonstrates the equivalence of the sign structures,  that is,  if $\text{sgn} \lp \cD \bm{T} \rp = \cS_{1}$  then $ \text{sgn} \lp \widetilde{\bm{M}}^{T}  \cD \bm{T} \widetilde{\bm{M}} \rp = \cS_{2} $.  Therefore we continue assuming $\text{sgn} \lp \cD \bm{T} \rp = \cS_{2}$,  critically that $\text{diag} \lp  \cD \bm{T} \lp \bm{z}_{1} \rp , ...,  \cD \bm{T} \lp \bm{z}_{N} \rp  \rp$ is a non-positive matrix.

The reduced graphs associated with the laminar pattern template $\cG_{k,\pi_{2}} $ have eigenvalues $\overline{\lambda}_{k,1} = 1$ and $\overline{\lambda}_{k,2} = a_{k} + b_{k} - 1$ with eigenvectors $\overline{\bm{v}}_{k,1} = [1,1]^{T}$ and $\overline{\bm{v}}_{k,2} = [1,\lp b_{k} - 1 \rp / \lp 1-a_{k} \rp]^{T}$,  noting that $a_{k},b_{k} \in (0,1)$ by definition of the reduced adjacency matrix $\overline {\bm{W}}_{k}$ (\ref{eqn:reduce_adj_mat}).  Subsequently,  the polarity dependent eigenvector has sign structure $\text{sgn} \lp \overline{\bm{v}}_{k,2} \rp = [+, -]^{T}$.  Furthermore,  as $\pi_{2}$ is equitable for all graphs $\cG_{k}$,  then there exists a matrix $\bm{L} \in \{0,1 \} ^{N \times 2}$ that maps the large-scale graph into the quotient graph such that 
\begin{equation}\label{eqn:quo_to_large_map}
\bm{L} \overline{\bm{W}}_{k} = \bm{W}_{k} \bm{L} 
\end{equation}
where $\bm{L}$ allocates the vertices of the large-scale system into the reduced groups associated with the laminar pattern template \cite{godsil2001algebraicBook}.  Owing to the layer-wise vertex indexing as constructed in Section \ref{sec:model_definition},  we have that 
\begin{equation}
\bm{L} = \begin{bmatrix}
\bm{1}_{|\cL_{1}| \times 1}  & \bm{0}_{|\cL_{1}|  \times 1} \\
\bm{0}_{|\cL_{2}|  \times 1}  & \bm{1}_{|\cL_{2}| \times 1} 
\end{bmatrix}.
\end{equation}
From the quotient to large-scale algebraic relation (\ref{eqn:quo_to_large_map}),  we have that $\bm{L} \overline{\bm{v}}_{k,2}$ is an eigenvector of $\bm{W}_{k}$ with eigenvalue $\overline{\lambda}_{k,2}$.  Specifically,  this implies that the eigenvector associated with laminar patterning in the large-scale graphs has the sign structure
$ \text{sgn} \lp \bm{L} \overline{\bm{v}}_{k,2} \rp = [+,...,+, -,...,-]^{T} $ which has $|\cL_{1}|$ and $|\cL_{2}|$ positive and negative entries,  respectively.  Hence,  the matrix $ \widetilde{\bm{R}} = \text{diag} \lp 1,...,1,-1,...,-1 \rp$ orientates the laminar patterning eigenvector $ \bm{L} \overline{\bm{v}}_{k,2} $ in the positive orthant,  i.e.,  $\widetilde{\bm{R}}   \bm{L} \overline{\bm{v}}_{k,2} >0 $.  \par  

We next introduce the transformation $\bm{w} = \bm{R}\bm{z}$ where $\bm{R} = \widetilde{\bm{R}}  \otimes \bm{I}_{r}$,  noting that $\bm{R}^{-1} = \bm{R}$.  Following this change of variables,  let $\bm{X}_{1} =  \text{diag} \lp \cD \bm{T} \lp \bm{z}_{1} \rp, ...,  \cD \bm{T} \lp \bm{z}_{|\cL_{1}|} \rp  \rp$ and $\bm{X}_{2} =  \text{diag} \lp \cD \bm{T} \lp \bm{z}_{|\cL_{1}| + 1} \rp, ...,  \cD \bm{T} \lp \bm{z}_{N} \rp  \rp$ be non-positive matrices,  then in combination with the layer-wise block formulation of the bilayer adjacency matrices,  $\bm{W}_{k}$,  the linearised auxiliary system (\ref{eqn:lin_large_aux}) has the form
\begin{align}\label{eqn:aux_large_jac_trans}
\bm{R} \frac{ \pd \bm{F}}{ \pd \bm{z} }  \bm{R} &= -\bm{I}_{rN} + \bm{R} \bm{P}  \begin{bmatrix}
 \bm{X}_{1}  & \bm{0}  \\
  \bm{0} &   \bm{X}_{2}
\end{bmatrix} \bm{R},  \nonumber \\ 
&=  -\bm{I}_{rN} + \lp \sum_{i=1}^{r}  \widetilde{\bm{R}} \bm{W}_{i} \otimes \bm{D}_{i} \rp \cdot 
\begin{bmatrix}
 \bm{X}_{1}  &   \bm{0}  \\
  \bm{0}  &   -\bm{X}_{2}
\end{bmatrix},
\nonumber \\
&= -\bm{I}_{rN} + \lp \sum_{i=1}^{r} \begin{bmatrix}
 \widehat{\bm{W}}_{1,\cL_{1}}^{[i]} \otimes \bm{D}_{i} &  \widehat{\bm{W}}_{2,\cL_{1}}^{[i]} \otimes \bm{D}_{i} \\
 - \lp \widehat{\bm{W}}_{2,\cL_{1}}^{[i]} \rp^{T} \otimes \bm{D}_{i} &  - \widehat{\bm{W}}_{1,\cL_{2}}^{[i]}\otimes \bm{D}_{i} 
\end{bmatrix} \rp \cdot \begin{bmatrix}
 \bm{X}_{1}  &  \bm{0}  \\
  \bm{0}  &  - \bm{X}_{2}
\end{bmatrix},  \nonumber \\
&= -\bm{I}_{rN}  + \sum_{i=1}^{r} \begin{bmatrix}
 \lp \widehat{\bm{W}}_{1,\cL_{1}}^{[i]} \otimes \bm{D}_{i}  \rp \bm{X}_{1} & - \lp \widehat{\bm{W}}_{2,\cL_{1}}^{[i]} \otimes \bm{D}_{i} \rp \bm{X}_{2} \\
 - \lp  \lp \widehat{\bm{W}}_{2,\cL_{1}}^{[i]} \rp^{T} \otimes \bm{D}_{i} \rp \bm{X}_{1} &  \lp \widehat{\bm{W}}_{1,\cL_{2}}^{[i]}\otimes \bm{D}_{i} \rp \bm{X}_{2}
\end{bmatrix},
\end{align}
by the mixed-product and block-product properties of the Kronecker product \cite{graham2018kronecker}.  The transformed auxiliary system (\ref{eqn:aux_large_jac_trans}) is monotone if the off-diagonal row-sum is non-negative by Lemma \ref{lemma:type_k}.  Namely,  if $\tau \lp i\rp = \lp i-1 \rp \bmod{r} + 1$ then
\begin{equation}
\sum_{j\neq i} \lp\bm{R} \frac{ \pd \bm{F}}{ \pd \bm{z} }  \bm{R}  \rp_{ij} = \begin{dcases}  
 \hat{w}_{1}^{[ \tau \lp i\rp]}   \sum_{j=1,j\neq i}^{r |\cL_{1}|}  \lp \bm{X}_{1} \rp_{ \tau \lp i\rp j} -  \hat{w}_{2}^{[\tau \lp i\rp]}  \sum_{j=1, j\neq i}^{r |\cL_{2}|} \lp \bm{X}_{2}  \rp_{ \tau \lp i\rp j}  & 1 \leq i \leq r|\cL_{1}|,  \\
- \hat{w}_{2}^{[ \tau \lp i\rp]}   \sum_{j=1,j\neq i}^{r |\cL_{1}|}  \lp \bm{X}_{1} \rp_{ \tau \lp i\rp j} +  \hat{w}_{1}^{[\tau \lp i\rp]}  \sum_{j= 1,j\neq i}^{r |\cL_{2}|} \lp \bm{X}_{2} \rp_{ \tau \lp i\rp j} &  \text{ }  r|\cL_{1}|+ 1 \leq i \leq rN. \\
\end{dcases}
\end{equation}
In particular,  as $\text{sgn} \lp \cD \bm{T} \rp = \cS_{2}$,  then all positive and negatives components of the row-sum are scaled by $w_{2}^{[k]}$ and $w_{1}^{[k]}$,  receptively.  Then for sufficiently small values of $w_{1}^{[k]}$ combined with relatively large values $w_{2}^{[k]}$,  confirmed by $w_{1}^{[k]} \ll w_{2}^{[k]}$,   we have that,
\begin{equation}
\sum_{j\neq i} \lp\bm{R} \frac{ \pd \bm{F}}{ \pd \bm{z} }  \bm{R}  \rp_{ij} \geq 0
\end{equation} 
for all $ 1 \leq i \leq rN$.  Therefore the auxiliary system (\ref{eqn:large_aux}) is type K by Lemma \ref{lemma:type_k} and so is monotone in the direction for solutions associated with laminar patterning in high polarity regimes.
\end{proof}

Applying the cooperative transformation in high-polarity regimes to an IO system (\ref{eqn:IO_system}) where the extrema of the spectra are preserved in the quotient mapping guarantees the global convergence of laminar patterns in the large-scale systems.  Critically,  this extends the existence statement of Theorem \ref{thm: existence_of_laminar} to sufficient conditions for large-scale laminar patterning.

\begin{theorem}[Global convergence of laminar patterns in highly-polarised regimes]\label{thm:global_convergence}
Consider the large-scale IO system (\ref{eqn:IO_system}) spatially coupled by the global adjacency matrix $\bm{P}$ (\ref{eqn:interwoven_adj_mat}) where $\bm{W}_{k} \in \bm{\cW}$ for $k \in \{1,...,r \}$.   Assuming that (A1) is satisfied, the laminar pattern partition $\pi_{2}$ is simultaneously equitable,  and each connectivity graph, $\cG_{k}$,  is highly polarised,  $w_{1}^{[k]} \ll w_{2}^{[k]}$.  Then if $\overline{\lambda}_{k,2} = \min \lp \Sp \lp \bm{W}_{k} \rp \rp$ such that the laminar pattern existence criterion,  Theorem \ref{thm: existence_of_laminar},  is satisfied,  then laminar patterns are globally convergent in the large-scale IO system (\ref{eqn:IO_system}).
\end{theorem}
\begin{proof}
Following Theorem \ref{thm: existence_of_laminar},  by analysing the quotient graphs there exists $\overline{\bm{ \Lambda}}_{2}$ such that the HSS instability condition (\ref{thm: HSS instability det}) is satisfied.  In addition,  Lemma \ref{lemma:large_scale_mono} guarantees that the IO system (\ref{eqn:IO_system}) generates monotone solutions in the direction of laminar patterns,  such that the eigenvector associated with $\overline{\lambda}_{k,2}$ is directed in the positive orthant,  $\widetilde{\bm{R}}\bm{L} \overline{\bm{v}}_{k,2} >0$.  Furthermore,  Lemma \ref{lemma:irreducible} ensures that the linearised IO system is irreducible and thus the IO system (\ref{eqn:IO_system}) is strongly monotone by Lemma \ref{lemma:strongly_monotone}.  \par  
By the identical arguments of Theorem \ref{thm: existence_of_laminar},  the corresponding large-scale auxiliary system (\ref{eqn:lin_large_aux}) has bounded solutions,  which induces the convergence of solutions to steady-state $\bm{Rz}^{**} \neq \bm{Rz}^{*}$ by the Cooperative Irreducible Convergence Theorem (Theorem 4.3.3 in \cite{smith2008monotone}).  Critically,  mapping back to the original coordinating system guarantees that vertices in different layers have contrasting solutions.
\end{proof}

The sufficient conditions for large-scale laminar patterning outlined in Theorem \ref{thm:global_convergence} ensure that the behaviour observed in the quotient systems is preserved in the corresponding large-scale systems.  Subsequently,  this enables an analytic approach to pattern prediction as we can fully determine the spectra of the quotient graphs $\cG_{2}$ independently and without imposing commutativity conditions on the reduced adjacency matrices.  The following example demonstrates the accessibility of the analysis for large-scale IO systems spatially coupled multilayer connectivity graphs.

\begin{example}\label{example:large_example}
We revisit Example \ref{example:laminar_existence} to seek a polarity regime that guarantees the global convergence of laminar patterns using analysis conducted in the quotient systems when templating the large-scale system using the equitable partition, $\pi_{2}$.  Namely,  in conjunction with the results of applying Theorem \ref{thm: existence_of_laminar} to the DIDO system  (\ref{example_io_system_1}-\ref{example_io_system_2}) as in Example \ref{example:laminar_existence},  we also invoke Theorem \ref{thm:global_convergence} to isolate regions of polarity parameter values for $w_{1}^{[1]}$ and $w_{1}^{[2]}$ in which the extrema of the quotient graph spectra are the extrema of the large-scale graphs.  \par 
As each of the connectivity large-scale and quotient graphs are row-stochastic,  we always have 
\begin{equation}
\max \lp \Sp \lp \overline{\bm{W} }_{k}\rp\rp = \max \lp \Sp \lp \bm{W} _{k}\rp\rp = 1,
\end{equation}
therefore,  the quotient graphs retain the maximum eigenvalues,  and so now we focus on the preservation of the minimal eigenvalues. \par

In Figure \ref{fig:nonbipartite_example} we have demonstrated that for $w_{1}^{[k]} < 0.5$ and $w_{2}^{[k]} = 1$,  the quotient connectivity graphs for short-range diffusion and contact-dependent signalling mechanism,  $\cG_{1,\pi_{2}}$ and $\cG_{2,\pi_{2}}$ (which are denoted $\cG_{1}$ and $\cG_{3}$ in Figure \ref{fig:nonbipartite_example},  respectively) have the capacity to bound the spectra of the large-scale graphs $\cG_{1}$ and $\cG_{2}$ from below.  Critically,  this implies that for any $w_{1}^{[k]} < 0.5$ with fixed $w_{2}^{[k]} = 1$,  which induced HSS instability,  solutions will be locally directed towards laminar patterning and so following from Theorem \ref{thm:global_convergence},  for sufficiently small $w_{1}^{[k]} <0.5$,  the large-scale DIDO system  (\ref{example_io_system_1}-\ref{example_io_system_2}) will converge to laminar patterns. \par 

To highlight the results of applying both theorems \ref{thm: existence_of_laminar} and \ref{thm:global_convergence} to the example DIDO system  (\ref{example_io_system_1}-\ref{example_io_system_2}),  regions of pattern convergence were found numerically in Figure \ref{fig:converge_ex_1} which includes examples of large-scale simulations for which laminar patterns are and are not dominant.  It is worth noting that the magnitude of the difference between $w_{1}^{[k]}$ and $w_{2}^{[k]}$ is dependent on the magnitude of the entries of $ \cD \bm{T} \lp \bm{u}_{i} \rp$ and thus assuming that $w_{1}^{[k]} \ll w_{2}^{[k]}$ is sufficient for the monotonicity of the large-scale system but is not necessary to satisfy the type K criteria (Lemma \ref{lemma:type_k}).  Subsequently,  simply selecting polarity parameters in which both the HSS instability condition for the reduced system (\ref{eqn:hss_instab_example}) and $\min \lp  \Sp  \lp \overline{\bm{W}}_{k} \rp \rp = \min \lp \Sp \lp \bm{W}_{k} \rp \rp$ are satisfied resulted in the large-scale system converging to laminar patterns without requiring significant layer-wise polarity.  

\begin{figure}[h!]
\includegraphics[width=\textwidth]{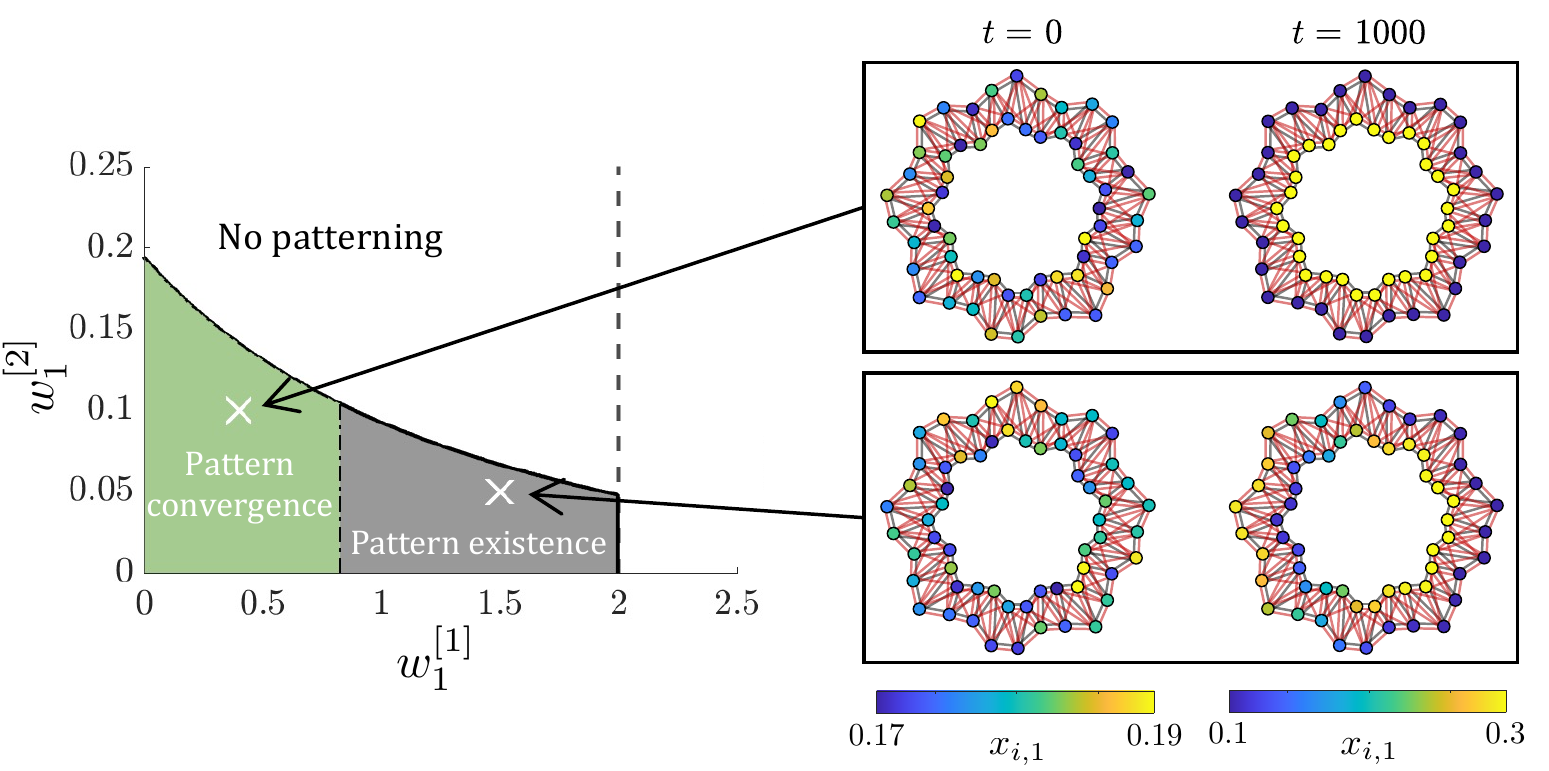}
\caption{Polarity parameter regimes for the existence and convergence of laminar pattern in the large-scale IO system (\ref{example_io_system_1}-\ref{example_io_system_2}).  The grey existence region is determined using the quotient system analysis and is defined by polarity-driven HSS instability inequality (\ref{eqn:hss_instab_example}) for fixed $w_{2}^{[1]}=w_{2}^{[1]} = 1$.  The green convergence region highlights the subset of grey region in $\lp w_{1}^{[1]},  w_{1}^{[2]} \rp$-space in which $\min \lp  \Sp  \lp \overline{\bm{W}}_{k} \rp \rp = \min \lp \Sp \lp \bm{W}_{k} \rp \rp$.  Example large-simulations are shown for polarity parameters inside the convergence region, (0.4,0.1),  and inside the existence region (1.5,0.05).  Large-scale bilayer graphs are shown with both $\cG_{1}$ and $\cG_{2}$ embedded in the same vertex set with edges in black and red,  respectively.  Vertex colour corresponds to the values of $x_{i,1}$ in each $v_{i}$.  Simulations were initiated from small random perturbations about the HSS of the IO system (\ref{example_io_system_1}-\ref{example_io_system_2}) and first and final states are shown following trajectory convergence.  IO system (\ref{example_io_system_1}-\ref{example_io_system_2}) parameter values and details on simulations are given in Appendix \ref{sec:ODE_methods}.  } \label{fig:converge_ex_1}
\end{figure}

\end{example}

As highlighted in Example \ref{example:large_example},  theorems \ref{thm: existence_of_laminar} and \ref{thm:global_convergence} facilitate the analytic study of laminar pattern formation in large-scale interconnected dynamical systems,  independent of the number of cells in the system or physical dimension owing to the topological definition of the connectivity graphs.  Hence the pattern analysis conducted on the quotient systems can evolve from explorative- which geometries enable laminar patterning,  to constructive- how much edge weight manipulation is required to robustly generate laminar patterns.

\section{Discussion}
In this study,  we have developed analytic methods for exploring the interplay of cellular polarity and multiple signalling mechanisms in the emergence of laminar patterns in bilayer tissues independent of the precise intracellular kinetics.  To facilitate such analysis we focused on methods of dimension reduction of large interconnected dynamical systems that preserve fundamental cellular behaviour.  Specifically,  we demonstrate that cell signalling transfer dynamics can be treated as a proxy for intracellular components,  reducing the dimensionality of the spatially discrete ODE systems by analysing only the spatially dependent intracellular components,  which enabled us to provide sufficient conditions for the existence and uniqueness of the homogeneous steady state.  \par

In addition,  we use properties of commuting graphs to decompose large MIMO systems into lower-order interconnected systems,  decoupling the spatial and temporal components.  This not only has advantages in reducing the computational cost associated with large-scale eigenvalue problems but also enables the direct analysis of the influence each signalling mechanism has on driving spatial instabilities of the homogeneous steady state.  From a practical standpoint,  the requirement of commuting graphs of cell signalling currently limits the applications of the large-scale HSS instability results in general pattern formation problems as there exist no analytically tractable methods for checking these conditions for large graphs.  Therefore developing a procedure for constructing commuting families of large signalling graphs is critical to broadening the scope of these modelling approaches.  \par 

Combining methods of multilayer graph partitions with monotone dynamical systems theory,  we demonstrate the existence of laminar pattern formation with competitive kinetics relies on the amount of signalling polarisation present within each graph.  Critically,  the application of equitable partitions to the connectivity structures where layer-wise symmetries are present enables drastic dimensionality reductions of the global dynamical system when seeking contrasting steady-states between the bilayer of cells.  Thereby exploiting the eigenvalue structure of the quotient graphs we demonstrate the instability conditions derived for large-scale interconnected dynamical systems that can be applied to the reduced system,  independent of commutativity of the quotient graphs,  which facilitates the investigation of whether the pre-defined contrasting states are achievable with the given kinetics.  The symmetry requirements of the equitable partitions need not be restricted to globally regular cell-cell interaction graphs.  We only require regularity within each partition which therefore permits the application of semi-regular graphs for dimension reduction.  Such graphs can then capture characteristic traits of the biological system such as subpopulation phenotypes and tissue curvature,  and their influence on intracellular behaviour.\par

These methods of prescribing patterns allude to studying the inverse problem,  specifically,  starting with the desired pattern of the tissue and then defining constraints for the intracellular kinetics that have the potential to induce such instabilities,  as previously demonstrated in spatially continuous Turing systems \cite{woolley2021bespoke}.  Additionally,  as the full and quotient system analysis depends only on the topology of the connectivity networks,  the results from this study are immediately applicable to more biologically relevant 3D morphologies. \revisedTC{ Classically,  introducing 3D structures drastically increases the computational complexity in pattern formation analysis \cite{arpon2021spatial,  okuda2018combining,krause2021introduction},  yet the topological approach allows for the transition between physical dimension with no additional requirements as discussed in \cite{moore2021algebraic}}. \par

\revisedTC{
Investigating the link between the reduced and large-scale dynamical systems when seeking laminar patterns,  we demonstrate the statements of laminar existence derived using pattern-templating have the capacity to be globally convergent in the corresponding large-scale interconnected system in high polarity regimes.  To show the existence of a monotone transformation we imposed weak but sufficient conditions that $w_{1}^{[k]} \ll w_{2}^{[k]}$,  highlighting the requirement of edge weight anisotropy for laminar pattern formation.  However, we suspect that this condition can be significantly refined by illustrating a dependence on the magnitude of entries of $\cD \bm{T} \lp \bm{u}_{i} \rp$ when applying the type K criterion for monotone solution behaviour,  namely,  having \textit{a priori} estimations of the size of the cellular output signals for given input signal regimes.} \par 

As discussed in the previous interconnected monotone systems studies of pattern formation \cite{Arcak2013,gyorgy2017pattern,gyorgy2017patternMulti},  the most limiting assumption in large-scale systems analysis is the existence of competitive to cooperative monotone kinetics transformation which previously has relied on the sufficient requirement of the connectivity graphs being bipartite.  However,  in Section \ref{sec:global_convergence} we not only demonstrate that laminar patterns are not the dominant pattern of bipartite bilayer graphs but also manipulating graph edge weights of non-bipartite graphs enables competitive to cooperative kinetics transformations for laminar pattern formation.  The key feature of cooperative dynamics used in these pattern formation studies is the guarantee of non-periodic solutions when considering bounded kinetics \cite{smith2008monotone}.  Therefore another promising direction to ensure such solution behaviour is the study of variational families associated with the interconnected systems \cite{angeli2021robust},  that is,  applying Lyapunov methods for non-oscillatory dynamics to enable the investigation of intracellular crosstalk inference in biologically relevant morphologies.

Throughout this study,  we have reserved precise definitions of the intracellular kinetics and associated signalling mechanisms to consider general competitive MIMO dynamics.  Subsequently,  the generality of results presented here enables the investigation of crosstalk of key molecular pathways with multiple spatially dependent intracellular signalling components,  such as the well-established Wnt-Notch interactions that have been observed in both intestinal and mammary epithelia \cite{collu2014wnt}.  Both the Wnt and Notch pathways are involved in cell-fate determination and have been observed to have active apical-basal polarity mechanisms during tissue development \cite{sjoqvist2019say,bertrand2016beta}.  However existing models have previously been limited to analysis of one or two cells \cite{kay2017role,agur2011dickkopf1},  the methods we provide here allow us to study how the geometry of the tissue influences such cell-fate choices,  specifically within the bilayer structures commonly found in mammary glands.

\section*{Acknowledgements}
JWM is supported by Knowledge Economy Skills Scholarships (KESS2), a pan-Wales higher-level skills initiative led by Bangor University on behalf of the Higher Education sector in Wales. It is part-funded by the Welsh Government’s European Social Fund (ESF).

\bibliographystyle{unsrt}
\bibliography{polar_patterns_mixed_D2}

\appendix

\section{Additional properties of interwoven matrices}\label{Appendix:Interweave}

Here,  we present further properties of interwoven matrices which have particular applications in dynamical network theory of mixed kernels.  Specifically,  an interwoven matrix $\bm{P}$ is composed of the sequence of real matrices $\bm{\mathcal{M}} = \{ \bm{M}_{i} \,  | \,   \bm{M}_{i} \in \R^{N \times N} \}$ called constructor matrices such that the rows and columns of each matrix $\bm{M}_{i}$ are uniformly separated by zero elements,  preserving the structure of $\bm{M}_{i}$,  where the order of $\bm{\mathcal{M}}$ defines the sequence of spacing.  Formally we define an interwoven matrix $\bm{P}$ by 
\begin{equation}\label{eqn: interwoven_Q}
\bm{P} = \sum_{i = 1}^{r} \bm{M}_{i} \otimes \text{diag} \lp \delta_{i,1},...,\delta_{i,r} \rp,
\end{equation}
where $\otimes$ is the Kronecker product and $\delta_{i,j}$ is the Kronecker delta function (\ref{eqn:Kron_delta}) as defined in Section \ref{sec:model_definition}.  However,  here we do not assume that the constructor matrices $\bm{M}_{i}$ are nonnegative.   \par
The following result for the diagonal decomposition of the $\bm{P}$.

\begin{lemma}\label{lemma:permutation}
Let $\bm{P}$ be the interwoven matrix defined as defined in equation\autoref{eqn: interwoven_Q} and let $\bm{Q}$  be the permutation matrix such that
 \begin{equation}
 \bm{Q} = [  \bm{I}_{N} \otimes \bm{e}_{1} , ...,     \bm{I}_{N} \otimes \bm{e}_{r} ].
 \end{equation}
 where $\bm{e}_{i} =  \begin{bmatrix} \delta_{i,1}, \,  \cdots,  \,  \delta_{i,r}\end{bmatrix}^{T}$.  Then
 \begin{equation}
 \bm{Q}^{T} \bm{P} \bm{Q} = \bigoplus_{i= 1}^{r} \bm{M}_{i},
 \end{equation}
where $\oplus$ is the direct sum of tensors.

\end{lemma}
\begin{proof}
Consider the permutation map $\tau: \{1,...,  rN\} \rightarrow \{1,...,rN \}$ such that
\begin{equation}\label{eqn:tau_map}
\tau \lp x \rp = \lp \lp x-1 \rp \,  \text{mod}\,  r  \rp N + \left\lfloor\dfrac{x-1}{r}\right\rfloor + 1
\end{equation}
which permutes the rows and columns of $\bm{P}$ so any row and columns $\bm{M}_{i}$ of $\bm{P}$ become adjacent for $1 \leq i \leq r $.  In cycle notation,  $\tau \lp x \rp $ defines the mapping
\begin{equation}\label{eqn:cycle_perm}
\begin{pmatrix}
1 & 2& \cdots & r & r+1 & \cdots & rN \\
1 & N + 1 &  \cdots & (r-1)N + 1 & 2 & \cdots & (r-1)N + N \\
\end{pmatrix}
\end{equation}
which represents the column and row permutation of $\bm{P}$.  The cycle (\ref{eqn:cycle_perm}) defined by equation (\ref{eqn:tau_map}) yields the following matrix representation 
\begin{equation}
\bm{Q} = [  \bm{I}_{N} \otimes \bm{e}_{1} , ...,     \bm{I}_{N} \otimes \bm{e}_{r} ],
\end{equation}
namely,  $ \bm{Q} _{i, \tau\lp i  \rp} = 1$ and zero entries else.  Therefore applying the transformation $\bm{Q}$ to $\bm{P}$ produces the block diagonal representation where
\begin{equation}
 \bm{Q}^{T} \bm{P} \bm{Q} = \text{diag}\lp \bm{M}_{1},...,\bm{M}_{r} \rp
\end{equation}
which is by definition the direct sum of matrices $\bm{M}_{i}$.
\end{proof}

The block diagonal representation of $\bm{P}$ following from \autoref{lemma:permutation} motivates the subsequent properties involving the spectra and inverse of the interwoven matrix $\bm{P}$.

\begin{lemma} \label{lemma:interweave_spec}
Let $\bm{P}$ be the interwoven matrix as defined in equation (\ref{eqn: interwoven_Q}).  Then $\bm{P}$ has the following properties:
\begin{enumerate}[label=(\roman*)]
\item  \begin{equation}
\Sp \lp \bm{P} \rp   =  \bigcup_{i = 1}^{r} \Sp \lp \bm{M}_{i} \rp
\end{equation} 
including multiplicities;

\item if $\bm{M}_{i}$ is invertible for all $1 \leq i \leq k $,  then the inverse of the interwoven matrix $\bm{P}$ is the interweave of the inverse of the construction matrices. That is,
\begin{equation}
\bm{P}^{-1} = \sum_{i = 1}^{r} \bm{M}_{i}^{-1} \otimes \text{diag} \lp \delta_{i,1},...,\delta_{i,r} \rp,
\end{equation}
\item the trace of the interwoven matrix is the sum of the traces of the constructor matrices
\begin{equation}
\Tr \lp \bm{P} \rp = \sum_{i=1}^{r} \Tr \lp  \bm{M}_{i }\rp,
\end{equation}
\item the determinant of the interwoven matrix is the product of the determinant of the constructor matrices
\begin{equation}
\det \lp \bm{P} \rp = \prod_{i=1}^{r} \det \lp  \bm{M}_{i }\rp.
\end{equation}
\end{enumerate}

\end{lemma}

\begin{proof}
Let $\lambda_{k,j} \in \Sp \lp \bm{M}_{k} \rp$ with its associated eigenvector $\bm{v}_{k,j}$.  Then define the interweave extension of $\bm{v}_{k,j}$ by 
\begin{equation}
\tilde{\bm{v}}_{k,j} = \bm{v}_{k,j} \otimes \begin{bmatrix} 
 \delta_{j,1}\\
 \vdots\\
  \delta_{j,r}
\end{bmatrix}.
\end{equation} 
For brevity,  denote the Kronecker diagonal matrix by $\bm{D}_{i} = \text{diag} \lp \delta_{i,1},...,\delta_{i,r} \rp$ and then by direct computation we have
\begin{align}
\bm{P} \tilde{\bm{v}}_{k,j}  =& \lp  \sum_{i = 1}^{r} \bm{M}_{i} \otimes \bm{D}_{i}   \rp  \tilde{\bm{v}}_{k,j}, \nonumber \\
=& \lp  \sum_{i = 1}^{r} \bm{M}_{i} \otimes \bm{D}_{i}  \rp   \lp   \bm{v}_{j}^{k} \otimes \begin{bmatrix} \delta_{j,1}, \,  \cdots,  \,  \delta_{j,r}\end{bmatrix}^{T}   \rp  , \nonumber\\
=&  \lp  \sum_{i = 1}^{r} \bm{M}_{i} \bm{v}_{k,j} \otimes \bm{D}_{i} \begin{bmatrix} \delta_{j,1}, \,  \cdots,  \,  \delta_{j,r}\end{bmatrix}^{T}   \rp , \nonumber \\
=& \,  \bm{M}_{j} \bm{v}_{k,j} \otimes   \begin{bmatrix} \delta_{j,1}, \,  \cdots,  \,  \delta_{j,r}\end{bmatrix}^{T} ,
\end{align}
where the last two equalities follow from the mixed product property of the Kronecker product and that direct multiplication of the Kronecker matrix and vector are non-zero only if $i = j$.  Therefore we have that 
\begin{equation}
\bm{P} \tilde{\bm{v}}_{k,j}  = \bm{M}_{j} \bm{v}_{k,j} \otimes   \begin{bmatrix} \delta_{j,1}, \,  \cdots,  \,  \delta_{j,r}\end{bmatrix}^{T}  = \lambda_{k,j}  \bm{v}_{k,j} \otimes  \begin{bmatrix} \delta_{j,1}, \,  \cdots,  \,  \delta_{j,r}\end{bmatrix}^{T} = \lambda_{k,j} \tilde{\bm{v}}_{k,j} ,
\end{equation}
thus $ \lambda_{k,j} $ is an eigenvalue of $\bm{P}$ with associated eigenvector $\tilde{\bm{v}}_{k,j} $.  \par
Next,  there exists $\bm{M}_{i}^{-1}$ for all $1 \leq i \leq r$ from the assumption in (ii).  Note that 
\begin{equation}\label{kronecker_mat_prop}
\bm{D}_{i} \bm{D}_{j} = \left\{
	\begin{array}{ll}
		\bm{D}_{i}   &  i = j, \\
		\bm{0}_{r \times r} &   i \neq j,
	\end{array}
\right. 
\end{equation} 
then consider the following matrix $\bm{R}$ defined by the multiplication
\begin{align}
\bm{R} =& \lp  \sum_{i = 1}^{r} \bm{M}_{i} \otimes \bm{D}_{i} \rp \lp  \sum_{i = 1}^{r} \bm{M}_{i}^{-1} \otimes \bm{D}_{i} \rp, \nonumber \\
=& \lp \bm{M}_{1} \otimes \bm{D}_{1} \rp \lp  \sum_{i = 1}^{r} \bm{M}_{i}^{-1} \otimes \bm{D}_{i} \rp + ...  +  \lp \bm{M}_{r} \otimes \bm{D}_{r} \rp \lp  \sum_{i = 1}^{r} \bm{M}_{i}^{-1} \otimes \bm{D}_{i} \rp \label{inverse_comp} .
\end{align}
From the mixed-product property of the Kronecker product and equation\autoref{kronecker_mat_prop},  we have that\autoref{inverse_comp} reduces to 
\begin{align}
\bm{R} =& \,  \bm{M}_{1}    \bm{M} _{1}^{-1} \otimes \bm{D}_{1} + ... + \bm{M}_{r}    \bm{M} _{r}^{-1} \otimes \bm{D}_{r} , \nonumber \\
=& \sum_{i = 1} ^{r} \bm{I}_{n} \otimes \bm{D}_{i}  = \bm{I}_{rN} ,
\end{align}
hence the inverse of $\bm{P}$ is given by $\bm{P}^{-1} =  \sum_{i = 1}^{r} \bm{M}_{i}^{-1} \otimes \bm{D}_{i}  $ as required for (ii).  \par

The trace of a Kronecker product is the product of the trace \cite{hardy2019matrix} such that $\Tr \lp \bm{M}_{i}  \otimes \bm{D}_{i}\rp = \Tr \lp \bm{M}_{i} \rp \Tr \lp \bm{D}_{i} \rp$.  Therefore applying the trace to the definition of $\bm{P}$ (\ref{eqn: interwoven_Q}) yields
\begin{align}
\Tr \lp  \bm{P} \rp &= \Tr \lp \sum_{i= 1}^{r} \bm{M}_{i} \otimes \bm{D}_{i} \rp = \sum_{i= 1}^{r} \Tr \lp \bm{M}_{i} \otimes \bm{D}_{i} \rp = \sum_{i= 1}^{r} \Tr \lp \bm{M}_{i}  \rp \Tr \lp \bm{D}_{i} \rp = \sum_{i= 1}^{r} \Tr \lp \bm{M}_{i}  \rp ,
\end{align}
where the second equality holds by the trace of the sum of matrices \cite{hardy2019matrix} and the fourth holds by $\Tr \lp \bm{D}_{i} \rp = 1$ for all $1 \leq i \leq r$.  \par

Property (iv) follows immediately from (i) by expressing the determinant of a matrix as the product of the eigenvalues including multiplicities \cite{hardy2019matrix}.  From (i) we have that $ \Sp \lp \bm{P} \rp   =  \Sp \lp \bm{M}_{1} \rp \cup ...\cup \Sp \lp \bm{M}_{r} \rp$ including multiplicities and so we know the eigenvalues of $\bm{P}$ are all the eigenvalues of each $\bm{M}_{i}$.  Subsequently,  the determinant of $\bm{P}$ must be the product of all these eigenvalues which leads to the required representation

\begin{equation}
\det{ \lp \bm{P} \rp} = \lp \prod_{j=1}^{N} \lambda_{1,j} \rp... \lp \prod_{j=1}^{N} \lambda_{r,j} \rp = \det \lp \bm{M}_{1} \rp... \det \lp \bm{M}_{r} \rp = \prod_{i= 1}^{r} \det \lp \bm{M}_{i} \rp.
\end{equation}

\end{proof}
A direct consequence of \autoref{lemma:interweave_spec} is that if $\bm{M}_{j}$ are nonnegative,  then the spectral radius, $ \rho$,  of interwoven matrix $\bm{P}$ is a real eigenvalue and is defined by
\begin{equation}
\rho \lp \bm{P} \rp = \max_{j} \lp  \rho \lp\bm{M}_{j}  \rp \rp  =  \max \lp  \bigcup_{i = 1}^{r} \Sp \lp \bm{M}_{i} \rp  \rp
\end{equation}
from the Perron-Frobenius theorem for nonnegative matrices \cite{smith2008monotone}. \par

In addition to its spectral properties,  the interwoven matrix (\ref{eqn: interwoven_Q}) also has the following exponent property.
\begin{lemma}\label{appendB:expo_lemma}
Let $\bm{P}$ be the interwoven matrix defined as defined in equation (\ref{eqn: interwoven_Q}).  Then for all $n \in \N$
\begin{equation} \label{eqn:induction_equation}
\bm{P}^{n} =  \sum_{i = 1}^{r} \bm{M}_{i}^{n} \otimes \bm{D}_{i}.
\end{equation} 
\end{lemma}
\begin{proof}
The result follows by induction.  Assume for some $k \in \N$ that equation\autoref{eqn:induction_equation} holds.  Consider the case for $k+1$,
\begin{align}
\lp  \sum_{i = 1}^{r}  \bm{M}_{i} \otimes \bm{D}_{i}   \rp^{k+1} &= \lp \sum_{i = 1}^{r}  \bm{M}_{i} \otimes \bm{D}_{i}  \rp \lp \sum_{i = 1}^{r}  \bm{M}_{i} \otimes \bm{D}_{i} \rp ^{k},  \nonumber \\
&=  \lp \sum_{i = 1}^{r}  \bm{M}_{i} \otimes \bm{D}_{i}  \rp  \lp \sum_{i = 1}^{r}  \bm{M}_{i}^{k} \otimes \bm{D}_{i} \rp,  \label{eqn:hyp_induction}
\end{align}
where the second equality follows from the inductive hypothesis.  Applying the multiplication property of the of Kronecker matrix\autoref{kronecker_mat_prop},  expansion of equation\autoref{eqn:hyp_induction} and the mixed-product property of the Kronecker product leads to the following cancellations,
\begin{align}
 \lp \sum_{i = 1}^{r}  \bm{M}_{i} \otimes \bm{D}_{i}  \rp  \lp \sum_{i = 1}^{r}  \bm{M}_{i}^{k} \otimes \bm{D}_{i} \rp &= \lp \bm{M}_{1} \otimes \bm{D}_{1} \rp\lp \sum_{i = 1}^{r}  \bm{M}_{i}^{k} \otimes \bm{D}_{i} \rp  + \cdots + \lp \bm{M}_{r} \otimes \bm{D}_{r} \rp\lp \sum_{i = 1}^{r}  \bm{M}_{i}^{k} \otimes \bm{D}_{i} \rp,  \nonumber \\
 &= \bm{M}_{1} \bm{M}_{1}^{k} \otimes \bm{D}_{1} + \cdots +  \bm{M}_{r} \bm{M}_{r}^{k} \otimes \bm{D}_{r} , \nonumber \\
  &=  \sum_{i = 1}^{r} \bm{M}_{i}^{k+1} \otimes \bm{D}_{i} . 
\end{align}
That is,  the inductive hypothesis is satisfied and thus equation (\ref{eqn:induction_equation}) holds for all $n \in \N$ by the principle of induction.

\end{proof}

\section{Computational methods}\label{sec:ODE_methods}
The ODE systems in this study were solved numerically using the \texttt{ODE15s} solver in Matlab (R2021a).  Simulations were performed over a total of 1000 time units in addition to an stopping event applied to the ODE solver to check for solution convergence.  Namely,  if all trajectories varied less than $1 \times 10^{-4}$ over four consecutive iterations,  then we assume that the system has converged to a steady state.  We note that all simulations presented in this study satisfied the convergence criteria.  The intracellular kinetics parameter values of the IO system (\ref{example_io_system_1}-\ref{example_io_system_2}) used in all simulations are given in Table \ref{tab:example_io_parameters} below. \par

\begin{table}[h]
\begin{tabular}{|l|c|c|c|c|c|c|c|c|c|c|}
\hline
\textbf{Parameter} & $a_{1}$ & $a_{2}$ & $b_{1}$ & $b_{2}$ & $b_{3}$ & $k_{1}$ & $k_{2}$ & $h_{1}$ & $h_{2}$ & $h_{3}$ \\ \hline
\textbf{Value}     & 0.01    & 1       & 100     & 100     & 100     & 2       & 2       & 2       & 2       & 1       \\ \hline
\end{tabular}
\vspace{1em}
\caption{Parameter values used in the simulations of the IO system (\ref{example_io_system_1}-\ref{example_io_system_2}). }\label{tab:example_io_parameters}
\end{table}

Random initial conditions were sampled from a uniform distribution using the \texttt{rand} function.  The homogeneous steady state of the system was calculated using the \texttt{fsolve} function that implements the trust-region-dogleg minimisation algorithm \cite{chapra2012applied}.  In addition,  heterotypic weighting parameters was set to $w_{2}^{[k]} = 1$ for all simulations.  Both quotient and large-scale ODE systems where solved using the same kinetics functions where respective adjacency matrices were introduced as an argument to these kinetics functions to ensure solution consistency.\par 

To visualise the approximate cell membranes in the large-scale simulation Voronoi diagrams were drawn around graph vertices using the \texttt{delaunayTriangulation} and \texttt{voronoi} functions within the Computational Geometry toolbox in Matlab (R2021a).  Ghost vertices were introduced to ensure that each graph vertex has a closed boundary.\par 

Eigenvalues of the adjacency matrices were calculated using the \texttt{eig} function from the Linear Algebra toolbox in Matlab (R2021a).  The edge structures of the semi-regular non-bipartite graphs used in the numerical spectral investigation are given in Table \ref{tab:non_bi_summary}.  These graphs were confirmed non-bipartite by violating the spectral symmetry property of bipartite graphs. \par
\begin{table}[h]
\begin{tabular}{|c|c|c|c|c|}
 \hline
&  $n_{1,\cL_{1}}$& $n_{2,\cL_{1}}$ &$n_{1,\cL_{2}}$  &  $n_{2,\cL_{2}}$ \\ \hline
  $\cG_{1}$ & 2 & 2 &2  &2  \\
  $\cG_{2}$ &  2&3  & 2 &3  \\
  $\cG_{3}$ &  2&  4&  2&4  \\
 $\cG_{4}$  & 4 & 3 & 4 & 3 \\
 \hline
\end{tabular}
\vspace{1em}
\caption{Summary of the bilayer edge connectivity structures for the graphs used in the numerical investigation of non-bipartite spectra in Figure \ref{fig:nonbipartite_example}.}\label{tab:non_bi_summary}
\end{table}

Source code for the simulations presented in this study can be found at \url{https://github.com/joshwillmoore1/Mixed_Signal_mechanisms}.

\section{Proof of $\pi_{2}$-dependent spectral gap for $\cG_{\text{2D}}$ and $\cG_{\text{3D}}$ from Figure \ref{fig:bipartite_example}B}\label{Appendix:spec_gap}

Bipartite graphs have many particularly convenient algebraic properties due to the existence of a simple canonical form of the respective adjacency matrix.  Namely,  for any bipartite graph $\cG_{k}$ with adjacency matrix $\bm{W}_{k}$ there exists a permutation matrix $\bm{U}$ that re-indexes the vertices with respect to the independent sets $V_{1}$ and $V_{2}$ such that
\begin{equation}
\bm{U}^{T} \bm{W}_{k} \bm{U}  =\begin{bmatrix}
\bm{0} & \bm{X}_{k} \\
\bm{X}_{k}^{T} & \bm{0} 
\end{bmatrix},
\end{equation}
where $\bm{X}_{k}$ is the biadjacency matrix \cite{godsil2001algebraicBook}.  Subsequently,  the spectra of the bipartite graphs have a distinct structure such that there is a symmetries of eigenvalues respective to the biadjacency matrices,  i.e.,  $\Sp \lp \bm{W}_{k} \rp = \Sp \lp \bm{X}_{k} \rp \cup \Sp \lp - \bm{X}_{k} \rp$.  
Leveraging the spectral symmetry of bipartite graphs and the spectral retention of equitable partitions,  we demonstrate that for the bipartite bilayer graphs $\cG_{\text{2D}}$ and $\cG_{\text{3D}}$ in Figure \ref{fig:bipartite_example},  the smallest eigenvalue of $\bm{X}_{k}$ is $-\overline{\lambda}_{k,2}$,  the polarity driven eigenvalue associated with laminar pattern template $\pi_{2}$.
 
\begin{lemma}\label{lemma: biadj_lower_bound}
Let $\cG_{k}$ be a regular bipartite bilayer graph with associated row-stochastic weighted adjacency matrix $\bm{W}_{k}$ (\ref{eqn:w_block_form}) for 2D or 3D tissues as shown in Figure \ref{fig:bipartite_example} ($k = \text{2D},\text{3D}$).  Consider the equitable partition $\pi_{2}$ such that the quotient graph,  $\cG_{k,\pi_{2}}$,  consists of only two representative vertices in each layer of $\cG_{k}$ and has the reduced adjacency matrix $\overline{\bm{W}}_{k}$ (\ref{eqn:reduce_adj_mat}).  Then biadjacency matrix $\bm{X}_{k}$ associated with $\bm{W}_{k}$ satisfies
\begin{equation}
-\overline{\lambda}_{k,2} = \min \lp \mu_{k,j} : \mu_{k,j} \in  \Sp \lp \bm{X}_{k} \rp \rp,
\end{equation}
where $\overline{\lambda}_{k,2}$ is the smallest eigenvalue of $\overline{\bm{W}}_{k}$ with associated eigenvector $\overline{\bm{v}}_{k,2}$.
\end{lemma}

\begin{proof} 
The proof for $\cG_{\text{2D}}$ is given as the argument follows identically for $\cG_{\text{3D}}$.  As we make use of the biadjacency form of $\bm{W}_{\text{2D}}$,  we first construct the biadjacency transformation $\bm{U}$.  The bipartite bilayer graph $\cG_{\text{2D}}$ has vertex indices in layer-wise order as defined in Section \ref{sec:model_definition} with block adjacency matrices given in Example \ref{example_io_system_2}.  To reorder the vertices of $\cG_{\text{2D}}$ such that vertex groups $V_{1}$ and $V_{2}$ are ordered consecutively,  we define the permutation matrix
\begin{equation}
\bm{U} = \begin{bmatrix}
\bm{I}_{|\cL_{1}|/2} \otimes \bm{D}_{1} & \bm{I}_{|\cL_{1}| /2} \otimes \bm{D}_{2}  \\
\bm{I}_{|\cL_{2}|/2} \otimes \bm{D}_{2} & \bm{I}_{|\cL_{2}|/2} \otimes \bm{D}_{1}  
\end{bmatrix},
\end{equation}
where $|\cL_{1}| = |\cL_{2}|$ as each layer has the same number of vertices.  In particular we have the biadjacency form
\begin{align}\label{eqn:biadj_form}
\bm{U}^{T} \bm{W}_{\text{2D}} \bm{U} &=  \begin{bmatrix}
\bm{I}_{|\cL_{1}|/2} \otimes \bm{D}_{1} & \bm{I}_{|\cL_{1}| /2} \otimes \bm{D}_{2}  \\
\bm{I}_{|\cL_{2}|/2} \otimes \bm{D}_{2} & \bm{I}_{|\cL_{2}|/2} \otimes \bm{D}_{1}  
\end{bmatrix}
\begin{bmatrix}
\widehat{\bm{W}}_{1,\cL_{1}} & \widehat{\bm{W}}_{2,\cL_{1}} \\
\widehat{\bm{W}}_{2,\cL_{1}}^{T} & \widehat{\bm{W}}_{1,\cL_{2}}
\end{bmatrix}
 \begin{bmatrix}
\bm{I}_{|\cL_{1}|/2} \otimes \bm{D}_{1} & \bm{I}_{|\cL_{1}| /2} \otimes \bm{D}_{2}  \\
\bm{I}_{|\cL_{2}|/2} \otimes \bm{D}_{2} & \bm{I}_{|\cL_{2}|/2} \otimes \bm{D}_{1}  
\end{bmatrix}, \nonumber \\ 
&= 
 \begin{bmatrix}
0 & \bm{X}_{\text{2D}}  \\
\bm{X}_{\text{2D}} & 0 
\end{bmatrix},
\end{align}
for $\bm{X}_{\text{2D}}$ in cyclic tridiagonal form
\begin{equation}\label{eqn:X2d}
\bm{X}_{\text{2D}} =  \begin{bmatrix}
 \hat{w}_{2}^{[\text{2D}]} & \hat{w}_{1}^{[\text{2D}]} & 0 & \cdots & 0 & \hat{w}_{1}^{[\text{2D}]} \\
\hat{w}_{1}^{[\text{2D}]} &  \hat{w}_{2}^{[\text{2D}]}  & \hat{w}_{1}^{[\text{2D}]} & & &\\
& & \ddots & & & \\
& & & \ddots & & \\
& &  &\hat{w}_{1}^{[\text{2D}]} & \hat{w}_{2}^{[\text{2D}]}  &\hat{w}_{1}^{[\text{2D}]} \\
\hat{w}_{1}^{[\text{2D}]}&0 & \cdots  &0 &\hat{w}_{1}^{[\text{2D}]} &   \hat{w}_{2}^{[\text{2D}]} 
\end{bmatrix},
\end{equation}
noting that $\bm{X}_{\text{2D}}^{T} =\bm{X}_{\text{2D}} $ by the regularity of $\cG_{\text{2D}}$ and therefore $\bm{U}^{T} \bm{W}_{\text{2D}} \bm{U}$ is symmetric.

 As the laminar pattern template partition $\pi_{2}$ is equitable there exists a lifting matrix $\bm{L} \in \{0,1\}^{N\times2}$ that maps the large-scale adjacency matrix $\bm{W}_{\text{2D}}$ into its reduced form $\overline{\bm{W}}_{\text{2D}}$ such that
\begin{equation}\label{eqn:lifting_relation}
\bm{W}_{\text{2D}} \bm{L} = \bm{L} \overline{\bm{W}}_{\text{2D}},
\end{equation} 
as demonstrated in \cite{Godsil1997}.  The lifting transformation is constructed by grouping vertices associated with the partition $\pi_{2}$ on $\bm{W}_{\text{2D}}$ for example $\bm{L}_{ij} = 1$ if $v_{i} \in \cL_{j}$.  Owing to the block structure of $\bm{W}_{\text{2D}}$ (\ref{eqn:w_block_form}) which follows from the layer-wise vertex indexing,  we have that 
\begin{equation}
\bm{L} = \begin{bmatrix}
\bm{1}_{| \cL_{1} |, 1} & \bm{0}_{| \cL_{1} |,1} \\
\bm{0}_{| \cL_{2} |,1} & \bm{1}_{| \cL_{2} |,1} \\
\end{bmatrix}.
\end{equation}
Critically,  the lifting matrix $\bm{L}$ provides the algebraic link between the quotient and large-scale graphs.  \par

Following from the regular structure of $\cG_{\text{2D}}$ and direct computation,  the eigenvector associated with $\overline{\lambda}_{\text{2D},2}$ has the form $\overline{\bm{v}}_{\text{2D},2} = [1,  -1]^{T}$.  The spectral retention property of the equitable partition, $\pi_{2}$,  guarantees that $\overline{\lambda}_{\text{2D},2} \in \Sp \lp \bm{W}_{\text{2D}} \rp$ where $\bm{L}\overline{\bm{v}}_{\text{2D},2}$ is the corresponding eigenvector for the large-scale graph $\cG_{\text{2D}}$ (by Theorem 9.3.3 in \cite{godsil2001algebraicBook}).  Explicitly,  we have that the lifted eigenvector is of the form 
\begin{equation}
\bm{L}\overline{\bm{v}}_{\text{2D},2} = \begin{bmatrix}
\bm{1}_{| \cL_{1} |, 1} \\
-\bm{1}_{| \cL_{2} |, 1}
\end{bmatrix},
\end{equation}
with associated eigenvalue $\overline{\lambda}_{\text{2D},2}$.  In the biadjacency matrix form (\ref{eqn:biadj_form}),  the corresponding eigenvector has the transformed representation
\begin{equation}
\bm{\nu} \coloneqq \bm{U}^{T} \bm{L}\overline{\bm{v}}_{\text{2D},2}  = \begin{bmatrix}
 \smash[b]{\blockTwo{|\cL_{1}|}} &  \smash[b]{\blockTwo{|\cL_{2}|}}
\end{bmatrix}^{T}.
\vspace{0.9em}
\end{equation}
The spectral symmetry of bipartite graphs ensures that if there exists an eigenpair  $ \lp \overline{\lambda}_{\text{2D},2} , \bm{\nu}  = [\bm{x},\bm{y}^{T}]\rp$ then there must also exist the eigenpair $ \lp -\overline{\lambda}_{\text{2D},2}, \tilde{ \bm{\nu} } = [\bm{x},-\bm{y}^{T}]\rp$  \cite{godsil2001algebraicBook}. Therefore the eigenvector associated with $- \lambda_{\text{2D},2}$ has biadjacency form
\begin{equation}
\tilde{\bm{\nu}}  = \begin{bmatrix}
 \smash[b]{\blockTwo{|\cL_{1}|}} &  \smash[b]{\blockOne{|\cL_{2}|}}
\end{bmatrix}^{T}.
\vspace{1.2em}
\end{equation}
which negates the signs of those entries associated with the latter half of the vertices in $\cG_{\text{2D}}$.  Subsequently,  the first $|\cL_{1}|$ entries of $\tilde{\bm{\nu}}$ are an eigenvector of $\bm{X}_{\text{2D}}$ with eigenvalue $- \overline{\lambda}_{\text{2D},2}$ following from the canonical bidjacency representation of $\bm{W}_{\text{2D}}$ (\ref{eqn:biadj_form}).  We denote this reduced eigenvector in normalised form
\begin{equation}
\tilde{\bm{\nu}}_{1}  = \frac{1} { \sqrt{|\cL_{1}|}} \begin{bmatrix}
 \smash[b]{\blockTwo{|\cL_{1}|}}
\end{bmatrix}^{T},
\vspace{1.2em}
\end{equation}
and therefore it remains to show that the eigenpair $\lp - \overline{\lambda}_{\text{2D},2},   \tilde{\bm{\nu}}_{1}  \rp$ is minimal in the spectrum of $\bm{X}_{\text{2D}}$.  \par
The Rayleigh quotient for $\bm{X}_{\text{2D}}$ is defined by
\begin{equation}
R_{\bm{X}_{\text{2D}}} \lp \bm{y} \rp  = \frac{\bm{y}^{T} \bm{X}_{\text{2D}} \bm{y}}{\bm{y}^{T} \bm{y}}
\end{equation}
and as $\bm{X}_{\text{2D}}$ is real and symmetric by the Min-Max theorem the Rayleigh quotient is bounded by the maximal and minimal eigenvalues of the matrix,  $R_{\bm{X}_{\text{2D}}} \lp \bm{y} \rp \in [\lambda_{\min}, \lambda_{\max}]$ \cite{godsil2001algebraicBook}.  In particular,  $R_{\bm{X}_{\text{2D}}} \lp \bm{y} \rp $ generates the eigenvalues of $\bm{X}_{\text{2D}}$ when $\bm{y}$ is the respective eigenvector.  Hence we show that $\tilde{\bm{\nu}}_{1} $ minimises $R_{\bm{X}_{\text{2D}}} \lp \bm{y} \rp$,  namely 
\begin{equation}
\argmin _{ \substack {\bm{y} \in  \R^{|\cL_{1}|} \\
|| \bm{y} || = 1}}    \lp R_{\bm{X}_{\text{2D}}} \lp \bm{y} \rp  \rp =\tilde{\bm{\nu}}_{1} .
\end{equation}
where the normality constraint follows from $\bm{X}_{\text{2D}}$ being real and symmetric and so the eigenvectors of $\bm{X}_{\text{2D}}$ are orthonormal with real eigenvalues.\par

The normalised form of $\tilde{\bm{\nu}}_{1}$ yields $\tilde{\bm{\nu}}_{1}^{T}\tilde{\bm{\nu}}_{1}  = 1$ and therefore the Rayleigh quotient evaluated at $\tilde{\bm{\nu}}_{1}$ simplifies to $R_{\bm{X}_{\text{2D}}} \lp \tilde{\bm{\nu}}_{1} \rp = \tilde{\bm{\nu}}_{1}^{T}  \bm{W}_{\text{2D}}  \tilde{\bm{\nu}}_{1}$.  By direct computation we have that 
\begin{align}
\tilde{\bm{\nu}}_{1}^{T}  \bm{X}_{\text{2D}}  \tilde{\bm{\nu}}_{1} &= \sum_{k = 1}^{|\cL_{1}|} \sum_{j=1}^{|\cL_{1}|} \lp  \bm{X}_{\text{2D}}  \rp_{k,j}    \lp  \tilde{\bm{\nu}}_{1} \rp_{k}   \lp \tilde{\bm{\nu}}_{1} \rp_{j} \nonumber \\ 
&= \sum_{i = 1}^{|\cL_{1}|} \lp   \bm{X}_{\text{2D}} \rp_{ii}  \lp \tilde{\bm{\nu}}_{1} \rp_{i}^{2} +  \sum_{i = 2}^{|\cL_{1}|-1}  \lp \tilde{\bm{\nu}}_{1} \rp_{i}  \lp \lp \bm{X}_{\text{2D}} \rp_{i,i-1}  \lp \tilde{\bm{\nu}}_{1} \rp_{i-1} + \lp \bm{X}_{\text{2D}} \rp_{i,i+1}  \lp \tilde{\bm{\nu}}_{1} \rp_{i+1}  \rp \nonumber \\
&	\quad \quad \quad \quad \quad \quad \quad \quad \quad \,  +   \lp \tilde{\bm{\nu}}_{1} \rp_{1} \lp \lp \bm{X}_{\text{2D}} \rp_{1,2} \lp \tilde{\bm{\nu}}_{1} \rp_{2} + \lp \bm{X}_{\text{2D}} \rp_{1,|\cL_{1}|}  \lp \tilde{\bm{\nu}}_{1} \rp_{|\cL_{1}|} \rp \nonumber \\
&	\quad \quad \quad \quad \quad \quad \quad \quad \quad \,  +  \lp \tilde{\bm{\nu}}_{1} \rp_{|\cL_{1}|} \lp \lp \bm{X}_{\text{2D}} \rp_{|\cL_{1}|,1} \lp \tilde{\bm{\nu}}_{1} \rp_{1} + \lp \bm{X}_{\text{2D}} \rp_{|\cL_{1}|,|\cL_{1}|-1}  \lp \tilde{\bm{\nu}}_{1} \rp_{|\cL_{1}|-1} \rp
\end{align}
by the cyclic tridiagonal form of $\bm{X}_{\text{2D}} $ (\ref{eqn:X2d}).  Critically as $\lp \bm{X}_{\text{2D}} \rp_{i,j} \geq 0 $ for all $i,j$,  then $\tilde{\bm{\nu}}_{1}^{T}  \bm{X}_{\text{2D}}  \tilde{\bm{\nu}}_{1}$ is minimised when $\text{Sgn}\lp \lp\tilde{\bm{\nu}}_{1}\rp_{k} \rp \neq \text{Sgn}\lp \lp\tilde{\bm{\nu}}_{1}\rp_{k+1} \rp $ for all $k \in \{1,..., |\cL_{1}|-1 \}$ which is satisfied by definition of $\tilde{\bm{\nu}}_{1}$.  Furthermore,  the orthonormal property of the eigenbasis of $\bm{X}_{\text{2D}}$ ensures that no other eigenvector has this alternating sign structure which implies that $-\overline{\lambda}_{\text{2D},2}$ is the smallest eigenvalue of $\bm{X}_{\text{2D}}$.

\end{proof}

A consequence of Lemma \ref{lemma: biadj_lower_bound} is the existence of a spectral gap about the origin for $\cG_{\text{2D}}$ and $\cG_{\text{3D}}$.

\begin{theorem}
Let $\cG_{k}$ and $\cG_{k,\pi_{2}}$ be defined as in Lemma \ref{lemma: biadj_lower_bound} and let $ \lambda_{k,j} \in \Sp \lp  \bm{W}_{k} \rp$.  If $\overline{\lambda}_{k,2} < 0$ then $\lambda_{k,j}  \not\in  \lp \overline{\lambda}_{k,2} , -\overline{\lambda}_{k,2} \rp $.
\end{theorem}
\begin{proof}
From Lemma \ref{lemma: biadj_lower_bound} we have that $-\overline{\lambda}_{k,2} = \min \lp \mu_{k,j} : \mu_{k,j} \in  \Sp \lp \bm{X}_{k} \rp \rp$ and thus $-\overline{\lambda}_{k,2} > 0$.  From the symmetry of the spectrum of bipartite graphs we have that $\overline{\lambda}_{k,2} < 0$ is the maximum of the negative eigenvalues of $\bm{W}_{k}$ therefore defining a region about the origin bounded by $\overline{\lambda}_{k,2}$ and $-\overline{\lambda}_{k,2}$ that contains no eigenvalues.
\end{proof}

\end{document}